\documentclass[a4paper,11pt]{article}
\usepackage[utf8]{inputenc}
\usepackage[T1]{fontenc}
\usepackage{amsmath}
\usepackage{amsfonts}
\usepackage{amssymb}
\usepackage{amsthm}
\usepackage{mathrsfs}
\usepackage[english]{babel}

\usepackage{dsfont}
\usepackage{amssymb}
\usepackage{eqnarray}
\usepackage{yfonts}
\usepackage{enumerate} 

\usepackage{graphicx}
\usepackage{variations}
 \usepackage[all]{xy}
\usepackage{tikz-cd}
\usepackage{tikz}

\newtheorem{theo}{Theorem}[section]
\newtheorem{defi}[theo]{Definition}

\newtheorem{prop}[theo]{Proposition}

\newtheorem{lemm}[theo]{Lemma}
\newtheorem{fact}[theo]{Fact}
\newtheorem{coro}[theo]{Corollary}
\newtheorem{rema}[theo]{Remark}

\newcommand{\Z}{\mathbb{Z}}

\newcommand{\R}{\mathbb{R}}
\newcommand{\Q}{\mathbb{Q}}
\newcommand{\C}{\mathbb{C}}

\newcommand{\of}{\mathfrak{o}_F}
\newcommand{\fA}{\mathfrak{A}}
\newcommand{\fP}{\mathfrak{P}}
\newcommand{\fB}{\mathfrak{B}}
\newcommand{\fQ}{\mathfrak{Q}}
\newcommand{\fo}{\mathfrak{o}}
\newcommand{\fp}{\mathfrak{p}}
\newcommand{\fK}{\mathfrak{K}}

\newcommand{\Id}{\mathrm{Id}}
\newcommand{\infl}{\mathrm{inf}}

\newcommand{\mB}{\mathrm{B}}
\newcommand{\mA}{\mathrm{A}}
\newcommand{\mM}{\mathrm{M}}
\newcommand{\Lie}{\mathrm{Lie}}
\newcommand{\Tr}{\mathrm{Tr}}
\newcommand{\Aut}{\mathrm{Aut}}
\newcommand{\End}{\mathrm{End}}

\newcommand{\Hom}{\mathrm{Hom}}
\newcommand{\BT}{\mathrm{BT}}
\newcommand{\ord}{\mathrm{ord}}
\newcommand{\Res}{\mathrm{Res}}
\newcommand{\Gal}{\mathrm{Gal}}

\newcommand{\cind}{\mathrm{c-ind}}

\newcommand{\nr}{\mathrm{nr}}
\newcommand{\rad}{\mathrm{rad}}

\title{Comparison of Bushnell-Kutzko and Yu's constructions of supercuspidal representations}

\author{Arnaud Mayeux}

\begin{document}

\maketitle

\textbf{Abstract}: We compare Bushnell-Kutzko \cite[$BK$]{BK} and Yu's \cite[$YU$]{YU} constructions of supercuspidal representations. We first treat the direction $BK \longrightarrow YU$: in a tame situation, at each step of Bushnell-Kutzko's construction, we associate a part of a Yu datum; and at the end, the supercuspidal representations associated with the $BK$ datum and the obtained $YU$ datum are equal. We also give a similar result for the direction $YU \longrightarrow BK.$

\tableofcontents

\section*{Acknowledgments}
I thank Anne-Marie Aubert, Colin Bushnell, Guy Henniart, Jeffrey Adler and Paul Broussous for many help, comments and suggestions.

\section*{Introduction}

We compare Bushnell-Kutzko's construction  of irreducible supercuspidal representations of $GL_N (F)$ to Yu's construction of tamely ramified representations of a general connected reductive group scheme ($F$ is a non-Archimedean local field).
Given an $F$-vector space $V$ of dimension $N$, Bushnell-Kutzko introduced the notion of a simple stratum in $A=\End_F(V)$, which is a $4$-uple $[ \fA , n , r , \beta ]$ ($\fA$ is an $\of$-hereditary order in $A$, $n>r$ are integers and $\beta$ is an element in $A$) that satisfies several conditions: in particular, the $F$-subalgebra generated by $\beta$ and $F$ in $A$ is a field. Given a fixed simple stratum $[ \fA , n , 0 , \beta ]$, Bushnell-Kutzko introduced the notion of a defining sequence attached to it, which is a finite sequence $[\fA, n , r_i , \beta_i ] , 0 \leq i \leq s $ of simple strata satisfying several conditions. Using a defining sequence, Bushnell-Kutzko introduced a group $H^1(\beta, \fA)$ and a set of characters of this group, called the set of simple characters and denoted by $\mathcal{C}(\fA, 0 , \beta)$. To the stratum $[ \fA , n , 0 , \beta ]$, Bushnell-Kutzko attached (also using a defining sequence) several groups $J^1(\beta , \fA) \subset J^0 (\beta, \fA)$ containing $H^1(\beta , \fA)$. None of the previously defined objects depend on the choice of the defining sequence that is used to define them. Given a simple character $\theta \in \mathcal{C}(\fA, 0 , \beta)$, they defined the $\beta$-extensions of $\theta$, which are certain representations of $J^0(\beta, \fA)$ containing $\theta$. Let us assume now that the fixed simple
stratum satisfies a maximality condition. Let $\kappa$ be a $\beta$-extension of a simple character $\theta \in \mathcal{C}(\fA, 0 , \beta)$.  Bushnell-Kutzko show that the compactly induced representation $c-ind _{E^{\times} J^0(\beta , \fA)}^G (\Lambda)$ is irreducible and supercuspidal. Here, $\Lambda$ is an extension of $\kappa \otimes \sigma $ to $E^{\times }J^0 (\beta , \fA)$, where $\sigma$ is an irreducible cuspidal representation of the finite general linear group $J^0(\beta , \fA) / J^1(\beta , \fA)$. They show that all irreducible supercuspidal representations of $G$ arise in this way.

 Yu constructed tamely ramified irreducible supercuspidal representations of a general reductive group $G$. Given a $5$-uple $(\overrightarrow{G}, y , \overrightarrow{\boldsymbol{r}} , \rho , \overrightarrow{\boldsymbol{\Phi}})$ (roughly, $\overrightarrow{G}=G^0 \subset \ldots \subset G^d=G$ is a tamely ramified twisted Levi sequence, $y$ is a point in the Bruhat-Tits building of $G^0$, $\overrightarrow{\boldsymbol{r}}= \boldsymbol{r}_0 < \dots \leq \boldsymbol{r}_d$ are real numbers, $\rho$ is a depth zero representation related to $G^0$, and $ \overrightarrow{\boldsymbol{\Phi}}$ is a sequence of characters of the groups $G^0(F), \dots , G^d(F)$), Yu associated a group $K^d$ compact modulo the center of $G(F)$ and an irreducible supercuspidal representation $\rho _d$ such that the compactly induced representation to $G(F)$ is irreducible and supercuspidal. Fintzen showed that Yu's construction is exhaustive if $p$ does not divide the order of the Weyl group of $G$ \cite{fintzen}.

 In this article, at each step of Bushnell-Kutzko's construction, we associate a part of a Yu datum, and at the end we obtain a full Yu datum. The supercuspidal representations associated to both data are the same. Let us explain the structure of our comparison. We use Bushnell-Kutzko's and Yu's original notation and terminology (this will be recalled in the following after the introduction) as in \cite{BK} and \cite{YU}.
Let $[\fA, n , 0 , \beta ]$ be a fixed maximal simple stratum, such that the attached field $E=F[\beta]$ is a tamely ramified extension of $F$.

 \begin{theo} \begin{enumerate}
 \item (Bushnell-Henniart) There exists an element $\gamma \in F[\beta] $ such that \[ [\fA,n, -k_0 (\beta , \fA) , \gamma ] \sim [\fA,n, -k_0 (\beta , \fA) , \beta ], \] here $k_0 (\beta , \fA)$ is Bushnell-Kutzko's critical exponent and $\sim$ denotes equivalence of strata.
 \item (M) For all $\gamma$ as in 1, the stratum $[\fB _{\gamma}, r , r-1 , \beta - \gamma ]$ is simple.

 \end{enumerate}
 \end{theo} This theorem is related to defining sequences and implies the following corollary:

 \begin{coro} \label{abcd} There exists a defining sequence $[ \fA , n , r_i , \beta _i ], 0 \leq i \leq s $ such that for $ 0 \leq i \leq s-1$

 \begin{enumerate}
 \item $F[\beta _{i+1}] \underset{\ne }{\subset} F[\beta _i]$ 
 \item $[\fB _{\beta _{i+1}}, r_{i+1} , r_{i+1} -1 , \beta _i - \beta _{i+1} ]$ is a simple stratum.
 \end{enumerate}

 \end{coro} Let us fix  a defining sequence as in Corollary \ref{abcd}. We assume in this introduction\footnote{This will be our $(Case ~B)$. The other case can be treated in the same way, up to minor details.} that $\beta _s \not \in F$. We put \begin{align*}
&d=s+1 \\
 &E_i=F[\beta _i] ~0 \leq i \leq s , E_d =F \\
 &B_i  = \End _{E_i} (V) ~ 0 \leq i \leq d \\
 &\fB _{\beta _i}  = \fA \cap B_i ~ 0 \leq i \leq s \\
 &U^0(\fB _{\beta _i }) = \fB _{\beta _i } ^{\times} ,  U^k(\fB _{\beta _i }) =1+( \rad (\fB _{\beta _i }))^k ~ k \geq 1 \\
 &c_i = \beta _i - \beta _{i+1} ~~ 0 \leq i \leq s-1, ~ c_s = \beta _s .\\
 \end{align*}

We can now explicitly describe the group $H^1(\beta , \fA)$.

\begin{fact}
$H^1(\beta , \fA)=U^1 (\fB _{\beta _0}) U^{[\frac{-\nu _{\fA} (c_0)}{2}]+1}(\fB _{\beta _1} )\ldots U^{[\frac{-\nu _{\fA} (c_{s})}{2}]+1}(\fA)$
\end{fact}

We now fix a simple character $\theta \in  \mathcal{C}(\fA, 0 , \beta)$. The fact that the sequence of fields $E_0, E_1 , \dots , E_d$ is decreasing implies that we can factorise the character $\theta$ as follows: 

\begin{theo} There exist $\phi _0 , \ldots , \phi _s $, smooth characters of $E_0 ^{\times } , \ldots , E_s ^{\times}$ such that \[ \theta = \prod _{i=0} ^s \theta ^i \] where $\theta ^i $ is determined by  \begin{align*}
 \bullet~&\theta ^i \mid _{H^1 ( \beta , \fA ) \cap B_i } = \phi \circ {\det} _{B_i} \\
 \bullet~&\theta ^i \mid _{H^1 (\beta , \fA) \cap U^{[-\frac{\nu _{\fA} (c_i)}{2}]+1}(\fA)} = \psi _{c_i} ~~~~ \text{where} ~ \psi _{c_i} (x) = \psi \circ Tr_{A/F} (c_i (x-1)).\\
\end{align*}

\end{theo}

This allows us to introduce parts of a Yu datum.

\begin{defi} We put \begin{align*}
&G^i = \Res _{E_i /F} \underline{\Aut}_{E_i}(V) ~0 \leq i \leq d,~ \overrightarrow{G}= G^0 , \ldots , G^d \\
&\boldsymbol{r}_i=-\ord (c_i) ~~0 \leq i \leq s, ~\boldsymbol{r}_d=\boldsymbol{r}_s, \overrightarrow{\boldsymbol{r}}= \boldsymbol{r}_0, \ldots , \boldsymbol{r}_d\\
& \boldsymbol{\Phi } _i = \phi _i \circ {\det}_{B_i} ~ 0 \leq i \leq s, \boldsymbol{\Phi } _d =1, ~ \overrightarrow{\boldsymbol{\Phi } }= \boldsymbol{\Phi } _0 , \ldots , \boldsymbol{\Phi } _d .\\
\end{align*}

\end{defi}

Let us now fix a $\beta$-extension $\kappa$ of $\theta$ and an irreducible cuspidal representation $\sigma$ of $J^0(\beta , \fA) / J^1 ( \beta , \fA)$. Let us also fix an extension $\Lambda$ of $\kappa \otimes \sigma $ to $E^{\times} J^0 (\beta , \fA)$. In this article we prove the following theorem, which is the end of the direction $BK \longrightarrow YU$.

\begin{theo} There exist $\rho $ and $y$ such that $(\overrightarrow{G}, y , \overrightarrow{\boldsymbol{r}}, \rho , \overrightarrow{\boldsymbol{\Phi}})$ is a Yu datum and such that \[\Lambda \simeq  \rho _d (\overrightarrow{G}, y , \overrightarrow{\boldsymbol{r}}, \rho , \overrightarrow{\boldsymbol{\Phi}}). \] In particular, the supercuspidal representations obtained by compact induction are equal.
\end{theo}

We give a similar Theorem for the direction $YU \longrightarrow BK $ in Section \ref{YUBK}, and we obtain the following corollary.

\begin{coro} The list of tame maximal extended simple types constructed by Bushnell-Kutzko \cite{BK} is equal to the list of $GL_N$ extended Yu types obtained through Yu's construction of supercuspidal representations \cite{YU}. The same holds for the so-called simple characters. See Corollary \ref{cororor} for a formal statement.
\end{coro}
Let us propose a list of steps for a first or rapid lecture:

\begin{enumerate}
\item Skip the end of the introduction after this list if you are not interested in the historical aspects.
\item If familiar with the basic idea of constructions of supercuspidal representations by compact induction, then  skip Section \ref{intcu}, otherwise read it briefly. 
\item If very familiar with BK, then skip Section \ref{sectionconstbk}. If not familiar with BK, then read Section \ref{sectionconstbk} and remember well the notion of defining sequence. 
\item If familiar with YU, then skip Section \ref{yu}, otherwise read it.
\item Skip Sections \ref{sectamesimple}, \ref{sectemrs} and \ref{geneassomini}. These are technical Sections for proofs.
\item Read Section \ref{tameca} but read only Subsection \ref{abstfact}; that is, skip Subsection  \ref{explifact}.
\item Read Section \ref{sectgecha} very well. Do not care about the disjunction $(Case ~A)$ and $(Case ~B)$ at first, read only $(Case ~A)$ for example. $(Case ~A)$ and $(Case ~B)$ only differ by one character.
\item Read Section \ref{extensionn}, the direction $BK \longrightarrow YU$ ends here.
\item Read Section \ref{YUBK}, this is $YU \longrightarrow BK$.
\item Read \ref{UNIF}.
\end{enumerate}

Heuristically, to find our comparison, we have compared Bushnell-Kutzko's construction with Moy's presentation \cite{MO} of Howe's construction \cite{HO} and Yu's construction with Moy's presentation of Howe's construction.
Then, we found statements for the direction $BK \longrightarrow YU $. We then proved the comparison for the direction $BK \longrightarrow YU $ and  worked on the reverse direction $YU \longrightarrow BK$. The following are the main hints for $BK \longleftrightarrow YU$ that are contained in the previous literature:

- Howe's construction \cite{HO} 

- Moy's presentation of Howe's construction\cite{MO}

- Hakim-Murnaghan's article \cite{Hamu} 

- Bushnell-Henniart 's series of articles  \cite{BHTL1} , \cite{bhtl3} , \cite{bhtl4} ,  \cite{Esse} ,\cite{bhess1}, \cite{bhess2} , \cite{bhess3}.

Almost all of the other works related to Bushnell-Kutzko 's and Yu's constructions should also be considered as a benefit, although we do not cite all of the literature here. 
We do not need to refer to these works in proofs of results of this article, except for Bushnell-Henniart \cite[3.1 Corollary]{Esse} for Proposition \ref{tssprop2}. We also use \cite{Brou} to compare the  filtrations. \\

\section*{Notations and conventions }

\begin{align*}
F &=\text{a fixed non-Archimedean local field}\\
\mathfrak{o}_F &=\text{ring of integer of } F \\
\mathfrak{p}_F &=\text{maximal ideal of  }\mathfrak{o}_F\\ 
\psi & = \text{a fixed additive character } (F,+) \to \C ^* \text{ of conductor }  \mathfrak{p}_F \\
k_F &= \text{residual field of } F  \\
\pi _F & = \text{ a fixed uniformiser of } F\\
 e(E \mid F) & = \text{ramification index of a finite extension } E/F \\
 \pi _E & =\text{a uniformiser of an extension } E \text{ of } F \\ \nu _E & = \text{unique valuation on a finite } \\ &~~~~\text{ extension $E/F$ such that }\nu _E (\pi _E ) =1 \\
\ord &= \text{unique valuation on algebraic }\\ 
& ~~~~\text{extensions of } F \text{ such that } \ord (\pi _F) =1 \\
\end{align*}
If $k$ is a field and if $G$ is a $k$-group scheme, then we denote by $\underline{\Lie}(G)$ the Lie algebra functor and $\Lie (G)$ the usual Lie algebra $\underline{\Lie}(G)(k)$. The Lie algebra functor, of a $k$-group scheme denoted with a big capital letter $G$, is also denoted by the same small Gothic letter $\mathfrak{g}$. 
 If $G$ is a connected reductive group defined over $F$, then we denote by $\BT^E(G,F)$ and $\BT ^R (G,F)$ the enlarged and reduced Bruhat-Tits buildings
 of $G$ over $F$ \cite{brti1}, \cite{brti2}. In this situation, if $y$ is a point of $\BT^E (G , F)$, then   $[y]$ denotes the image of $y$ via the canonical projection ${\BT^E(G,F) \to \BT ^R(G,F)}$.
  The group $G(F)$ acts on ${\BT^E(G,F)} $ and ${ \BT ^R(G,F)}$. Here,  $G(F)_{y}$ and $G(F)_{[y]}$ denote the stabilisers in $G(F)$ of $y$ and $[y]$. If $G$ splits over a tamely ramified extension, we consider the so called Moy-Prasad filtration, which was defined by Moy and Prasad \cite{Mopr} \cite{Mopr2}. This is the filtration used by Yu \cite{YU}. We use Yu's notations. So for each real number $r \geq 0$ and each $y$ in $\BT^E(G,F)$, we have some groups $G(F)_{y,r}$ and $G(F)_{y,r+}$. As in \cite{Mopr} and \cite{YU}, we also have a filtration of the Lie algebra $\Lie (G)=\mathfrak{g}(F)$ and of the dual of the Lie algebra $\mathfrak{g}^*(F)$.  Let us just recall here that \begin{center} ${\mathfrak{g}^*(F) _{y,-r}= \{ X \in \mathfrak{g}^*(F) \mid X(Y) \in \mathfrak{p}_F~ \text{for all}~ Y \in \mathfrak{g}(F) _{y, r+}\}}$,\end{center} and \begin{center}
  ${\mathfrak{g}^*(F) _{y,(-r)+}= \displaystyle \bigcup_{s <r}  \mathfrak{g}^*(F)_{y , -s}}$.\end{center}
  If $s<r$, then $G(F) _{y,s:r}$ denotes the quotient $G(F) _{y,s} / G(F) _{y,r}$.
  If $G$ is a torus, then we can avoid the symbol $y$ and we write, for example,  $G(F)_{r}$ and $\Lie^*(G) _{-r}$.
  If $H\subset G $ are groups and $\rho$ is a representation of $H$, then  $I_G(\rho)$ denotes the intertwining of $\rho $ in $G$; that is, the set 

  \begin{center}
  $I_G (\rho)= \{ g \in G \mid \Hom _{{}^g H \cap H} (^g \rho , \rho ) \not = 0 \}$
  \end{center}

\section{Intertwining, compact induction and supercuspidal representations} \label{intcu}

  Let $G$ be a connected reductive group defined over $F$ and let $P=MN$ be a parabolic subgroup of $G$. As usual in the literature, the notation $P=MN$ means that $M $ is a Levi subgroup of $P$ and $N$ is the unipotent radical of $P$. Let $r_P^G$ denote the normalised parabolic restriction functor from the category $\mathcal{M}(G)$ of smooth representations of $G(F)$ to the category $\mathcal{M}(M)$of smooth representations of $M(F)$.
    Let us recall the definition of a supercuspidal representation.
  \begin{defi} A representation $\pi \in \mathcal{M}(G)$ is supercuspidal if $r_P^G (\pi)=0 $ for all proper parabolic subgroups $P$ of $G$.
  \end{defi}

  The following lemma is an important characterisation of supercuspidal representations.

  \begin{lemm} \cite{Rena} \label{coef} A representation $\pi \in \mathcal{M}(G)$ is supercuspidal if and only if its matrix coefficients are compactly supported modulo the center of $G(F)$.
  \end{lemm}

  If $K$ is an open subgroup of $G(F)$, then the symbol $\cind_K^G$ denotes the compact induction functor. Lemma  \ref{coef} allows to prove the following proposition:

 \begin{prop} \cite{Cara} \label{Cara} Let $K$ be an open subgroup of $G(F)$, which is compact modulo the center of $G(F)$. Let $\rho$ be a smooth irreducible representation of $K$ and let $\pi= \cind_K^G(\rho)$ be the compactly induced representation of $\rho$ on $G(F)$. The following assertions are equivalent: \begin{enumerate}

 \item[(i)] The intertwining $I_G(\rho)$ of $\rho$ is reduced to $K$.

 \item[(ii)] The representation $\pi$ is irreducible and supercuspidal.

 \end{enumerate}
 \end{prop}

 This observation (Proposition \ref{Cara}) is absolutely fundamental and both constructions of  supercuspidal representations studied  in this paper are based on this fact.

\section{Bushnell-Kutzko's construction of supercuspidal representations for $GL_N$} \label{sectionconstbk}

For each irreducible supercuspidal representation $\pi$ of $GL_N(F)$, Bushnell and Kutzko
\cite{BK}
have constructed an open subgroup $K$, compact modulo the center of $GL_N(F)$, and a smooth irreducible representation $\Lambda$ of $K$ such that $\pi=\cind _K^{GL_N(F)} (\Lambda)$. In the following, we describe the construction of Bushnell and Kutzko, as in their book \cite{BK}. 

\subsection{Simple strata}\label{simplestrata}

Let $V$ be an $F$-vector space of dimension $N$. Let $A$ be the algebra $ \End_F(V)$. If $\mathfrak{A}$ is a hereditary $\of$-order in $A$, then  $\fP$ denotes its Jacobson radical and $\nu _{\fA} $ denotes the valuation on $\fA$ given by $\nu _{\fA} (x) = \max \{k \in \Z \mid x \in \fP ^k\}. $ A stratum in $A$ is a quadruple $[\fA , n,r , \beta]$, where $\fA$ is a hereditary $\of$-order, $n>r$ are integers and $\beta$ is an element in $A$ such that $\nu _{\fA} (\beta) \geq -n$. Let $e(\fA \mid \of)$ denote the period of an $\mathfrak{o}_F$-lattice chain associated to $\fA$. Let $\mathfrak{K}(\fA)$ be the normaliser of $\fA$ in $\mathrm{G}=A^{\times}$.
Before giving the definition of a pure stratum, let us prove an elementary lemma that will be used often in other sections of this paper.

\begin{lemm} \label{valval} Let $\fA$ be an hereditary $\of$-order in $A$, and let $E$ be a field in $A$ such that $E^{\times } \subset \mathfrak{K} (\fA )$. Let $\beta$ be an element in $E$, then 

\begin{equation}
\nu _{\fA} (\beta) e ( E \mid F ) = e (\fA \mid \of ) \nu _{E} (\beta) . \label{eqval}
\end{equation}
\end{lemm}
\begin{proof}Let $\pi_E$ denote a uniformiser  element in $E$. Since $E^{\times} \subset \mathfrak{K}(\fA)$, the elements $\pi_E , \pi _F $ and $\beta$ are in $\mathfrak{K}(\fA)$. Thus, the equality \cite[1.1.3]{BK} is valid for these elements and we  use it in the following equalities.
 On the one hand \begin{equation}\beta ^{e(E\mid F)}\fA = \pi _E ^{\nu _ E (\beta ) e(E\mid F)} \fA= \pi _F ^{\nu _E (\beta)} \fA . \label{val1}\end{equation}
On the other hand \begin{equation}\beta ^{e(E\mid F)}\fA = \fP ^{\nu _{\fA} (\beta) e (E \mid F ) } .\label{val2}\end{equation}
Moreover by definition of $e(\fA \mid \mathfrak{o}_F)$ (see \cite[1.1.2]{BK}), we have \begin{equation}\pi _F ^{\nu _E (\beta)} \fA = \fP ^{e (\fA  \mid \of )\nu _E (\beta )}. \label{val3}\end{equation}
The equalities \ref{val1} , \ref{val2} and \ref{val3} show that \begin{equation}\fP ^{\nu _{\fA} (\beta) e (E \mid F )} =\fP ^{e (\fA  \mid \of )\nu _E (\beta )}.\end{equation}
Consequently, $\nu _{\fA} (\beta) e (E \mid F )= e (\fA  \mid \of )\nu _E (\beta ) $ and equality \ref{eqval} holds as required.

\end{proof}

\begin{defi} \label{defpure} \cite[1.5.5]{BK} A stratum is pure if the following conditions hold.
\begin{enumerate}
\item[(i)] The $F$-algebra $E=F[\beta]$, generated by $F$ and $\beta$ in $A$, is a field.
\item[(ii)]  $E^{\times}$ is included in $\mathfrak{K}(\fA)$.
\item[(iii)]  The equality $\nu _{\fA} (\beta) =- n$ holds.
\end{enumerate}

\end{defi}Let $[\fA , n , r , \beta ]$ be a pure stratum, for each $k \in \Z$ let $\mathfrak{N}_k(\beta, \fA)$ be the set \cite[1.4.3]{BK}
\begin{center}
$\mathfrak{N}_k(\beta, \fA):=\{x \in \fA \mid \beta x - x \beta \in \fP ^k \}$.
\end{center}
Put $B=\End_{F[\beta]}(V)$ and $\mathfrak{B}=B\cap \mathfrak{A}$. We can define the following \textit{critical exponent} $k_0 (\beta , \fA )$ \cite[1.4.5]{BK}:

\begin{center}
$k_0(\beta,\fA):=\left\{ \begin{array}{ll}
          -\infty ~ \text{if}~ E=F  \\

        \max \{ k \in \Z \mid \mathfrak{N} _k (\beta , \fA)  \not \subset \mathfrak{B} + \fP \}~ \text{if}~ E\not =F.
    \end{array}  
\right.$
\end{center}

\begin{defi} \label{defsimple}\cite[1.5.5]{BK} A stratum $[\fA,n,r,\beta] $ is simple if it is pure and $r < - k_0 (\beta , \fA)$.
\end{defi}

The simple strata are constructed inductively from minimal elements, through a process that is the object of Section 2.2 of Bushnell-Kutzko's book \cite[2.2]{BK}. The following is the definition of a minimal element giving birth to a stratum with just one iteration:

\begin{defi}\label{defminimal}\cite[1.4.14]{BK} Let $E/F$ be a finite extension.
 An element $\beta \in E$ is minimal relatively to $E/F$ if the following three conditions are satisfied.
 \begin{enumerate}
\item[(i)] The field $F[\beta] $ is equal to the field $E$.

\item[(ii)]The integer $\gcd(\nu _E (\beta), e(E\mid F))$ is equal to $1$.

\item[(iii)] The element $\pi ^{-\nu _{E} (\beta)}_F \beta ^{e(E\mid F)} + \mathfrak{p} _E$ generates the residual field $k _E$ over $k_F$.

\end{enumerate}

An element $\beta$ in $\overline{F} $ is minimal over $F$ if it is minimal relatively to the extension $F[\beta] /F$.

\end{defi}
\begin{prop} \label{minisimplealfalfa} Let $[\fA,n,n-1, \beta]$ be a pure stratum in the algebra $\End_F (V)$. The following assertions are equivalent.

\begin{enumerate}
\item[(i)] The element $\beta$ is minimal over $F$.
\item[(ii)] The critical exponent $k_0 (\beta , \mathfrak{A})$ is equal to $-n$ or is equal to $- \infty $.

\item[(iii)] The stratum $[\fA , n , n-1 , \beta ]$ is simple.

\end{enumerate}\end{prop}\begin{proof} This is a direct consequence of \cite[1.4.15]{BK}. Indeed, assume that $\beta \in F$, then $\beta$ is clearly minimal over $F$, moreover $k_0 (\beta , \mathfrak{A}) = -\infty$ by definition, and thus $k_0 (\beta , \mathfrak{A}) < - (n-1)$, so the stratum $[\mathfrak{A},n,n-1,\beta]$ is simple. The three properties, being always satisfied in this case, are equivalent.
Assume now that $\beta \not \in F$, by \cite[1.4.15]{BK} $(i)$ and $(ii)$ are equivalent, moreover it is clear that $(ii)$ implies $(iii)$.  If $(iii)$ is true, then $k_0(\beta , \fA) < - (n-1)$ by definition of a simple stratum, moreover \cite[1.4.15]{BK} shows that $-n \leq k_0 (\beta , \mathfrak{A})$. So $k_0(\beta, \fA) =-n$ and the assertion $(ii)$ holds.

\end{proof}

For the rest of the paper, we need to define the notion of a tame corestriction \cite[1.3]{BK}. Let $E/F$ be a finite extension of $F$ contained in $A$. Let $B$ denote $\End_E(V)$, the centraliser of $E$ in $A$.

\begin{defi}\label{tamecorestridef}\cite[1.3.3]{BK} A tame corestriction on $A$ relatively to $E/F$ is a $(B,B)$-bimodule homomorphism $s:A \to B$ such that  $s( \mathfrak{A} ) = \mathfrak{A} \cap B $ for every hereditary $\of$-order $\fA$ normalised by $E ^{\times}$.

\end{defi}

The following proposition shows that such maps exist:

\begin{prop} \label{tamecoresprop} \cite[1.3.4, 1.3.8 (ii)]{BK} With the same notations as before, the following holds.

\begin{enumerate}
\item[(i)] Let $\psi _E $, $\psi _F$ be complex, smooth, additive characters of $E,F$ with conductor $\mathfrak{p}_E$, $\mathfrak{p}_F$ respectively. Let $\psi _B$ and $\psi _A$ be the additive characters defined by $\psi _B = \psi _E \circ \Tr _{B/E}$ and $\psi _A = \psi _F \circ \Tr _{A/F}. $ Then there  exists a unique map $s:A \to B$ such that $\psi _A (ab) = \psi _B (s(a)b)$, $a\in A$, $b \in B$. The map $s$ is a tame corestriction on $A$ relatively to $E/F$.

\item[(ii)] If the field extension $E/F$ is tamely ramified, then there exists a tame corestriction $s$ such that $s \mid _{B} = \Id _{B}$.

\end{enumerate}

\end{prop}

\subsection{Simple characters}\label{simplecharacters}

Let $[\fA,n,r,\beta]$ be a simple stratum. Bushnell and Kutzko attached to it a group $H^1(\beta,\fA)$ and a set of characters $\mathcal{C}(\beta,0,\fA)$ of $H^1(\beta,\fA)$ whose intertwining in $\mathrm{G}$ is remarkable. This is the object of this section.

  \begin{defi} \label{stratequiv}Two strata $[\fA,n,r,\beta _1 ] $ and $[\fA,n,r,\beta _ 2]$ are equivalent if $\beta_1 -\beta _2 \in \fP ^{-r}$. The notation $[\fA,n,r,\beta _1 ] \sim [\fA,n,r,\beta _ 2]$ means that  $[\fA,n,r,\beta _1 ] $ and $ [\fA,n,r,\beta _ 2]$ are equivalent.
  \end{defi}
The following theorem is fundamental for the construction of the group $H^1(\beta , \fA)$.

\begin{theo} \cite[2.4.1]{BK} \label{approxi} \begin{enumerate}
\item[($i$)] Let $[\fA,n,r,\beta]$ be a pure stratum in the algebra $A$. There exists a simple stratum $[\fA,n,r,\gamma]$ in $A$ equivalent to $[\fA,n,r,\beta]$, i.e. such that

\begin{center}

$[\fA,n,r,\gamma]\sim [\fA,n,r,\beta] .$

\end{center}

Moreover, for any simple stratum  $[\fA,n,r,\gamma]$  satisfying this condition, $e(F[\gamma]\mid F)$ divides $e(F[\beta]\mid F) $ and $f(F[\gamma]\mid F)$ divides $f(F[\beta] \mid F)$. Moreover, among all pure strata $[\mathfrak{A},n,r,\beta']$ equivalent to the given pure stratum $[\mathfrak{A},n,r,\beta]$, the simple ones are precisely those for which the field extension $F[\beta '] /F$ has minimal degree.

\item[($ii$)] Let $[\fA,n,r,\beta]$ be a pure stratum in $A$ with $r=-k_0(\beta,\fA)$. Let  $[\fA,n,r,\gamma]$ be a simple stratum in $A$ which is equivalent to $[\fA,n,r,\beta]$, let $s_{\gamma}$ be a tame corestriction on $A$ relative to $F[\gamma]/F$, let $B_{\gamma}$ be the $A$-centraliser of $\gamma$, i.e $B_{\gamma}=\End_{F[\gamma]}(V),$ and $\fB _{\gamma} = \fA \cap B_{\gamma}.$ Then $[\fB_{\gamma},r,r-1,s_{\gamma}(\beta - \gamma)]$ is equivalent to a simple stratum in $B_{\gamma}$.

\end{enumerate}
\end{theo}

\begin{rema} \label{degstriapprox}Let $[\mathfrak{A},n,r,\beta]$ be a pure stratum that is not simple and let $[\mathfrak{A},n,r,\gamma]$ be a simple stratum equivalent to $[\mathfrak{A},n,r,\beta]$, by \ref{approxi} $(i)$ the degree $[F[\beta]:F]$ is strictly bigger than the degree $[F[\gamma]:F]$.

\end{rema}

\begin{coro} \label{suiteap} \cite[2.4.2]{BK}  Given a pure stratum $[\fA,n,r,\beta]$, the previous theorem and remark allow us to associate an integer $s$ and a family $\{[\fA,n,r_i,\beta _i ]$, $0\leq i \leq s\}$  such that 
\begin{enumerate}
\item[($i$)]$[\fA , n , r_i , \beta _i ]$  is a simple stratum for $0\leq i \leq s$,
\item[($ii$)]$[\fA,n,r_0,\beta _0 ] \sim [\fA , n , r , \beta ]$,
\item[($iii$)] {\footnotesize $r=r_0 <r_1<\ldots <r_s <n $ and $[F[\beta_0]:F]> [F[\beta _1] :F]> \ldots > [F[\beta _s] :F],$}
\item[($iv$)]$r_{i+1}=-k_0(\beta _ i , \fA),$ and $ [\fA, n , r_{i+1} , \beta _{i+1} ]$ is equivalent to $[\fA,n,r_{i+1} , \beta _i] $ for $0 \leq i \leq s-1$,
\item[($v$)]$k_0(\beta _ s , \fA) =-n $ or $- \infty $,
\item[($vi$)] Let $\fB _{\beta_i} $ be the centraliser of $\beta _ i$ in $\fA$ and $s_i$ a tame corestriction on $A$ relatively to $F[\beta _ i ] / F$. The derived stratum $[\fB _{\beta_{i+1}} , r_{i+1} , r_{i+1} -1 , s_{i+1}(\beta _{ i} - \beta _{i+1} )]$ is equivalent to a simple stratum for $ 0 \leq i \leq s-1$.
\end{enumerate}

This family is not unique and is called a \textbf{defining sequence} for $[\fA , n , r , \beta ]$.

\end{coro}

To help the reader, we give an explanation for this corollary.

\begin{proof} $\bullet$ If $[\fA,n,r,\beta]$ is a simple stratum, put $[\fA , n , r_0 , \beta _0 ]=[\fA,n,r,\beta]$ (remark that $r_0 < -k_0 (\beta _0 , \fA)$). We now have an algorithm. If $\beta _0$ is minimal over $F$, put $s=0$,
then $(i)$ and $(ii)$ are obviously satisfied, $r=r_0 <n$ is satisfied by definition of a simple stratum and because the rest of condition $(iii)$ is empty. Condition $(iv)$ is empty in this case, and so is satisfied. Condition $(v)$ is satisfied by proposition \ref{approxi}. Condition $(vi)$ is empty in this case, and so is satisfied. If $\beta _0$ is not minimal, then consider the stratum $[\fA , n , -k_0(\beta_0,\fA) , \beta _0 ]$, it is pure but not simple. We now have a general process: the theorem \ref{approxi} shows that there exists a simple stratum $[\fA , n , -k_0(\beta_0,\fA) , \beta _1 ]$ equivalent to $[\fA , n , -k_0(\beta_0,\fA) , \beta _0 ]$ (remark that $[F[\beta _0 ] : F] > [ F[\beta _1] :F ]$ by \ref{degstriapprox}) such that for any tame corestriction $s_{\beta _1}$ the stratum ${[\fB _{\beta _0} , r , r-1 , s_{\beta _1} (\beta _0 - \beta _1 )]}$ is simple. Put $r_1 = -k_0(\beta_0 , \fA)$.
 If $\beta_1$ is minimal over $F$, put $s=1$, then the conditions $(i)$, $(ii)$, $(iii)$, $(iv)$ are now obviously satisfied. The condition $(v)$ is also satisfied by proposition \ref{minisimplealfalfa} and because $\beta _1$ is minimal over $F$. The condition $(vi)$ is now obviously satisfied.
 If $\beta _1$ is not minimal over $F$, then consider the stratum $[\fA , n , -k_0(\beta_1,\fA) , \beta _1 ]$, it is pure but not simple. As before, we apply the process to get a stratum $[\fA,n, -k_0(\beta _1 , \fA) , \beta _2 ]$ equivalent to $[\fA , n , -k_0(\beta_1,\fA) , \beta _1 ]$. Put $r_2=-k_0 (\beta _1 ,\fA)$. If $\beta _2$ is minimal, then put $s=2$. As before, the conditions $(i)$ to $(vi)$ are easily satisfied.
If $\beta _2$ is not minimal, then we can apply the process and get a simple stratum $[\fA,n,-k_0( \beta_2 , \fA), \beta _3]$, if $\beta _3$ is minimal we put $s=3$ and $r_3 =-k_0 (\beta _2 , \fA)$. If $\beta_3$ is not minimal, then we apply the process and get a new stratum and an element $\beta_4$, and so on. We claim that there exists an integer $s$ such that this algorithm stops; that is, $\beta _s$ is minimal. Assume the contrary, then we have an infinite strictly increasing sequence of numbers between $r$ and $n$
\begin{center}
{\small {{  $r=r_0 < r_1 = -k_0 (\beta _0 ,\fA) < r_2 = -k_0 (\beta_1 ,\fA) < \ldots < r_{i+1} =-k_0 (\beta_{i}, \fA) < \ldots <n$}}}
\end{center}

this is a contradiction and concludes the proposition in this case.

$\bullet$ If $[\fA, n , r, \beta]$ is pure but not simple, there exists a simple stratum $[\fA,n,r,\beta_0]$ equivalent to it and the previous case completes the proof.

\end{proof}

Fix a simple stratum $[\fA , n , r , \beta ]$, and let $r$ be the integer $-k_0 (\beta , \fA)$. The following is the definition of various groups and orders associated to  $[\fA , n , r , \beta ]$. Choose and fix a defining sequence $\{[\fA,n,r_i , \beta_i], 0 \leq i \leq s \}$ of $[\fA,n,r,\beta]$ (we thus have $\beta = \beta _0$). If $s>0$, the element $\beta_1$ is often denoted $\gamma$. We now define by induction on the length of the defining sequence various objects.

\begin{defi}
\label{defsimpl} \cite[3.1.7 ,3.1.8, 3.1.14]{BK}  \begin{enumerate}
\item[($i$)] Suppose that $\beta$ is minimal over $F$.
Put \begin{enumerate} \item $\mathfrak{H}(\beta,\fA)=\fB_{\beta} + \fP ^{[\frac{n}{2}]+1}$,
\item $\mathfrak{J}(\beta,\fA)=\fB_{\beta} + \fP ^{[\frac{n+1}{2}]}$. \end{enumerate}

\item[($ii$)] Suppose that $r<n$, and let $[\fA,n,r,\gamma]$ be the simple stratum equivalent to $[\fA,n,r,\beta]$ chosen in the previously fixed defining sequence.
Put \begin{enumerate} \item $\mathfrak{H}(\beta,\fA)=\fB_{\beta} + \mathfrak{H}(\gamma ,\fA) \cap \fP ^{[\frac{r}{2}]+1}$,
 \item $\mathfrak{J}(\beta,\fA)=\fB_{\beta} + \mathfrak{J}(\gamma ,\fA) \cap \fP^{[\frac{r+1}{2}]}$.
\end{enumerate}

\item[($iii$)] For $k \geq 0 $, put \begin{enumerate}
\item $\mathfrak{H}^k(\beta,\fA)=\mathfrak{H}(\beta,\fA)\cap \fP ^k$,

\item $\mathfrak{J}^k(\beta,\fA)=\mathfrak{J}(\beta,\fA)\cap \fP ^k$.
\end{enumerate}
\item[($iv$)] Finally, put $U^m(\fA)=\left(1+\mathfrak{P}^m \right)$ if $m>0$ and $U^m(\fA)= \mathfrak{A}^{\times}$ if $m=0$ and put \begin{enumerate}
\item $H^m(\beta, \fA)=\mathfrak{H}(\beta,\fA) \cap U^m(\fA) $,
\item $J^m(\beta, \fA)=\mathfrak{J}(\beta,\fA) \cap U^m(\fA)$.

The set $H^m(\beta, \fA)$ and $J^m(\beta, \fA)$ are groups. The group $J^0(\beta , \fA)$ is also denoted by $J(\beta , \fA)$.

\end{enumerate}
\end{enumerate}
 \end{defi} 
 \begin{rema} In the case $r<n$, $\mathfrak{H}(\beta , \fA)$ is defined inductively:  the order $\mathfrak{H }(\beta _s , \fA) $ is well-defined because $\beta _ s $ is minimal, and then $\mathfrak{H }(\beta _{s-1} , \fA) $ is well-defined, and so on. The same remark occurs for $\mathfrak{J}(\beta , \fA)$.
 \end{rema}
 \begin{rema}By \cite[3.1.7, 3.1.9 (v)]{BK}, $\mathfrak{J}^k(\beta,\fA)$ and  $\mathfrak{H}^k(\beta,\fA)$ are well-defined, and they do not depend on the choice of a defining sequence. So the same is true for $H^m(\beta, \fA)$ and $J^m(\beta, \fA)$.

 \end{rema}

 \begin{prop} \label{propertyHHJJ} \cite[3.1.15]{BK} Let $m\geq 0$ be an integer then the following assertions hold.

 \begin{enumerate}
 \item[$(i)$] The groups $H^m(\beta , \mathfrak{A} )$ and $J^m(\beta,\mathfrak{A})$ are normalised by $\mathfrak{K}(\mathfrak{B}_{\beta})$, so in particular by $F[\beta]^{\times}$.

 \item[$(ii)$] The group $H^m(\beta , \mathfrak{A} )$ is included in $J^m(\beta,\mathfrak{A})$.

 \item[$(iii)$] The group $H^{m+1}(\beta , \mathfrak{A})$ is a normal subgroup of  $J^0(\beta , \mathfrak{A})$.

 \end{enumerate}

 \end{prop}
  The following is devoted to the definition of the so called simple characters.
  Let $ \Psi$ be an additive character of $F$ with conductor $\mathfrak{p}_F$.
  Let $\psi _A$ be the function on $A$ defined by $\psi_{A}(x)=\psi \circ \Tr_{A/F} (x)$. To any  $b\in A$, is associated a function  $\psi _ b$ on $A$ given by
 \begin{center}

 $\psi _{b}(x)=\psi _A (b (x-1))$.

 \end{center}

 \begin{defi}\begin{enumerate}[(i)] \label{carsimpl}\item Suppose that $\beta $ is minimal over $F$. 

 For ${0 \leq m \leq n-1}$, let $\mathcal{C}(\fA,m,\beta)$ denote the set of characters $\theta$ of $H^{m+1}(\beta)$ such that:

 \begin{enumerate}
 \item $\theta \mid _{ H^{m+1}(\beta) \cap U^{[\frac{n}{2}]+1}(\fA)}=\psi_{\beta}$,
 \item $\theta \mid _{H^{m+1}(\beta) \cap B_{\beta}^{\times}}$ factors through  $\det_{B_{\beta}}:B_{\beta}^{\times} \to F[\beta]^{\times}.$
 \end{enumerate}

 \item Suppose that $r<n. $ For $0\leq m \leq r-1$, let $\mathcal{C}(\fA,m,\beta)$ be the set of characters $\theta$ of $H^{m+1}(\beta)$ such that the following conditions hold.\begin{enumerate}

 \item $\theta \mid H^{m+1}(\beta) \cap B_{\beta}^{\times}$ factors through $\det_{B_{\beta}}$
 \item $\theta$ is normalised by $\mathfrak{K}(\mathfrak{B}_{\beta})$
 \item if $m'=\max\{m,[\frac{r}{2}]\},$ the restriction $\theta \mid H^{m'+1}(\beta)$ is of the form $\theta _0 \psi _c $ for some $\theta _0 \in \mathcal{C}(\fA,m',\gamma)$ where $c=\beta -\gamma$ and $\gamma$ is the first element of the fixed defining sequence.

 \end{enumerate}
 \end{enumerate}
  \end{defi}
 \begin{rema} In the second case, $\mathcal{C}(\fA,m,\beta)$ is defined by induction: recall that we have fixed a defining sequence $\{[\fA , n , r_i , \beta _i ]$, $0\leq i \leq s\}$ of $[\fA,n,0,\beta]$, the last term of the defining sequence is such that $\beta _ s $ is minimal over $F$ and by the first case, there is a set of characters attached. Then, those attached to $[\fA , n , r_{s-1} , \beta _{s-1} ]$  are defined, and by iteration the set $\mathcal{C}(\fA,m,\beta)$ is defined. 
\end{rema}
\begin{prop} \cite[3.2]{BK} The set  $\mathcal{C}(\fA,m,\beta)$  defined above is independent of the choice of the defining sequence.
\end{prop}
 \begin{prop} \label{interwisimplecara} \cite[3.3.2]{BK} Let $[\fA,n,0,\beta]$ be a simple stratum in the algebra $A$. Put $r=-k_0(\beta,\fA)$. For $0\leq m \leq [\frac{r}{2}]$ and $\theta \in \mathcal{C}(\fA,m,\beta)$, the intertwining of $\theta$ in $G$ is given by 
 \begin{center}
 $I_{G}(\theta)=J^{[\frac{r+1}{2}]}(\beta,\fA) B_{\beta}^{\times} J^{[\frac{r+1}{2}]}(\beta,\fA)$.
 \end{center}
 \end{prop}

 \subsection{Simple types and representations}\label{simpletypesandrepresentations}
 This section is devoted to the definition of simple types and to one of the main theorems of Bushnell-Kutzko's theory.  
 Let $[\mathfrak{A},n,0,\beta]$ be a simple stratum and let $\theta \in \mathcal{C}(\beta, 0, \fA)$ be a simple character attached to this stratum. There exists a unique, up to isomorphism, irreducible representation $\eta$ of $J^1(\beta, \fA)$ containing $\theta$ \cite[5.1.1]{BK}. The dimension of $\eta$ is equal to $ [J^1(\beta , \fA) : H^1(\beta ,\fA )]^{\frac{1}{2}}$. \\

 \begin{defi} \label{betaextensiondefi}\cite[5.2.1]{BK} A $\beta $-extension of $\eta$ is a representation $\kappa$ of $J^0(\beta , \fA)$ such that the following conditions hold:

 \begin{enumerate}
 \item[($i$)] $\kappa \mid _{ J^1(\beta, \fA)}= \eta $
 \item[($ii$)] $\kappa$ is intertwined by the whole of $B^{\times} $.

 \end{enumerate}

 We say that $\kappa$ is a $\beta$-extension of $\theta$ if there exists an irreducible representation $\eta$ of $J^1(\beta, \fA)$ containing $\theta$ such that $\kappa$ is a $\beta$-extension of $\eta$.

 \end{defi}
  \begin{prop} \label{beta} Let $\kappa$ be an irreducible representation of $J^0(\beta , \fA)$. The following assertions are equivalent. \begin{enumerate}
  \item[($i$)] The representation $\kappa$ is a $\beta -$extension of $\theta$.
  \item[($ii$)] The representation $\kappa$ satisfies the following three conditions. \begin{enumerate}
  \item $\kappa $ contains $\theta$
  \item $\kappa$ is intertwined by the whole of $B^{\times} $
  \item $\dim(\kappa)= [J^1(\beta , \fA) : H^1(\beta ,\fA )]^{\frac{1}{2}}$.
  \end{enumerate}
  \end{enumerate}
  \end{prop}

  \begin{proof} If $\kappa$ is a $\beta-extension$, then $\kappa$ satisfies $(a),(b),(c)$. Indeed, by definition $\kappa$ restricted to $J^1(\beta, \fA)$ is equal to an irreducible representation $\eta$ that contains $\theta$, thus $\kappa$ contains $\theta$ and $\dim(\kappa)= \dim(\eta)= [J^1(\beta , \fA) : H^1(\beta ,\fA )]^{\frac{1}{2}}.$ By definition, $\kappa$ is intertwined by the whole of $B^{\times}$.   Reciprocally, if $\kappa$ satisfies $(a),(b),(c)$, then ${(\kappa \mid _{J^1 (\beta , \fA)}}) \mid _{H^1(\beta , \fA )}$ contains $\theta $, so $\kappa \mid _{J^1 (\beta , \fA )}$ contains an irreducible representation $\eta$ that contains $\theta$, and the equality on dimension thus shows  $\kappa \mid _{J^1 (\beta , \fA) }= \eta$. Thus $\kappa$ is a $\beta -$extension as required.
  \end{proof}

\begin{prop} \label{kappakappa} Let $\kappa _1$ and $\kappa _2$ be two $\beta$-extension of $\theta$. There exists a character $\chi : U^0(\mathfrak{o}_E) / U^1(\mathfrak{o}_E)\to \C ^{\times}$ such that $\kappa _1$ is isomorphic to ${\kappa _2 \otimes \chi \circ \det _{B}}$.
\end{prop}

\begin{proof}  There exists $\eta _1 $ and $\eta _2$, irreducible representations containing $\theta$, such that $\kappa _1$ is a $\beta$-extension of $\eta _1$ and $\kappa _2$ is a $\beta $-extension of $\eta _2$. The representation $\eta _1$ is isomorphic to $\eta _2$. The proposition \ref{kappakappa} is now a consequence of \cite[5.2.2]{BK}.

\end{proof}

\begin{defi}\label{defisimpletype} A simple type in $\mathrm{G}$ is one of the following $(a)$ or $(b)$.
\begin{enumerate}[(a)] 
\item An irreducible representation $\lambda = \kappa \otimes \sigma $ of $J(\beta, \fA )$ where:
\begin{enumerate}[(i)]
\item $\fA$ is a principal $\mathfrak{o} _F $-order in $\fA$ and $[\fA,n,0,\beta]$ is a simple stratum;
\item $\kappa$ is a $\beta -extension$ of a character $\theta \in \mathcal{C}(\fA,0,\beta)$;
\item if we write $E=F[\beta],\mathfrak{B} =\fA \cap \End_{E}(V)$, so that
\begin{center}
$J(\beta, \fA ) / J^1(\beta ,\fA) \simeq U(\mathfrak{B})/U^1(\mathfrak{B})\simeq \mathrm{GL}_f(k_E)^e$

\end{center}

for certain integers  $e$, $f$,
then $\sigma$ is the inflation of a representation $\sigma _0 \otimes \cdots \otimes \sigma _0 $ where $\sigma _0 $ is an irreducible cuspidal representation of $\mathrm{GL}_f(k_E)$,
\end{enumerate} 
\item \label{typeb} An irreducible representation $\sigma$ of $U(\fA)$ where:
\begin{enumerate}[(i)]
\item $\fA$ is a principal $\mathfrak{o} _ F -$order in $A$,
\item if we write $U(\fA)/U^1(\fA) \simeq \mathrm{GL}_f(k_F)^e, $ for certain integers $e,f$, then $\sigma$ is the inflation of a representation $\sigma_0 \otimes \cdots \otimes \sigma _0$, where $\sigma _0$ is an irreducible cuspidal representation of $\mathrm{GL}_f(k_F)$.
\end{enumerate}
 \end{enumerate}
\end{defi}

The following is one of the main theorems of Bushnell-Kutzko theory:

 \begin{theo}\cite[8.4.1]{BK} \label{theo8bkexi}
 Let $\pi$ be an irreducible supercuspidal representation of $\mathrm{G}=\Aut_F(V)\simeq \mathrm{GL}_N(F)$. There exists a simple type   $(J,\lambda)$ in $\mathrm{G}$ such that $\pi \mid J $ contains $\lambda$. Furthermore,
 \begin{enumerate}
 \item[($i$)] the simple type $(J,\lambda)$ is uniquely determined up to $\mathrm{G}$-conjugacy.
 \item[($ii$)] if $(J,\lambda)$ is given by a simple stratum $[\fA,n,0,\beta]$ in $A=\End_F(V)$ with $E=F[\beta]$, then there is a a uniquely determined representation $\Lambda$ of $E^{\times} J$ such that $\Lambda \mid _{ J} = \lambda $ and $\pi=\mathrm{\cind}(\Lambda)$, in this case $\mathfrak{A} \cap \End_E(V)$ is a maximal $\mathfrak{o}_E$-order $ \End_E(V)$.
 \item[($iii$)] if $(J,\lambda)$ is of the form  (\ref{typeb})---that is,  if  $J=U(\fA)$ for some maximal $\mathfrak{o} _F$-order  $\fA$ and $\lambda$ is trivial on $U^1(\fA)$---, then there is a uniquely determined representation $\Lambda$ of $F^{\times} U(\fA)$ such that $\Lambda \mid _{U(\fA)} = \lambda $ and $\pi = \mathrm{\cind} (\Lambda)$.
 \end{enumerate}
 \end{theo}

Let us now introduce a terminology specific to the purpose of this text. 

\begin{defi}\label{bkdatum} A Bushnell-Kutzko datum in $A$ is one of the following sequence: \begin{enumerate}[(a)]
\item An uple of the form $([\fA,n,0,\beta],\theta,\kappa, \sigma , \Lambda )$ such that:

\begin{enumerate}[(i)]
\item $[\fA,n,0,\beta]$ is a simple stratum in $A$ such that $\fB _{\beta}$ is a maximal $\mathfrak{o}_E$-order,
\item $\theta \in \mathcal{C}( \fA , 0 , \beta )$ is a simple character attached to  $[\fA,n,0,\beta]$,
\item $\kappa$ is a $\beta$-extension of $\theta$,
\item $\sigma$ is an irreducible cuspidal representation of $U^0(\fB _{\beta}) / U^1(\fB _{\beta})$,
\item $\Lambda$ is an extension to $E^{\times } J^0 (\beta , \fA) $ of $\kappa \otimes \sigma $ ($\Lambda$ is called a maximal extended simple type).
\end{enumerate}
\item An uple of the form $(\fA, \sigma ,\Lambda )$ where $\fA$ is a maximal $\of$-order in $A$, $\sigma$ is a cuspidal representation of $U^0(\fA) / U^1 (\fA)$ and $\Lambda$ is an extension to $F^{\times} U^0( \fA) $ of $\sigma$ .
\end{enumerate}
\end{defi}
\begin{rema} \label{superficial} As in definition \cite[5.5.10]{BK}, this distinction $(a)$ and $(b)$ is quite superficial (see the remark after \cite[5.5.10]{BK}).
\end{rema}

 We will associate a Yu datum to each Bushnell-Kutzko datum satisfying a tameness condition. The following is the definition of a tame Bushnell-Kutzko datum.

\begin{defi}\label{tamebkdatum} A tame Bushnell-Kutzko datum is a Bushnell-Kutzko datum $([\fA,n,0,\beta],\theta,\kappa, \sigma , \Lambda ) $ of type $(a)$ such that $[\fA,n,0,\beta]$ is a tame simple stratum (see \ref{tssdef1} for the definition of a tame simple stratum) or a Bushnell-Kutzko datum of type $(b)$. In this situation, $(E^{\times } J^{0} (\beta , \fA) , \Lambda )$ is called a tame Bushnell-Kutzko maximal extended simple type and $\theta$ is called a tame Bushnell-Kutzko  maximal simple character in $ G$.
\end{defi}

\section{Yu's construction of tame supercuspidal representations} \label{yu}
Given a connected reductive algebraic $F$-group $G$, Yu \cite{YU} constructs irreducible supercuspidal representations of $G(F)$, which are called tame. Adler's work \cite{Adle} has inspired parts of Yu's construction. Kim \cite{Kim} has proven that Yu's construction is exhaustive when the residual characteristic of $F$ is sufficiently big.  Fintzen has a better exhaustion result \cite{fintzen} using another method.
 In the following, we describe Yu's construction and introduce his notations, closely following Yu's paper \cite{YU}. For proofs and further details, see \cite{YU}.

  \subsection{Tamely ramified twisted Levi sequences and groups} \label{subsectionyutwi}
   In this section, we introduce some notations and facts relative to those used in Yu's construction. We refer to Sections 1 and 2 of \cite{YU} for proofs.

   We also refer the reader to \cite[6.4.1]{brti1} for a definition of the totally ordered commutative monoid $\tilde{\R}= \R ~\sqcup~~ \R+ ~ \sqcup ~\infty.$

   \begin{defi}
  A tame twisted Levi sequence $\overrightarrow{G}$ in $G$ is a sequence \begin{center}${(G^0 \subset G^1 \subset \ldots \subset G^d =G)}$ \end{center} of reductive $F$-subgroups of $G$ such that there exists a tamely ramified finite Galois extension $E/F$ such that 
  ${G^i \times _{F}E}$ is a split Levi subgroup of ${G \times _{F} E}$, for $0 \leq i \leq d$.
  \end{defi}

Let $\overrightarrow{G}$ be a tame twisted Levi sequence,  there exists a maximal torus $T \subset G^0$ defined over $F$ such that $T \times _{ F} E$ is split.
  For each $0 \leq i \leq d $, let $\Phi _i$ be the union of the set of roots $\Phi (G^i,T,E)$ and $\{0\}$; that is, $\Phi _i = \Phi (G^i,T,E) \cup \{0\}$.   For each $a \in \Phi _d \setminus \{0\}$, let $G_a \subset G=G^d$ be the root subgroup corresponding to $a$, and let $G_a$ be $T$ if $a=0$. Let $\mathfrak{g}(E)$ be the Lie algebra of $G$ over $E$, and and let $\mathfrak{g}^*(E)$ be the dual of $\mathfrak{g}(E)$. For each $a \in \Phi _d$ let $\mathfrak{g}_a (E)$ (resp $\mathfrak{g}^*_a (E)$ ) be the $a$-eigenspace of $\mathfrak{g}(E)$ (resp $\mathfrak{g}^*(E)$) as a rational representation of $T$. Then, $\mathfrak{g}_a (E)$ is the Lie algebra of $G_a$, and $\mathfrak{g}_a ^* (E) $ is the dual of $\mathfrak{g}_{-a} (E)$.
  If $ 0 \leq i \leq j \leq d$, we have $\Phi _i \subset \Phi _j $. 
   Let $\overrightarrow{r}= (r_0 , \ldots , r_i , \ldots , r_d) $ be a sequence of numbers in $\tilde{\R}$, we define a function $f _{\overrightarrow{r}}: \Phi (G^d , T , E) \to \tilde{\R}$ as follows: $f(a)=r_0$ if $a \in \Phi _0$, $f(a) =r_k $ if  $ a \in \Phi _k \setminus \Phi_{k-1}$.
  By definition, a sequence $\overrightarrow{r}=(r_0 , r_1 , \ldots , r_d)$ of numbers in $\tilde{\R}$ is admissible if there exists $\nu \in \Z$ such that $0 \leq \nu \leq d$ and 

\begin{center}
$0 \leq r_0 = \ldots = r _{\nu} , \frac{1}{2}r_{\nu}  \leq r_{\nu +1} \leq \ldots \leq r_d$.
\end{center}
  Let $y$ be in the apartment $A(G,T,E) \subset \BT^E(G,E)$. 
  For each $a \in \Phi _d$, let $\{G_a(E) _{y,r} \} _{r \in \tilde{\R} , r \geq 0 }$, $\{\mathfrak{g}_a (E) _{y,r} \} _{r \in \tilde{\R}} $ and $\{\mathfrak{g}_a ^* (E) _{y,r} \} _{r \in \tilde{\R}}$ the associated filtrations.  For any $\tilde{\R}$-valued function $f$ on $\Phi_d$ such that $f(0) \geq 0$, let $G(E)_{y,f}$ be the subgroup generated by $G_a(E) _{y , f(a)}$ for all $a \in \Phi _d$, and Yu similarly defines $\mathfrak{g}(E) _{y,f}$ and  $\mathfrak{g}^* (E) _{y,f}$. The group ${G}(E)_{y, f_{\overrightarrow{r}}}$, ${\mathfrak{g}}(E)_{y, f_{\overrightarrow{r}}}$
   and ${\mathfrak{g}}^*(E)_{y, f_{\overrightarrow{r}}}$ are denoted by  
   $\overrightarrow{G}(E) _{ y , \overrightarrow{r}} $, $\overrightarrow{\mathfrak{g}}(E) _{ y , \overrightarrow{r}} $
   and $\overrightarrow{\mathfrak{g}}^*(E) _{ y , \overrightarrow{r}} $.
   Let $\overrightarrow{r} , \overrightarrow{s}$ be two admissible sequences of elements in $\tilde{\R}$. We write $\overrightarrow{r} < \overrightarrow{s}$ (resp $\overrightarrow{r} \leq \overrightarrow{s}$) if $r_i < s_i$  (resp $r_i \leq s_i$) for $0 \leq i \leq d$.
   If $\overrightarrow{r} < \overrightarrow{s}$, Yu puts

  \begin{center}
  $ \overrightarrow{G}(E) _{y, \overrightarrow{r}: \overrightarrow{s}} = \overrightarrow{G}(E) _{y , \overrightarrow{r}} / \overrightarrow{G}(E) _{y, \overrightarrow{s}}$   and   $ \overrightarrow{\mathfrak{g}}(E) _{y, \overrightarrow{r}: \overrightarrow{s}} = \overrightarrow{\mathfrak{g}}(E) _{y , \overrightarrow{r}} / \overrightarrow{\mathfrak{g}}(E) _{y, \overrightarrow{s}}$.

   \end{center}
 We have assumed that $y \in A(G,T,E) \subset \BT^E(G,E)$. This determines a point $y_i$ in $A(G^i , T,E)$ modulo the action of $X_* (Z(G^i),E) \otimes _{\Z} \R$. A choice of $y_i$ determines an embedding $j_i : \BT^E (G^i , E) \to \BT^E (G,E)$, which is $G^i(E)$-equivariant and maps $y_i$ to $y$. We now fix $y_i$ for $0\leq i \leq d $ and identify $\BT^E(G^i,E) $ with its image in $\BT^E(G,E)$ under $j_i$. We thus identify $y_i$ with $y$. 

 \begin{prop} \cite{YU} The following assertions hold.
 \begin{enumerate}[(i)]
 \item $\overrightarrow{G}(E) _{y,\overrightarrow{r}}$, $\overrightarrow{\mathfrak{g}}(E) _{y,\overrightarrow{r}}$ and $\overrightarrow{\mathfrak{g}}(E) _{y,\overrightarrow{r}}$ are independent of the choice of $T$.
 \item  If $\overrightarrow{r}, \overrightarrow{s}$ are two admissible sequences such that 

 \begin{center}
 $0 < r_i \leq s_i \leq \min  (r_i, \ldots , r_d ) + \min (\overrightarrow{r}) $ for $0 \leq i \leq d$ 
 \end{center}
   then $\overrightarrow{G}(E) _{y , \overrightarrow{r}:\overrightarrow{s}}$ is abelian and isomorphic to $\overrightarrow{\mathfrak{g}}(E)_{y, \overrightarrow{r}: \overrightarrow{s}}$.

   \item If $\overrightarrow{r}$ is an admissible increasing sequence, then we have 
   \begin{center}
$\overrightarrow{G} (E)_{y, \overrightarrow{r}} = G^0(E)_{y,r_0} G^1(E)_{y,r_1}\ldots G^d(E) _{y,r_d}$  
   \end{center}\end{enumerate}
   where $G^i(E) _{y,r_i}$, $0 \leq i \leq d$, are Moy-Prasad's groups (see Notation).
  \end{prop}

  The sets $A(G,T,E)$ and $\BT^E(G,F)$ are both subsets of $\BT^E(G,E)$. Yu puts $A(G,T,F) = A(G,T,E) \cap \BT^E(G,F)$, it does not depend on the choice of  $E$.
  There exists $y \in A(G,T,F) \subset A(G,T,E)$. Let $\overrightarrow{r}$ be an ($\tilde{\R}$-valued) admissible sequence of length $d+1$. We define $\overrightarrow{G}(F)_{y,\overrightarrow{r}}$ to be $\overrightarrow{G}(E)_{y,\overrightarrow{r}} \cap G(F)$, which does not depend on the choice of $E$. The group $\overrightarrow{G}(E)_{y,\overrightarrow{r}}$ is Galois stable and $\overrightarrow{G}(F)_{y,\overrightarrow{r}}={\overrightarrow{G}(E)_{y,\overrightarrow{r}}}^{\Gal(E/F)}$.
  The lattices $\overrightarrow{\mathfrak{g}}(F)_{y,\overrightarrow{r}}$ and $\overrightarrow{\mathfrak{g}}^*(F)_{y,\overrightarrow{r}}$ are similarly defined. Yu puts $\overrightarrow{G}(F)_{y,\overrightarrow{r}:\overrightarrow{s}}= \overrightarrow{G}(F)_{y,\overrightarrow{r}} / \overrightarrow{G}(F)_{y,\overrightarrow{s}}$, 
  $\overrightarrow{\mathfrak{g}}(F)_{y,\overrightarrow{r}:\overrightarrow{s}}$ and $\overrightarrow{\mathfrak{g}}^*(F)_{y,\overrightarrow{r}:\overrightarrow{s}}$.

  \begin{prop} \cite{YU} Let $0\leq \overrightarrow{r} \leq \overrightarrow{s}$ and $\overrightarrow{s}>0$. Then 

  \begin{enumerate}[(i)]
  \item The natural morphisms of groups 

  \begin{center}
  $\overrightarrow{G}(F)_{y,\overrightarrow{r}:\overrightarrow{s}}\to{\overrightarrow{G}(E)_{y,\overrightarrow{r}:\overrightarrow{s}}}^{\Gal(E/F)}$
  \end{center}

  and 

  \begin{center}
   $\overrightarrow{\mathfrak{g}}(F)_{y,\overrightarrow{r}:\overrightarrow{s}}\to{\overrightarrow{\mathfrak{g}}(E)_{y,\overrightarrow{r}:\overrightarrow{s}}}^{\Gal(E/F)}$
  \end{center}

  are surjective.

  \item If $0<\overrightarrow{r} < \overrightarrow{s} , s_i \leq \min (r_i , \ldots , r_d ) + \min (\overrightarrow{r})$ for all $i$, and $E/F$ is a splitting field of $\overrightarrow{G}$ that is Galois and tamely ramified, then the isomorphism $\overrightarrow{G}(E)_{y , \overrightarrow{r}:\overrightarrow{s}} \to \overrightarrow{\mathfrak{g}} (E) _{y,\overrightarrow{r}: \overrightarrow{s}} $ induces an isomorphism 
  \begin{center}
  $\overrightarrow{G}(F)_{y,\overrightarrow{r}:\overrightarrow{s}} \to \overrightarrow{\mathfrak{g}}(F)_{y,\overrightarrow{r}:\overrightarrow{s}}.$
  \end{center}

  \end{enumerate}
  \end{prop}

  We have assumed that $y \in \BT^E (G,E) \cap A(G,T,E)$. We may assume that $y_i$ is fixed by $\Gal (E/F)$ . Then, $y_i$ is a point in $\BT^E(G^i , F)$ by a result of Rousseau. The embedding $j_i: \BT^E(G^i , E) \to \BT^E (G , E)$ is Galois equivariant, hence induces an embeddings $\BT^E(G^i , F) \to \BT ^E (G , F)$ by an other result of Rousseau. We identify $\BT^E(G^i,F)$ with its image in $\BT^E(G,F)$. Therefore, we identify $y_i$ with $y$.
  We now have another important proposition 

   \begin{prop} \cite[2.10]{YU} If $\overrightarrow{r}$ is increasing with $r_0 >0$, we have 
  \begin{center}
  $\overrightarrow{G}(F)_{y,\overrightarrow{r}}= G^0 (F) _{y, r_0} G^1 (F) _{y , r_1} \ldots G^d (F) _{y , r_d}$
  \end{center} 
   where $G^i(F) _{y,r_i}$, $0 \leq i \leq d$, are Moy-Prasad's groups (see Notation).
  \end{prop}

  \subsection{Generic elements and generic characters}\label{subsectionyugeneric}

  If $L$ is a lattice in an $F$-vector space $V$, then the dual lattice $L^*$ is defined to be

  \begin{center} ${\{x \in V^* \mid x(L) \subset \mathfrak{o}_F \}.}$ \end{center}
  Put $L^{\bullet}=L^* \otimes _{\mathfrak{o}_F} \mathfrak{p}_F$.
 If $L\subset M$ are lattices in $V$, then
 the Pontrjagin dual of $M/L$ can be identified with $L^{\bullet} / M^{\bullet}$ 
via an additive character $\psi_F$ of conductor $\mathfrak{p}_F$.  Explicitly, every element $a \in L^{\bullet}$ defines a character $\chi = \chi _a$ on $M$ by $\chi _a (m)= \psi _F (a(m))$.
  We say that $a$ realises the character $\chi$.
If $\overrightarrow{r}=(r_0 , \ldots , r_d)$ is an $\R$-valued sequence, then  $\overrightarrow{r}+$ denotes the sequence $(r_0+, \ldots , r_d+)$. 
Then, $\mathfrak{g}^*(F)_{y,\overrightarrow{r}}$ is equal to $\mathfrak{g}(F)_{y,(-\overrightarrow{r})+}^*\otimes _{\mathfrak{o}_F} \mathfrak{p}_F$ and 
$\mathfrak{g}^*(F)_{y,\overrightarrow{r}+}$ is equal to 
$\mathfrak{g}(F)_{y,-\overrightarrow{r}}^* \otimes _{\mathfrak{o}_F} \mathfrak{p}_F$. 
Let $r>0$ and let $S$ such that  $G(F)_{y,(r/2)+} \supset S \supset G(F)_{y,r}$. Then, $S/G(F)_{y,r+} \simeq \mathfrak{s}/\mathfrak{g}(F)_{y,r+}$, where $\mathfrak{s}$ is a lattice between $\mathfrak{g}(F)_{y,(r/2)+}$ and $\mathfrak{g}(F)_{y,r}$. 

\begin{defi} \label{realizedef} A character of $S/G(F)_{y,r+}$ is realised by an element $a \in \mathfrak{g}^*(F)_{y,-r} = (\mathfrak{g}(F)_{y,r+})^{\bullet}$ if it is equal to the composition

\centerline{\xymatrix{
S/G(F)_{y,r+} \ar[r] ^{\sim}& \mathfrak{s}/\mathfrak{g}(F)_{y,r+} \ar[r]^{\chi _a}& \C ^{\times}.
}}

\end{defi}

We now introduce the notion of generic element. A generic character will be defined as certain characters whose restrictions are realised by generic elements.
 Let $G' \subset G$ be a tamely ramified twisted Levi sequence.
 Let $Z'$ be the center of $G'$, and let $T$ be a maximal torus of $G'$. The space $\Lie ^*( (Z')^{\circ})$ can be regard as a subspace of $\Lie ^* (G')$ in a canonical way (see \cite[§8]{YU}).  The space $\Lie ^* (G')$ can also be regarded as a subspace of $\Lie ^* (G)$ in a canonical way.

 \begin{defi}
 An element $X^*$ of $(\Lie ^* (Z')^{\circ})_{-r} $ is called $G$-generic of depth $r\in \R$ if two conditions, \textbf{GE1} and \textbf{GE2}, hold. These conditions are explained below.  \end{defi} Let us explain \textbf{GE1}. Let $a$ denote a root 
 in $\Phi (G,T, \overline{F})$, let $a^{\vee}$ be the coroot of $a$, and let $\mathrm{d}a^{\vee}$ denote the differential of $a^{\vee}.$ Let $H_a$ denote the element $\mathrm{d}a^{\vee} (1)$.

 \begin{rema}\label{generema} In the following definition of Yu, it is implicit that we see $X^*$ canonically as an element in $\Lie^* ({Z'} ^{\circ} \times _{F}  \overline{F})$. This is done by remarking two elementary facts.
  First, $\Lie({\mathrm{G}} ^{\circ} \times _{F} \overline{F})$ is canonically isomorphic to $\Lie ( {\mathrm{G}} ^{\circ}) \otimes _{F} \overline{F}$ because $F$ is a field, and thus their duals are canonically isomorphic. The canonical injective map 

 \begin{align*}
 \Lie^* ({\mathrm{G}} ^{\circ}) &\to (\Lie ({\mathrm{G}} ^{\circ}) \otimes _F \overline{F})^*\\
  f &\mapsto  ( z \otimes \lambda \mapsto f(z) \lambda )
 \end{align*}
 ends this remark.
 \end{rema}

 \begin{defi} \label{defgeun} An element $X^*$ of $(\Lie ^* (Z')^{\circ})_{-r} $ satisfies \textbf{GE1} with depth $r$ if $\ord(X^* (H_a))=-r$ for all root $a \in \Phi (G,T, \overline{F}) \setminus \Phi (G,T, F)$.
\end{defi}

We refer to Section $8$ in 
\cite{YU}
for a definition of the condition \textbf{GE2}.In most cases, the condition \textbf{GE2} is implied by the condition. In particular, in this paper the condition \textbf{GE2} will always hold as soon as the condition \textbf{GE1} will hold thank to the following propositions.
 We refer to Section 7 of \cite{YU}, or \cite{Steinbergtor} for the notion of torsion prime for a root datum.

\begin{prop} \label{ge1implyge2tors}\cite[8.1]{YU} If the residual characteristic of $F$ is not a torsion prime for the root datum $(X, \Phi (G,T,\overline{F}) , X ^{\vee} , \Phi ^{\vee}(G,T,\overline{F})$, then \textbf{GE1} implies \textbf{GE2}.
\end{prop}

\begin{prop} \cite{Steinbergtor} Let $(X,\Phi , X ^{\vee} , \Phi ^{\vee} )$ be a root datum of type $A$. Then, the set of torsion prime for $(X,\Phi , X ^{\vee} , \Phi ^{\vee} )$ is empty.
\end{prop}

As announced earlier, the definition of a generic element is as follows:

\begin{defi} \label{defelegeneric} An element $X^*$ of $(\Lie ^* (Z')^{\circ})_{-r} $ is called $G$-generic of depth $r\in \R$ if the conditions \textbf{GE1} and \textbf{GE2} hold.
\end{defi}

We can now give Yu's definition of generic characters:

\begin{defi}\begin{enumerate}[(i)]
\item A character $\chi$ of $G'(F)$ is called $G$-generic if it is realised  (in the sense of definition \ref{realizedef}) by an element ${X^*}$  in 

 ${(\Lie ^* (Z')^{\circ})_{-r} \subset (\Lie ^* G') _{y,-r}}$, which is $G$-generic of depth $r$.
\item  A character $\boldsymbol{\Phi}$ of $G'(F)$ is called $G$-generic (relative to y) of depth $r$ if $\boldsymbol{\Phi}$ is trivial on $G'(F)_{y,r+}$, non-trivial on $G'(F)_{y,r} $ and $\boldsymbol{\Phi}$ restricted to $G'(F)_{y,r:r+}$ is $G$-generic of depth $r$ in the sense of $(i)$.\end{enumerate}
\end{defi}

  \subsection{Yu data}
In this article we will only use the following notion of Yu datum, and will only consider Yu's construction and associated objects for these Yu data.
\begin{sloppypar}
 \begin{defi} \label{defyudatum} A Yu datum consists in the following objects.
 \begin{enumerate}
 \item[($\overrightarrow{G}$)]  An anisotropic tame twisted Levi sequence in $G$, that is 
  \begin{center}{$G^0\subset \cdots\subset G^i \subset \cdots \subset G^d=G $}\end{center} such that
 \begin{enumerate}
 \item there exists a finite tamely ramified Galois extension $E/F$ such that ${G^i \times _{F} E}$ is a split Levi subgroup of ${G \times _{F} E}$, 
 \item $Z(G^0)/Z(G)$ is anisotropic.

 \end{enumerate}
 \item[($y$)] A point $y \in \BT^E(G^0,F) \cap A(G,T,E)$  where $T$ is a maximal torus of $G^0$, such that $T \times _{F} E $ is split and $A(G,T,E)$ denotes the apartment associated to $T$ over $E$,
 \item[($\overrightarrow{r}$)] A sequence of real numbers $0<\mathbf{r}_0<\mathbf{r}_1<...<\mathbf{r}_{d-1}\leq \mathbf{r}_d$ if $d>0$ , $0\leq \mathbf{r}_0$ if $d=0$,
 \item[($\rho$)] An irreducible representation $\rho$ of $K^0=G^0_{[y]}$ such that ${\rho \mid _{G^0(F)_{y,0+}}=1}$ and such that $\pi_{-1}:=\mathrm{\cind}_{K^0}^{G^0(F)}(\rho)$ is irreducible and supercuspidal.
 \item[($\overrightarrow{\boldsymbol{\Phi}}$)] A sequence ${\boldsymbol{\Phi} _0,\ldots, \boldsymbol{\Phi} _d }$ of characters of ${G^0(F),\ldots,G^d(F)}$. We assume that $\boldsymbol{\Phi} _i $ is trivial on $G^i(F)_{y,\mathbf{r}_i+} $ but not on $G^i(F)_{y,\mathbf{r}_i} $ for ${0\leq i \leq d-1}$. If ${\mathbf{r}_{d-1} < \mathbf{r}_d}$, then we assume that $\boldsymbol{\Phi} _d$ is trivial on $G^d(F)_{y,\mathbf{r}_d +}$ but not on $G^d(F)_{y,\mathbf{r}_d}$. If $\mathbf{r}_{d-1} = \mathbf{r}_d$, then we assume that $\boldsymbol{\Phi} _d=1$.
The characters are assumed to satisfy Yu's generic condition: $\boldsymbol {\Phi} _i $ is $G^{i+1}$-generic of depth $\mathbf{r} _i$ for $0\leq i \leq d-1$. 
 \end{enumerate}

 \end{defi}
 \end{sloppypar}
 \subsection{Yu's construction} \label{yuconstrutru}
 We will fix a generic Yu datum.
 The three first objects $(\overrightarrow{G} , y ,\overrightarrow{r} )$ allow us to define various groups. The point $y$ can be seen as a point in the enlarged Bruhat-Tits Building of $G^i$ for each $i$ using embeddings  \[\BT^E(G^0,F)\hookrightarrow \BT^E(G^1,F) \hookrightarrow \ldots \hookrightarrow \BT^E(G^d,F) ,\] which  was explained in Section 2 of Yu's paper \cite[§2 , page 589 line 5]{YU}. Yu defined several groups:

 \begin{defi} \label{defyugroup} \cite[§3, 15.3]{YU} Put $\mathbf{s}_i=\frac{\mathbf{r}_i}{2}$ for $0 \leq i \leq d.$

 For $i=0$, put 

  \begin{enumerate}[(i)]

 \item $K^0 _+ = G^0(F)_{y,0+}$

 \item $^{\circ} K^0 = G^0(F)_{y}$

 \item $ K^0=G^0(F)_{[y]}$.
 \end{enumerate}

  For  $1\leq i \leq d $, put

 \begin{enumerate}[(i)]

 \item \begin{align*}
 K^i _+ &= G^0(F)_{y,0+} G^1(F)_{y,\mathbf{s}_0+} \cdots G^i(F)_{y,\mathbf{s}_{i-1}+} \\
 & =(G^0, G^1 , \ldots , G^i) (F) _{y,(0+,s_0+,\ldots , s_{i-1}+)}
 \end{align*}

 \item \begin{align*}
^{ \circ} K^i&=G^0(F)_{y} G^1(F)_{y,\mathbf{s}_0} \cdots G^i(F)_{y,\mathbf{s}_{i-1}}\\
 &= G^0(F) _{y}(G^0, G^1 , \ldots , G^i) (F) _{y,(0,s_0,\ldots , s_{i-1})}
 \end{align*}

 \item \begin{align*}
  K^i&=G^0(F)_{[y]} G^1(F)_{y,\mathbf{s}_0} \cdots G^i(F)_{y,\mathbf{s}_{i-1}}\\
  &=G^0(F) _{[y]}(G^0, G^1 , \ldots , G^i) (F) _{y,(0,s_0,\ldots , s_{i-1})}.
  \end{align*}

 \end{enumerate}
 \end{defi}

 \begin{prop} \label{propyugroup}\cite{YU} Let $0 \leq i \leq d$.
 \begin{enumerate}[(i)]
 \item The three objects $K_+^i$, $ ^{\circ}K^i$, $K^i $ defined precedently are groups.
 \item They do not depend on the choice of the embeddings 

 \begin{center}$\BT^E(G^0,F) \hookrightarrow \BT^E(G^1,F) \hookrightarrow  \ldots \hookrightarrow  \BT^E(G^i,F)$.\end{center}
 \item
 There are inclusions $K_+^i ~\subset ~ ^{\circ}K^i ~\subset ~K^i$.
 \item The groups $K_+^i$ and $ ^{\circ} K^i$ are compact and $  K^i$ is compact modulo the center. Moreover $^{\circ} K^i$ is the maximal compact subgroup of $K^i$.
\end{enumerate} \end{prop}

Yu also defined  groups $J^{i}$ and $J^{i}_+$ for $1\leq i \leq d$, as follows. For $1 \leq i \leq d$ , $(r_{i-1} , s_{i-1} )$ and $(r_{i-1} , s_{i-1}+)$ are admissible sequence 

\begin{defi} 
 Let $J^i $ be the group $(G^{i-1}, G^i ) (F) _{(r_{i-1}, s_{i-1})}
$ and $J^i_+$ be the group $(G^{i-1}, G^i ) (F) _{(r_{i-1}, s_{i-1}+)}
$. 
\end{defi}

\begin{prop} Let $0 \leq i \leq d-1$. The following equalities of groups  hold:
 \begin{enumerate}[(i)]
\item $K^{i-1} J^i = K^i$
\item $K_+^{i-1} J_{+}^i=K_+^i$.
\end{enumerate}
\end{prop}
 Thanks to $\overrightarrow{\boldsymbol{\Phi}}$, Yu defines a character $\prod\limits_{i=1}^d \hat{\boldsymbol{\Phi}}_i$ on $K^d _+$. He then constructs a representation $\rho _d=\rho _d( \overrightarrow{G},y,\overrightarrow{\mathbf{r}},\rho ,\overrightarrow{\boldsymbol{\Phi}}) $ on $K^d$ \cite[§4]{YU}. Let us explain the construction of these objects.
  Let $0 \leq i \leq d-1$.
   Put $T^i=(Z(G^i))^{\circ}$, let us consider the adjoint action of $T^i$ on $\mathfrak{g}$, the space $\mathfrak{g}^i=\Lie(G^i)$ is the maximal subspace on which $T^i$ acts trivially. Let $\mathfrak{n}^i$ be the sum of the remaining isotypic subspaces.
 Let $\mathbf{s} \geq 0 \in \tilde{\R}$, then $\mathfrak{g}(F) _s = \mathfrak{g} ^i (F) _s \oplus \mathfrak{n} ^i (F) _s $ where $\mathfrak{n} ^i (F) _s \subset \mathfrak{n} ^i (F)$.
 There exists a sequence of morphisms as follows (see \cite[section 4]{YU}).
 \begin{equation}
G^i(F)_{\mathbf{s}_{i+} :\mathbf{r}_{i+}} \simeq \mathfrak{g}^i(F)_{\mathbf{s}_{i+} :\mathbf{r}_{i+}} \subset \mathfrak{g}^i(F)_{\mathbf{s}_{i+} :\mathbf{r}_{i+}} \oplus \mathfrak{n}^i(F)_{\mathbf{s}_{i+} :\mathbf{r}_{i+}}\simeq G(F)_{\mathbf{s}_{i+} :\mathbf{r}_{i+}}\label{iso}
 \end{equation}
 The character $\boldsymbol{\Phi} _i$ of $G^i(F)$ is of depth $\mathbf{r} _i $. Thus, thanks to the isomorphism  \eqref{iso}, it induces a character on  $\mathfrak{g}^i(F)_{\mathbf{s}_{i+} :\mathbf{r}_{i+}}$. We extend the latter to   $\mathfrak{g}^i(F)_{\mathbf{s}_{i+} :\mathbf{r}_{i+}} \oplus \mathfrak{n}^i(F)_{\mathbf{s}_{i+} :\mathbf{r}_{i+}}$ by decreting that it is  $1$ on $\mathfrak{n}^i(F)_{\mathbf{s}_{i+} :\mathbf{r}_{i+}}$. Thanks to the isomorphim  \ref{iso}, we obtain a character on $G(F)_{\mathbf{s}_{i+}}$, which Yu denotes by $\hat{\boldsymbol{\Phi}}_i$. By construction, the following equality holds: $\hat{\boldsymbol{\Phi}}_i \mid _{G^i(F)_{\mathbf{s}_{i+}}}= \boldsymbol{\Phi} _i \mid _{G^i(F)_{\mathbf{s}_{i+}}}$. There exists a unique character on $G^0(F)_{[y]}G^i(F)_0 G(F)_{\mathbf{s} _{i+}}$ which extends $\boldsymbol{\Phi} _i$ and $\hat{\boldsymbol{\Phi}} _i$. Yu also denotes this character by the symbol $ \hat{\boldsymbol{\Phi}} _i$. Because of
 $K_+^d \subset G^0(F)_{[y]}G^i(F)_0 G(F)_{\mathbf{s} _{i+}}$, we have defined a character $ \hat{\boldsymbol{\Phi} }_i$ on $K_+^d$. The character  $ \hat{\boldsymbol{\Phi} }_i$ depends only on $( \overrightarrow{G},y,\overrightarrow{\mathbf{r}}, \boldsymbol{\Phi}_i)$, we sometimes denote it as $ \hat{\boldsymbol{\Phi}} _i= \hat{\boldsymbol{\Phi} }_i( \overrightarrow{G},y,\overrightarrow{\mathbf{r}}, \boldsymbol{\Phi}_i)$. Let $ \theta _d =\theta _d ( \overrightarrow{G},y,\overrightarrow{\mathbf{r}}, \overrightarrow{\boldsymbol{\Phi}})$ be the character $\prod\limits_{i=0}^d\hat{ \boldsymbol{\Phi} }_i \mid _{K_+ ^d}$. We put $\hat{\boldsymbol{\Phi}} _d= \boldsymbol{\Phi} _d$.
Then, Yu constructs for $ 0 \leq j \leq d$ a representation $\rho _j$ of $K^j$. The compactly induced representation $\cind ^{G^j(F)}_{K^j}(\rho _j)$ is an irreducible and supercuspidal representation of $G^j(F)$. However, we are mainly interested in the case $j=d$. We also write $\rho_d= \rho_d ( \overrightarrow{G},y,\overrightarrow{\mathbf{r}}, \overrightarrow{\boldsymbol{\Phi}},\rho).$ We will use similar notations in the following. For each $j$, the representation $\rho _j$ of $K^j$ is naturally expressed as a tensor product of representations.

\begin{lemm} \cite[§4]{YU}\label{tildisemi}Let $0 \leq i \leq d-1$, there exists a canonical irreducible representation $\tilde{\boldsymbol{\Phi}}_i$ of $K^i \rtimes J^{i+1}$ such that the following conditions hold:

\begin{enumerate}[(i)]

\item The restriction of $\tilde{\boldsymbol{\Phi}}_i$ to $1 \ltimes J_+^{i+1} $ is $(\hat{\boldsymbol{\Phi}}_i \mid _{J^{i+1}_+})-$isotypic.

\item The restriction of $\tilde{\boldsymbol{\Phi}}_i$ to $K^i_+ \ltimes 1 $ is $1-$isotypic. 

\end{enumerate}
 \end{lemm}

  \begin{lemm} \label{inflfacto} Let $0 \leq i \leq d-1$. Let $\infl (\boldsymbol{\Phi} _i ) $ be the inflation of $\boldsymbol{\Phi} _i \mid _{K ^i}$ to $K^i \ltimes J^{i+1}$. 
  Let $\tilde{\boldsymbol{\Phi}}_i$ be the canonical irreducible representation introduced in lemma \ref{tildisemi}. Then $\infl (\boldsymbol{\Phi} _i ) \otimes \tilde{\boldsymbol{\Phi}}_i$ factors through the map  \begin{center}${K^i \ltimes J^{i+1} \to K^i J^{i+1} = K^{i+1}}$. \end{center}

 \end{lemm}

 \begin{proof} This is easy and is proven in Section 4 of \cite{YU}.

 \end{proof}

 \begin{defi} \label{phiprim}
 Let us denote by $\boldsymbol{\Phi}_i'$ the representation of $K^{i+1}$ whose inflation to $K^i \ltimes J^{i+1} $ is $\infl (\boldsymbol{\Phi} _i ) \otimes \tilde{\boldsymbol{\Phi}}_i$.

 \end{defi}

 \begin{lemm} \label{lemmhamuinfid} \cite[page 50]{Hamu}  The following assertions hold.\begin{enumerate}[(i)] \item If $\mu$ is a representation of $K^i$ which is $1$-isotypic on $K^i \cap J^{i+1} = G^i(F) _{y,\mathbf{r}_i}$ then there is
a unique extension of $\mu$ to a representation, denoted by $\inf ^{K^{i+1}} _{K^i} (\mu),$ of $K^{i+1}$, which is $1$-isotypic on $J^{i+1}$. If $i<d-1$, then this inflated representation is $1$-isotypic on $K^{i+1} \cap J^{i+2}$.
 \item  We may repeatedly inflate $\mu$. More precisely, if $0 \leq i \leq j \leq d$, then we may define $\inf^{K^j} _{K^i} ( \mu ) = \inf ^{K^j}_{ K^{j-1}} \circ \ldots \circ  \inf ^{K^{i+1}}_{K^i} (\mu)$.
\end{enumerate}
\end{lemm} 

 \begin{defi} \label{kappaijijiji} Let $0 \leq j \leq d$.  Let $ 0\leq i < j$. Let $\kappa ^j _i $ be the inflation of $ \boldsymbol{\Phi}_i'$ to $K^j$; that is, $\kappa ^j _i=\inf ^{K^j}_{ K^{i+1}} (\boldsymbol{\Phi}_i')$. Let $\kappa _j ^j $ be $\boldsymbol{\Phi}_j \mid _{K^j}$. Let $\kappa _{-1} ^j$ be the inflation of $\rho $ to $K^j$; that is, $\kappa _{-1}^j  = \inf _{K^0}^{K^j} (\rho)$.

 \end{defi}

 If $j=d$ and $-1 \leq i \leq d$, then we also denote $\kappa ^d _i$ by $\kappa _i$.  This notation and the statement of the following proposition is due to Hakim-Murnaghan.

 \begin{prop} \label{rhortensorproduct}The representation $\rho _j$ constructed by Yu is isomorphic to

 \begin{center} $ \kappa _{-1} ^j\otimes \kappa _0 ^j\otimes \ldots \otimes \kappa _j^j$.
 \end{center}
  In particular, the representation $\rho_d$ constructed by Yu is isomorphic to \begin{center} $\kappa _{-1} \otimes \kappa _0 \otimes \ldots \otimes \kappa _d$.
 \end{center}
 \end{prop}
 \begin{sloppypar}
 \begin{proof} 
 The representation $\rho _j $ is constructed in \cite{YU} on page 592. Yu inductively constructs two representations: $\rho _j$ and ${\rho _j} '$.

 Let us show by induction on $j$ that ${{\rho _j }'= \kappa ^j _{-1} \otimes \kappa ^j _0 \otimes   \ldots \otimes \kappa ^j _{j-1}}$ and

 ${\rho _j = \kappa ^j _{-1} \otimes \kappa ^j _0  \otimes \ldots \otimes \kappa ^j _{j}}$  
 If $j=0$, then by definition the representation $\rho _0 '$ constructed by Yu is $\rho$ and $\rho _0$ is $\rho _0 ' \otimes (\boldsymbol{\Phi}_0 \mid _{K^0})$. We have $\kappa _{-1}^0 = \rho $ and $\kappa _0 ^0 = {\boldsymbol{\Phi}_0}\mid _{K^0} $. So the case $j=0$ is complete.
 Assume that ${\rho _{j-1}' = \kappa ^{j-1} _{-1} \otimes \kappa ^{j-1} _0  \ldots \otimes \kappa ^{j-1} _{j-2}}$ and ${\rho _{j-1} = \kappa ^{j-1} _{-1} \otimes \kappa ^{j-1} _0 \otimes \ldots \otimes \kappa ^{j-1} _{j-1}}$. 
 Then by definition $\rho _j '$ is equal to ${\infl _{K^{j-1}}^{K^j} (\rho _{j-1} ' ) \otimes \boldsymbol{\Phi} _{j-1} '} $. By definition $\boldsymbol{\Phi} _{j-1} '$ is equal to $\kappa ^j _{j-1}$. 
 Moreover \begin{align*} \infl  _{K^{j-1}}^{K^j} (\rho _{j-1} ' )& =\infl  _{K^{j-1}}^{K^j} (\kappa ^{j-1} _{-1} \otimes \kappa ^{j-1} _0 \otimes \ldots \otimes \kappa ^{j-1} _{j-2})\\&=  \infl_{K^{j-1}}^{K^j}( \kappa ^{j-1} _{-1} )\otimes \infl _{K^{j-1}}^{K^j} (\kappa ^{j-1} _0) \otimes \ldots \infl _{K^{j-1}}^{K^j} (\kappa ^{j-1} _{j-2})\\
 &=\kappa ^j _{-1} \otimes \kappa ^j _0 \otimes  \ldots \otimes \kappa ^j _{j-2}
 \end{align*}

 Consequently ${{\rho _j }'= \kappa ^j _{-1} \otimes \kappa ^j _0 \otimes \ldots \otimes \kappa _i ^j \otimes \ldots \otimes \kappa ^j _{j-1}}$. Finally, by Yu's definition, $\rho _j $ is equal to $\rho _j ' \otimes \boldsymbol{\Phi} _j \mid _{K^j} $, and thus $\rho _j =\kappa ^j _{-1} \otimes \kappa ^j _0\otimes \ldots \otimes \kappa ^j _{j}$, as required.

 \end{proof}
\end{sloppypar}

\begin{prop}\label{dimensionphiprim} Let $0 \leq j \leq d $. Let $0 \leq i < j$. The dimension of $\kappa _i^j$ is equal to the dimension of $\boldsymbol{\Phi}_i '$. The dimension of $\boldsymbol{\Phi}_i '$ is equal to $[J^{i+1} : J^{i+1}_+]^{\frac{1}{2}}$.

\end{prop}

\begin{proof} By definition $\kappa _i ^j$ is an inflation of $\boldsymbol{\Phi}_i$, consequently theses representations have equal dimensions. The representation $\boldsymbol{\Phi}_i '$ is the unique representation of $K^i+1$ whose inflation to $K^i \ltimes J^{i+1}$ is  $\tilde{\boldsymbol{\Phi}}_i$. Thus, the dimension of $\boldsymbol{\Phi}_i '$ is equal to $\tilde{\boldsymbol{\Phi}}_i$. The representation $\tilde{\boldsymbol{\Phi}}_i$ is constructed in \cite[11.5]{YU} and is the pull back of the Weil representation of ${Sp(J^{i+1}/J^{i+1}_+) \ltimes (J^{i+1}/N_i)}$ where $N_i= \ker (\hat{\boldsymbol{\Phi}}_i)$ (see \cite{YU}). Thus, the dimension of $\tilde{\boldsymbol{\Phi}}_i$ is $[J^{i+1} : J^{i+1}_+]^{\frac{1}{2}}$.

\end{proof}

 \begin{theo} \label{rhodinductsuper} (Yu)  \cite[4.6 , §15]{YU} The representation $\mathrm{\cind}_{K^d}^{G(F)}\rho_d$ is irreducible and supercuspidal.

 \end{theo}

 \begin{defi} \label{yuextended} We say that $(K^d ,\rho_d)$  is a tame Yu extended type for $G=G^d$.\\

 \end{defi}

 \begin{sloppypar}
 We now introduce some notations for later use.
Put ${}^{\circ}\rho_d={}^{\circ}\rho_d( \overrightarrow{G},y,\overrightarrow{\mathbf{r}},\rho , \overrightarrow{\boldsymbol{\Phi}})=\rho_d \mid _{^{\circ}K^d}$. Put also ${}^{\circ} \kappa _i = \kappa _i  \mid _{ {}^{\circ} K_d }$ and ${}^{\circ}\lambda  = ^{\circ}\kappa _0  \otimes \ldots \otimes ^{\circ} \kappa _d $.
The following theorem shows that the construction of Yu is exhaustive when the residual characteristic is sufficiently large. \end{sloppypar}
 \begin{theo}
 (Kim) \cite{Kim} \label{exhaukim} Let $G$ be a connected reductive $F$ group, if the residue characteristic
$p$ of $F$ is sufficiently large, for each irreducible supercuspidal representation $\pi$ of $G(F)$, then there exists $(\overrightarrow{G},y,\overrightarrow{\mathbf{r}},\rho ,\overrightarrow{\boldsymbol{\Phi}})$, such that $\pi=\mathrm{\cind}_{K^d}^{G(F)} \rho _d( \overrightarrow{G},y,\overrightarrow{\mathbf{r}},\rho ,\overrightarrow{\boldsymbol{\Phi}})$.
  \end{theo}

   Fintzen has recently ameliorated this exhaustion result \cite{fintzen}.

 \section{Tame simple strata } \label{sectamesimple}
 
This section will focus on the approximation process for simple strata $[\fA, n , r , \beta ]$, which was previously described in section \ref{sectionconstbk}, when the field extension $F[\beta]/F$ is tamely ramified. In this situation, an approximation element $\gamma$ can be chosen inside the field $F[\beta]$. We will refer to Bushnell-Henniart for this result, which will be recalled as proposition \ref{tssprop2} in this section. The main new result in this section is Proposition \ref{tssprop3}. The Proposition \ref{tssprop1} is used to prove Proposition \ref{tssprop3}.

 \begin{defi}\label{tssdef1} A pure (resp simple) stratum $[\fA, n , r , \beta ]$ is a tame pure (resp tame simple) stratum if the field extension $F[\beta]/F$ is tamely ramified.
 \end{defi}

 Let $[\fA , n ,r , \beta ]$ be a tame pure stratum in the algebra $A=\End_F (V)$, set $E=F[\beta]$. Also set $B_E = \End_E(V)$.  Let $s:A \to B_E$ be the tame corestriction that is the identity  of $B_E$, we recall that such maps exist by \ref{tamecoresprop}. The element $s(b)$ is denoted by "$b$" when $b$ is in $B_E$. Let $\fP$ be the Jacobson radical of $\fA$. Set $\fB _E = \fA \cap B_E $ and $\fQ _E = \fP \cap \fB _E $. Thus, $\fB _E$ is an $\mathfrak{o}_E$-hereditary order in $B_E$ and $\fQ _E$ is the Jacobson radical of $\fB _E$.
 The following is a analogous to \cite[2.2.3]{BK}, the main differences are that the tameness condition is assumed and a maximality condition is removed.

 \begin{prop}\label{tssprop1} Let $[\fA,n,r,\beta ] $ be a tame simple stratum. Let $b \in \fQ _E ^{-r}$, and suppose that the stratum $[\fB _E , r , r-1 , b ]$ is simple. Then

 \begin{enumerate}

 \item[(i)] The stratum $[\fA , n , r-1, \beta +b ]$ is simple

 \item[(ii)] The field $F[ \beta +b ]$ is equal to the field $F[ \beta , b ]$

 \item[(iii)]We have

 \begin{center} $k_0 (\beta +b , \fA ) =\left\{ \begin{array}{ll}
        -r=k_0(b,\mathfrak{B}_E)$ if $b \not \in E \\
        k_0(\beta , \mathfrak{A})$ if $b\in E
    \end{array}  
\right.$

 \end{center}
 \end{enumerate}
 \end{prop}

 \begin{proof} Let $\mathcal{L}=\{L_i\} _{i\in \Z} $ be an $\mathfrak{o} _F$-lattice chain such that
  \begin{center}${\fA = \{x \in A \mid x(L_i) \subset L_i , i \in \Z \}.}$ \end{center}
 By definition \cite[2.2.1]{BK}, \begin{center}${\mathfrak{K} (\fA) = \{ x \in G \mid x(Li) \in  \mathcal{L}, i \in \Z \}} $\end{center} and \begin{center} ${ \mathfrak{K} (\fB _E)=\{x \in G_E \mid x(L_i) \in \mathcal{L}, i \in \Z \}.} $ \end{center} Thus, \begin{equation}\label{kk}\mathfrak{K}(\fB _E) \subset \mathfrak{K} (\fA).\end{equation}The
 stratum $[\fB _{E} , r , r-1 , b ]$ is simple, thus the definition of a simple stratum shows that  \begin{equation} \label{ek}E[b] ^{\times } \subset \mathfrak{K}(\mathfrak{B}_E).\end{equation} Put $E_1= E[b]= F[\beta , b]$.
Equations \ref{kk} and \ref{ek} imply that $E_1 ^{\times} \subset \fK (\fA)$. This allows us to use the machinery of \cite[1.2]{BK} for $\fA$ and $E_1$.

  Set $B_{E_1}= \End_{E_1} (V)$ and $\fB _{E_1} = \fA \cap \End _{E_1} (V)$. The proposition \cite[1.2.4]{BK} implies that $\mathfrak{B}_{E_1}$ is an $\mathfrak{o}_{E_1}$-hereditary order in $B_{E_1}$. Let $A(E_1)$ be the algebra $\End_F(E_1)$ and let $\fA (E_1)$ be the $\fo _F$-hereditary order in $ A(E_1)$ defined by $\mathfrak{A}(E_1) = \{x \in \End _F (E_1) \mid x (\fp _{E_1} ^i) \subset \fp _{E_1} ^i , i \in \Z \}. $ Let $W$ be the $F$-span of an $\fo_{E_1}$-basis of the $\fo _{E_1}$-lattice chain $\mathcal{L}$. The proposition \cite[1.2.8]{BK} shows that the $(W,E_1)$-decomposition of $A$ restricts to an isomorphism $\fA \simeq \fA (E_1) \otimes _{\fo _{E_1} } \fB $ of $(\fA (E_1) , \fB _{E_1})$-bimodules. Similarly, we have  a decomposition ${\fB _E \simeq \fB _E (E_1) \otimes _{\fo _{E_1}} \fB _{_{E_1}}}$.
    Set $B_E (E_1)= \End_E (E_1)$ and $\fB _E (E_1) = B_E (E_1) \cap \fA (E_1)$. Also set $n(E_1)= \dfrac{n}{e(\fB _{E_1} \mid \fo _{E_1})}$ and $r(E_1)= \dfrac{r}{e(\fB _{E_1} \mid \fo _{E_1})}$. Let us prove that the following two equalities hold.

   \begin{equation} \label{111} \nu _{\fA (E_1)} (\beta) =-n (E_1)      \end{equation}

   \begin{equation}
    \nu _{\fB _E (E_1)}(b) =-r(E_1)  \label{eqv2}
    \end{equation}
    Let us prove that the equation \ref{111} holds. By definition of $E_1$,  the element $\beta$ is inside $E_1$ and thus $\nu _{\fA (E_1)} (\beta) = \nu _{E_1} (\beta). $ Thus, Lemma \ref{valval} shows that \begin{equation} \label{v3}\nu _{\fA} (\beta) e(E_1 \mid F) = e( \fA \mid \fo _F) \nu _{\fA (E_1)} (\beta). \end{equation} The proposition \cite[1.2.4]{BK} give us the equality  \begin{equation}\label{44} e(\fB _{E_1} \mid \fo _{E_1} ) =\dfrac{e(\fA \mid \fo _F)}{e(E_1 \mid F) }.\end{equation}
Because $[\fA , n , r , \beta ]$ is a simple stratum, $n$ is equal to $- \nu _{\fA} (\beta)$, consequently using equations \ref{v3} and \ref{44}, the following sequence of equality holds.

\begin{align*}
\nu _{\fA (E_1)} (\beta) = \frac{ \nu _{\fA} (\beta) e ( E_1 \mid F )}{e(\fA \mid \mathfrak{o} _F )}= \frac{\nu _{\fA} (\beta)}{e(\fB _{E_1} \mid \mathfrak{o}_{E_1})}=  \frac{-n}{e(\fB _{E_1} \mid \mathfrak{o}_{E_1})}=-n(E_1)
\end{align*}
This concludes the proof of the equality \ref{111} and the equality \ref{eqv2} is easily proven in the same way.
Proposition \cite[1.4.13]{BK}  gives

\begin{center}
$\left\{ \begin{array}{ll}
        k_0(\beta,\fA (E_1))=\dfrac{k_0 ( \beta , \fA )}{e(\fB _{E_1} \mid \fo _{E_1})} \\
        k_0(b , \fB _E (E_1) )= \dfrac{k_0 (b , \fB _E) }{e (\fB _{E_1} \mid \fo _{E_1})}
    \end{array}  
\right.$
.
\end{center}
Consequently $[\fA (E_1) , n (E_1) , r(E_1) , \beta ]$ and $[\fB _E (E_1) , r(E_1) , r(E_1) -1 , b ]$ are simple strata and satisfy the hypothesis of the proposition \cite[2.2.3]{BK}. Consequently $[\fA (E_1) , n , r-1, \beta + b ]$ is simple and the field $F[\beta + b]$ is equal to the field $F[\beta , b ]$.  Moreover, \cite[2.2.3]{BK} implies that \begin{center}

$k_0 (\beta +b , \fA (E_1) ) = \left\{ \begin{array}{ll}
        -r(E_1)=k_0(b,\mathfrak{B}_E(E_1))$ if $b \not \in E \\
        k_0(\beta , \mathfrak{A}(E_1))$ if $b\in E.
    \end{array}  
\right.$

\end{center}
The valuation $\nu _{\fA (E_1)} (\beta +b)$ is equal to $-n(E_1)$ and the same argument as before shows that $\nu _{\fA} (\beta +b ) =-n $. The proposition \cite[1.4.13]{BK} shows that $k_0 (\beta +b , \fA) = k_0 (\beta +b , \fA (E_1) ) e(\fB _{E_1} \mid \fo _{E_1} ) $.
Thus

\begin{center}
$k_0 (\beta +b , \fA  ) = \left\{ \begin{array}{ll}
        -r=k_0(b,\mathfrak{B}_E)$ if $b \not \in E \\
        k_0(\beta , \mathfrak{A})$ if $b\in E
    \end{array}  
\right.$
\end{center}
This completes the proof.

\end{proof}

Given a non-necessary tame, pure stratum $[\fA , n , r, \beta ]$, the existence of a simple stratum $[\fA , n , r , \gamma ]$ equivalent to $[\fA , n , r, \beta ]$ is a fundamental theorem in Bushnell-Kutzko's theory. Given such $[\fA , n , r, \beta ]$ and $[\fA , n , r, \gamma ]$,  there is (in general) no  inclusion between the field $F[\beta]$ and $F[\gamma]$; however, the following arithmetical properties are always true.\begin{center} \begin{equation}
e(F[\gamma] \mid F ) \mid e (F[\beta] \mid F )
\end{equation}\begin{equation}
f(F[\gamma] \mid F ) \mid f (F[\beta] \mid F )\end{equation} 

\end{center} Moreover, if $[\fA , n , r, \beta ]$ is not simple, then the degree $[F[\beta] : F]$ is strictly bigger than $[F[\gamma] : F]$ by \ref{approxi}.
In the tame situation, a new property is always true. Given a tame pure stratum $[\fA , n , r, \beta ]$  such that ${r=-k_0 ( \beta , \fA )}$, there is an equivalent tame simple stratum $[\fA , n , r, \gamma ]$ such that the field $F[\gamma]$ is included in the field $F[\beta]$. We refer to Bushnell-Henniart for the proof of this fact. This property is given in the following proposition:

\begin{prop} \cite[3.1 Corollary]{Esse} \label{tssprop2}Let $[\fA, n , r , \beta ] $ be a tame pure stratum in the algebra $A= \End_F (V)$ such that ${r=-k_0 ( \beta , \fA )}$. There is an element $\gamma$ in the field $F[\beta]$ such that the stratum $[\fA , n , r, \gamma]$ is simple and equivalent to $[\fA , n , r , \beta ]$

\end{prop}

To make an explicit link between Bushnell-Kutzko's and Yu's formalisms, the following proposition is used crucially in Section  \ref{sectgecha}  of this paper.

\begin{prop}\label{tssprop3} Let $[\fA , n , r , \beta ] $ be a tame pure stratum such that
\begin{center} ${r=-k_0 ( \beta , \fA )}$. \end{center} For all elements $\gamma$ in the field $F[\beta]$ such that $[\fA , n , r , \gamma ]$ is a simple stratum equivalent to $[\fA , n , r , \beta ]$, the stratum $[\fB _{\gamma} , r , r-1 , \beta - \gamma ]$ is simple, here $\fB _{\gamma }= \End _{F[\gamma]} (V) \cap \fA $.

\end{prop}

\begin{proof} Using a similar argument as that in Proposition \ref{tssprop1}, it
 is enough to prove the proposition in the case where $F[\beta]$ is a maximal subfield of the algebra $A=\End_F(V)$. 
So let $[\fA , n , r , \beta ]$ be a tame pure stratum such that $F[\beta]$ is a maximal subfield of $A$ and $k_0 ( \beta , \fA) = -r$. Let $\gamma$ be in $F[\beta]$ such that $[\fA , n , r , \gamma ] $ is simple. The stratum $[\fB _{\gamma} , r , r-1 , \beta - \gamma ] $ is pure in the algebra $\End_{F[\gamma]} (V)$  because it is equivalent to a simple one by \cite[2.4.1]{BK}. Moreover ${[\fB _{\gamma} , r , r-1 , \beta - \gamma ] }$
 is tame pure, so Proposition \ref{tssprop2} shows that there exists a simple stratum $[\fB _{\gamma} , r , r-1 , \alpha  ] $
  equivalent to $[\fB _{\gamma} , r , r-1 , \beta - \gamma ] $, such that $F[\gamma][\alpha] \subset F[\gamma][\beta - \gamma ].$ By proposition \ref{tssprop1},
  $[\fA , n , r-1 , \gamma + \alpha ]$ is simple and $F[\gamma + \alpha ]$ is equal to the field $F[\gamma , \alpha ]$. Set $\mathfrak{Q} _{\gamma} = \rad ( \fB _{\gamma} ) = \fB _{\gamma} \cap \fP$. The equivalence $[\fB _{\gamma} , r ,r-1 , \alpha ] \sim [\fB _{\gamma} , r , r-1 , \beta - \gamma ]$ shows that $ \alpha \equiv \beta - \gamma$ $(\mod~ \mathfrak{Q} _{\gamma} ^{-(r-1)})$. This implies $\gamma + \alpha \equiv \beta$ $ (\mod~ \fP ^{-(r-1)}). $ We deduce that $[\fA , n , r -1 , \gamma + \alpha ] $ and $[\fA , n , r-1 , \beta ]$ are two simple strata equivalent. Indeed, the first is simple by construction and the second is simple by hypothesis because $k_0 (\beta , \fA )=-r$. The definitions shows that $F[\gamma + \alpha ] \subset F[\beta]$, and \ref{approxi} shows that $[F[\gamma + \alpha ]:F] =[F[\beta]:F]$. Thus, $F[\gamma + \alpha ] = F[\beta]$. The trivial inclusions $F[\gamma + \alpha] \subset F[\gamma , \alpha] \subset F[\beta]$ then shows that $ F[\gamma + \alpha] = F [\gamma , \alpha ] = F [\beta ]$.
  We have thus obtained that the following three assertions hold:

 - The stratum $[\fB _{\gamma} , r , r-1 , \alpha ]$ is a simple stratum in $\End_{F[\gamma]}(V)$.

 - The field $F[\gamma][\alpha]$ is a maximal subfield of the $F[\gamma]$-algebra $\End_{F[\gamma]}(V)$.

 -$[\fB _{\gamma} , r , r-1 , \alpha ] \sim[\fB _{\gamma} , r ,r-1, \beta - \gamma ]$\\ 
Consequently, by \cite[2.2.2]{BK}, $[\fB _{\gamma} , r ,r-1 , \beta - \gamma ]$ is simple,   as required.

\end{proof}

\section{Tame minimal elements}\label{sectemrs}

  In this section, we study Bushnell-Kutzko's tame minimal elements. We use standard representative elements introduced by Howe\footnote{Howe's construction of supercuspidal representations should be considered as the common ancestor of \cite{BK} and \cite{YU}, and Moy's presentation of Howe's construction was an important hint in our work.} \cite{HO}.  The main result of this section is Proposition \ref{minigene}. The following describes the multiplicative group of a non-Archimedean local field.

 \begin{prop} \cite[Chapter 2 Proposition 5.7]{Neuk} \label{multistru}

 Let $K$ be a non-Archimedean local field and $q=p^f$ the number of elements in the residue field of $K$. Let  
  $\mu _{q-1}$  denote the group of $(q-1)$-th roots of unity in $K$. Let $\pi _ K$ be a uniformiser in $K$. Then, the following hold:

 \begin{enumerate}
 \item[($i$)] If $K$ has characteristic $0$, then one has the following isomorphisms of topological groups
 \begin{center}
 $K^{\times} \simeq \pi _ K ^{\Z} \times \mathfrak{o}_K ^{\times} \simeq \pi _ K^{\Z} \times \mu _{q-1} \times (1+ \mathfrak{p} _K )\simeq \Z \times \Z / (q-1) \Z \times \Z /p^a \Z \times \Z _p ^d $
 \end{center}

 where $a\geq 0$ and $d=[K:\Q _p]$.

 The first three groups are denoted multipticatively and the last group is denoted additively.

 \item[($ii$)] If $K$ has characteristic $p$, then one has the following isomorphisms of topological groups:
 \begin{center}
 $K^{\times} \simeq   \pi _ K ^{\Z} \times \mathfrak{o}_K ^{\times}\simeq \pi _ K ^{\Z} \times \mu _{q-1} \times 1+ \mathfrak{p} _K  \simeq \Z \times \Z /(q-1)\Z \times \Z _p ^{\mathbf{N}}$. \end{center}

 The first three groups are denoted multipticatively and the last group is denoted additively.
 \end{enumerate}
 \end{prop}

  The previous proposition allows us to deduce the following corollary, which is a well known result. Recall that we have a fixed uniformiser $\pi _F$.

\begin{coro} \label{corost} Let $E$ denote a tamely ramified extension of $F$. There exists a uniformiser $\pi _E$ of $E$ and a root of unity $z \in E$, of order prime to $p$, such that $\pi _ E ^e z = \pi _F$.

\end{coro}
 \begin{proof} Let $\pi$ be a uniformiser of $E$. The proposition $\ref{multistru}$ shows that there exist an isomorphism $f:E^{\times} \simeq \pi ^{\Z} \times \mu _{q-1} \times G' $ where $G' = 1+ \mathfrak{p} _E$ is a multiplicatively denoted group. Each element of $G'$ has an $e$-th root. Indeed, the proposition \ref{multistru} shows that $1+ \mathfrak{p}  _E$ is isomorphic to the additive
  group $\Z / p ^a \Z \times \Z _p ^d $ or to the additive group $\Z _p ^{\mathbf{N}}$. The
   image of $\pi _F$ by $f$ is $(e,z,g)$, where ${(e,z,g)\in \pi _E ^{\Z} \times \mu _{q-1} \times G' }$; that is, $\pi_F = \pi  ^e z g $. Let $r$ be in $G'$ such that $r^e =g$. Then $r \pi $ is a uniformiser of $E$ and  $\pi _F =(r \pi )^e z$. So $\pi _E = r \pi $ has the required property.

 \end{proof}

\begin{defi}\label{defice}Let $E/F$  and $\pi_E$ be as in the previous corollary; that is, such that $\pi_F = \pi_E ^e z$ with $z$ a root of unity of order prime to $p$. Let $C_E$ be the group generated by $\pi _E $ and the roots of unity of order prime to $p$ in $E^{\times}$.

\end{defi}

\begin{prop} \label{ceindep}The group $C_E$ is independent of the choice of 
$\pi _E$ used in \ref{defice} to define it.
\end{prop}

\begin{proof} Let $\pi _1$ and $\pi _2$ be two uniformisers of $E$ and $z_1$ , $z_2$ be two roots of unity of order prime to $p$ such
 that ${\pi ^e z_1 = \pi _F} $ and $\pi _2 ^e z_2 = \pi _F $. Let $C^1$ be the 
group generated by $\pi _1$ and the root of unity of order prime to $p$. Let $C^2$ be the group generated by $\pi _2$ and the root of unity
 of order prime to $p$. By symmetry, it is enough show that $C^1 \subset C^2$. It is also enough to show that $\pi _1 \in C_2$. The equation $\pi _1 ^e z_1 = \pi _F $ implies
 that $\pi _1 ^e \in C_2$, thus there exists a root of unity $z$ of order prime to $p$ such that $\pi _1 ^e = \pi _2 ^e z$. We have $(\pi _1 \pi _2 ^{-1} )^e = z $. Let $o_z$ be the order of $z$, which is an integer prime
  to $p$. We have  $(\pi _1 \pi _2 ^{-1} )^{eo_z} = 1 $.  The integer $e o_z$ is prime to $p$, indeed $e=e(E \mid F)$ is prime to $p$ because $E/F$ is a tamely ramified extension and $o_z$ is prime to $p$. Consequently $\pi _1 \pi _2 ^{-1}$ is a root of unity of order prime to $p$. This implies that $\pi _1 \in C_2$, as required.

\end{proof}

 We have fixed a uniformiser $\pi _F$ at the beginning of the text. So to each tamely ramified extension $E/F$, the group $C_E$ is well-defined and does not depend on any choice. 

 \begin{prop} \label{srunik}Let $E/F$ be a tamely ramified extension. Let $c$ be an element in $E^\times$. The following holds. 
 \begin{enumerate}[(i)]
 \item  There exists a unique element $sr(c) \in C_E$, called the standard representative of $c$ and a unique element $x \in 1+\mathfrak{p}_E$
 such that $c=sr(c)\times x $.
 \item The element $sr(c)$ is the unique element in $C_E$, such that $\nu _ E (sr(c)- c) >\nu _E (c) $
 \item The map $sr:E^{\times} \to C_E$, $c \mapsto sr(c)$ is a morphism of groups whose restriction to $C_E$ is the identity map.
 \end{enumerate}
 \end{prop}
\begin{proof}\begin{enumerate}[(i)]
\item The proposition \ref{multistru} shows that $E^{\times} \simeq C_E \times ( 1+ \mathfrak{p} _E )$ and $(i)$ is a consequence.
\item The element $sr(c)$ is the unique element in $C_E$ such that $c=sr(c) \times (1+ y)$, with $y \in \mathfrak{p}_E$. Thus, $sr(c)$ is the unique element in $C_E$ such that $c- sr(c) \in sr(c) \mathfrak{p} _E$. Thus, $(ii)$ holds, remarking that $sr(c)$ and $c$ have the same valuation.

\item This is obvious because by definition the map $sr$ is the projection map $E^{\times}\simeq C_E \times 1+ \mathfrak{p}_E \to C_E$.
\end{enumerate}
\end{proof}

 \begin{prop}\label{CE} Let $E'/E/F$ be a tower of finite tamely ramified extensions.  Let $s \in C_E$. The following assertions hold.\begin{enumerate}
 \item[($i$)] The group $C_E $ is included in the group $ C_{E'} $.
 \item[($ii$)]If $E/F$ is a Galois extension, then $C_E$ is stable under the Galois action of $\mathrm{Gal}(E/F)$ on $E$. Moreover, if $\sigma _ 1 $ and $ \sigma _2 $ are elements in $ \mathrm{Gal}(E/F)$  such that $\sigma_1 (s) \not = \sigma _2 (s)$, then \begin{center}$\nu_E (\sigma _ 1(s) - \sigma _2(s))= \nu _E (s)$ .\end{center}

\item[($iii$)] Even if $E/F$ is not necessarily a Galois extension, then for any pair of morphisms of $F$-algebras from $E $ to $\overline{F}$ such that $\sigma _1 (s) \neq \sigma _2 (s) $, we have 
 \begin{center}$\ord (\sigma _ 1(s) - \sigma _2(s))= \ord (s)$ .\end{center}

 \item[($iv$)] Let $c$ be an arbitrary element in $E^{\times}$. Let $\sigma $ be a morphism of $F$-algebra from $E$ to $\overline{F}$. Then $\sigma (sr(c))=sr(\sigma (c))$.
   \end{enumerate}
  \end{prop}
  \begin{sloppypar}
 \begin{proof}
 ($i$) Recall that the group $C_E$ and $C_{E'}$ are independent of the choices of uniformisers used to define them by \ref{ceindep}. Let $\pi _E$ be a uniformiser of $E$ and $z$ a root of unity of order prime to $p$ in $E$ such that $\pi _E ^{e(E\mid F)} z = \pi _F$. Because $E'/E$ is tamely ramified, there exists a uniformiser $\pi _{E'} \in E'$ and a root of unity $w$ of order prime to $p$ in $E'$ such that $\pi _{E'} ^{e(E' \mid E )} w = \pi _E$. Elevating to the power $e(E\mid F)$, we have $\pi _{E'} ^{e(E' \mid E ) e(E \mid F)} w ^{e(E\mid F)}= \pi _E ^{e(E\mid F)}$. We thus get $\pi _{E'} ^{e(E\mid F)} w ^{e(E\mid F)} z  = \pi _F$. The element $w ^{e(E\mid F)} z $ is a root of unity of order prime to $p$.
 Consequently, $C_{E'}$ is the group generated by $\pi _{E'}$ and the roots of unity of order prime to $p$ in $E'$. The equation $\pi _{E'} ^{e(E' \mid E )} w = \pi _E$ shows that $\pi  _E$ is inside $C_{E'}$. Trivially, the roots of unity of order prime to $p$ in $E$ are inside the roots of unity of order prime to $p$ in $E'$. Consequently $C_E$ is inside $C_{E'}$, as required.

   $(ii)$ Let $\sigma \in Gal(E/F)$, and let $\pi _E$ be an element such that $\pi _E ^e z = \pi _F$ for $z $ a root of unity in $E$ of order prime to $p$. Let $o_z$ be the order of $z$. It is enough to show that $z$ and $\pi _E$ are mapped in $C_E$ by $\sigma$. The equality $(\sigma (z))^{o_z}=1 $ shows that $\sigma (z) $ is a root of unity of order prime to $p$ and thus inside $C_E$. The equality $\sigma (\pi _E )^e \sigma (z) = \pi _F$ together with \ref{ceindep} show that we can use $\sigma (\pi _E )$ to define $C_E$, and thus $\sigma  (\pi _E )$ is inside $C_E$.
 This proves the first part of the assertion.   
      The element $\sigma _1(s)$ is in $C_E$, so $sr(\sigma _1 (s))=\sigma _1 (s)$.
      Consequently 
      $\nu_E (\sigma _ 1(s) - \sigma _2(s))= \nu _E (\sigma_1(s))$,
       indeed assume ${\nu _E (\sigma _1 (s) -\sigma _2 (s) )\not = \nu _E (\sigma _1 (s) )}$, then $\nu _E (\sigma _1 (s) -\sigma _2 (s) )> \sigma _1 (s) $, and so $\sigma _2 (s) = sr(\sigma _1 (s)) = \sigma _1 (s)$ by \ref{srunik}, this is a contradiction. This completes the second part of the assertion and the proof of the proposition.

     $(iii)$ Let $\overline{E}$ be the Galois closure of $E$. Then, $C_{E} \subset C_{\overline{E}}$ by the first assertion. Moreover $\sigma _1$ and $\sigma _2$ extend to $\overline{E}$ and $(iii)$ now follows from $(ii)$.

     $(iv)$ We have $\ord(sr(c)-c)>\ord(c)$ and so $\ord (\sigma (sr(c))- \sigma (c) ) > \ord ( \sigma (c))$ because $\sigma$ preserves $\ord$; this implies that $\sigma (sr(c))= sr(\sigma (c))$.
 \end{proof}
 \end{sloppypar}

 We need to remark an elementary lemma  to prove Proposition \ref{minigene}, which is the main result of this section.

 \begin{lemm} \label{reduuni} Let $E/F$ be a finite unramified extension. Let $z \in E$ be a root of unity of order prime to $p$. Then, $z$ generates $E/F$ if and only if $z+ \mathfrak{p} _E$ generates the residual field extension $k_E /k_F$.
 \end{lemm} 

 \begin{proof} If $z$ generates $E$ over $F$, then $z$ generates $\mathfrak{o} _E $ over $\mathfrak{o}_F$ by \cite[7.12]{Neuk}. Thus, $z$ generates the residual field extension $k_E /k_F $.  Let us check the reverse implication.  Assume that $z+ \mathfrak{p}_E$ generates $k_E /k_F $. The field extension $E/F$ is unramified, so  $[k_E : k_F] = [E:F]$. Let $P_z \in F[X]$ be the minimal polynomial of $z$ and $d$ its degree, clearly $P_z $ is in $\mathfrak{o}_F[X]$. It is enough to show that $d=[E:F]$.  We have ${d \leq [E:F]}$. The reduction $\mod \mathfrak{p}_E$ of $P_z$ is of degree $d$ and annihilates $z+ \mathfrak{p} _E$, a generator of $k_E / k_F$, and thus $[k_E : k_F] \leq d $. So ${[k_E : k_F] \leq d \leq [E:F]}$. So $d =[E:F]$, and this concludes the proof.

 \end{proof}

 \begin{lemm}Let $E/F$ be a finite tamely ramified extension. Let $c,c' \in C_E$ such that $c\neq c'$ and $\ord (c) = \ord( c')$. Then \[ \ord (c - c' ) = \ord (c) . \]

 \end{lemm}

 \begin{proof} This is an obvious consequence of definitions. Indeed, assume that $ \ord (c - c' ) \neq  \ord (c) $, then $\ord (c - c' ) > \ord (c)$ and by definition of standard representative we have $sr(c)=c'$. Because $c \in C_E$, we have $sr(c)=c$. So $c=c'$ and it is a contradiction. So $\ord (c - c' ) = \ord (c)$.

 \end{proof}

  \begin{prop}\label{minigene}
Let $E/F$ be a finite tamely ramified extension, let $\beta$ be an element in $ E$ such that $E=F[\beta]$. Put $-r=\ord (\beta)$. The following assertions are equivalent.

\begin{enumerate}

\item[($i$)] The element $\beta $ is minimal over $F$.

\item[($ii$)] The standard representative element of $\beta$ generates the field extension $E/F$; that is,  $F[sr(\beta)]=E$.

\item[($iii$)] For all morphisms of $F$-algebra $\sigma  \neq \sigma ' $ from $E$ to $\overline{F}$, we have \[\ord (\sigma ( \beta ) - \sigma ' (\beta)  ) =-r .\]

 \end{enumerate}

 \end{prop}

\begin{proof}

Let us prove that ($i$) implies ($ii$). Assume that $\beta $ is minimal over $F$. Let us remark that the definition of $sr(\beta)$ implies trivially that $F[sr(\beta)] \subset E$. Let $E^{\nr}$ denote the maximal unramified extension contained in $E$. To prove the opposite inclusion $ E \subset F[sr(\beta)]$, it is enough to show that  ${E^{\nr}\subset F[sr(\beta)]}$ and ${E\subset E^{\nr}[sr(\beta)]}$. Put $\nu=\nu_E(\beta)$, $ e =e(E \mid F)$. The valuation of  $\pi_F^{-\nu}\beta^e$ is equal to $0$, consequently by \ref{srunik} we have  $\nu_E(sr(\pi_F^{-\nu}\beta^e)-\pi_F^{-\nu}\beta ^e)>0$, and so ${sr(\pi_F^{-\nu}\beta^e)+\mathfrak{p}_E=\pi_F^{-\nu}\beta ^e+\mathfrak{p}_E.}$ We have $sr(\pi_F^{-\nu}\beta ^e)=\pi_F^{-\nu}sr(\beta)^e$, and this is a root of unity of order prime to $p$. The definition of being minimal implies that $\pi_F^{-\nu}sr(\beta)^e +\mathfrak{p}_E$ generates $k_E / k_F$. So $\pi_F^{-\nu}sr(\beta)^e$ generates $E^{\nr}$ by \ref{reduuni}. So $E^{\nr}\subset F[sr(\beta)]$.
We have $\nu_E(\beta)=\nu_E(sr(\beta)),$ so $\mathrm{gcd}(\nu_E(sr(\beta)),e)=1.$ Let $a $ and $b$  be integers such that $a\nu_E(sr(\beta))+be=1$. Thus, $\nu_E(sr(\beta)^a \pi_F ^b)=1$ and so $E^{\nr}[sr(\beta)^a \pi_F ^b]=E$ because a finite totally ramified extension is generated by an arbitrary uniformiser. So $E^{\nr}[sr(\beta)]=E$ and ($i$) hold. We have thus shown that ${E^{\nr}\subset F[sr(\beta)]}$ and ${E\subset E^{\nr}[sr(\beta)]}$ and so  ($i$) implies ($ii$).

Let us prove that $(ii)$ implies $(i)$. Assume that $F[sr(\beta)]=E$. We start by showing that $e$ is prime to $\nu$. The field $E ^{\nr}$ is generated over $F$ by the roots of unity of order
 prime to $p$ contained in $E$. Let $d=\gcd (\nu , e)$ and $b=\frac{e}{d}$. Let $\pi_E$ be a uniformiser in $E$ such that $\pi _E ^{e} z = \pi _F$ with $z$ a root of unity of order prime to $p$. The element $sr(\beta ) $ is in $C_E$ and so $sr(\beta) = \pi _E ^{\nu} w$ with $w$ a root of unity of order prime to $p$ in $E$. The equalities 
${sr(\beta)^b = (\pi _E ^e ) ^{\frac{\nu}{gcd(\nu , e)} }w^b= (\pi _F z ^{-1})^{\frac{\nu}{gcd(\nu , e)} }w^b}$ show that $sr(\beta) ^b$ is contained in $E^{\nr}$. By hypothesis, the element $sr(\beta)$ generates
 $E$ over $F$ and so generates $E$ over $E^{\nr}$. Consequently, the field $E$ is generated
 by an element whose $b$-th power is in $E^{\nr}$. Therefore, the inequality  $[E:E^{\nr}]\leq b$ holds. The extension $E^{\nr}$ is the maximal unramified extension contained in $E$, so $[E:E^{\nr}]=e$. Thus, the inequality  $e \leq b \leq \frac{e}{d}$ holds. This implies $d=1$ and so $\nu$ is prime to $e$.
 Let us prove that $\pi _F ^{- \nu } \beta ^e + \mathfrak{p} _E $ generates the residue field extension $k _E $ over $k _F$. Because $\pi _F ^{- \nu} \beta ^e + \mathfrak{p} _E = \pi _F ^{- \nu} sr(\beta)^e + \mathfrak{p} _E$, it is equivalent to show that $x+\mathfrak{p} _E$ generates $k_E$ over $k _F$ , where $x=\pi _F ^{-\nu} sr (\beta) ^e. $ The element $sr(\beta)$ generates $E$ over $F$ by hypothesis; that is, $E=F[sr(\beta)]$.  So the inequality $[E : F[x]]\leq e$ holds, indeed $E $ is generated over $F[x]$ by the element $sr(\beta)$ whose $e$-th power is in $F[x]$. Because  $x$ is a root of unity of order prime to $p$, the field $F[x]$ is include in $E^{\nr}$, so $[E : E^{\nr}] \leq [E : F[x]] $. Consequently, the identity $e=[E : E^{\nr}] \leq [E : F[x]] \leq e $ holds. Because $F[x] \subset E^{\nr}$, the previous identity implies that $F[x]=E^{\nr}$. Thus by \ref{reduuni} the element $x +\mathfrak{p} _E$ generates $k_E$ over $k_F$. So $\beta$ is minimal over $F$.

 Let us prove that $(ii)$ is equivalent to $(iii)$. Let $\sigma \neq \sigma '$ be two morphisms of $F$-algebra from $E$ to $\overline{F}$. Put \begin{align*}
 &A= \sigma (\beta) - \sigma ' (\beta) \\
 & B = \sigma (sr(\beta)) - \sigma ' ( sr(\beta)).
 \end{align*}
 We have $\ord (A) \geq -r $ and $\ord (B) \geq -r$. We have \begin{align*}
 \ord (A-B)& = \ord ( ~~\sigma (\beta) - \sigma ' (\beta) - (~\sigma (sr(\beta))- \sigma '(sr(\beta))  ~)  ~ ~  ) \\
 &= \ord ( ~~ \sigma (\beta) - \sigma (sr(\beta)) - (~ \sigma ' (\beta) -\sigma ' (sr(\beta))~) ~~)\\
  &= \ord ( ~~ \sigma (\beta) -  sr(\sigma(\beta)) - (~ \sigma ' (\beta) - sr(\sigma '(\beta))~) ~~)\\
 &>-r ~~~~
 \end{align*}
 because $ \ord (~\sigma (\beta) - sr(\sigma(\beta))~) >-r $ and $ \ord (~\sigma '(\beta) - sr(\sigma '(\beta))~) >-r$,  by definition of standard representatives.
 Let us prove $(ii) \Rightarrow (iii)$. If $(ii)$ holds, then $\sigma ( sr(\beta) )\neq \sigma ' ( sr(\beta))$ and $\ord ( B ) = \ord ( \sigma (sr (\beta )) - \sigma ' ( sr (\beta)) ) =-r $ because  $\sigma (sr (\beta ))$ and $\sigma ' ( sr (\beta))$ are both in $C_{\overline{E}}$. So $\ord(A) \geq -r , \ord (B) =-r , \ord ( A-B) >-r$, this implies $\ord (A)=-r$. Let us prove $(iii) \Rightarrow (ii)$. We assume that $(iii)$ holds and we want to prove that $sr(\beta)$ generates $E/F$. It is enough to show that $\sigma (sr(\beta)) \neq \sigma ' (sr(\beta))$ for any $\sigma , \sigma '$ as before. 
 By hypothesis $\ord (A) =-r$; because $\ord (B) \geq -r$ and $\ord (A-B) >-r$, we deduce $\ord (B) =-r$, in particular $\sigma (sr(\beta)) \neq \sigma ' (sr(\beta))$ as required.
 This finishes the proof of the proposition \ref{minigene}.

 \end{proof}

\section{Generic elements and tame minimal elements} \label{geneassomini}

In this section, we study the relationship between (tame) minimal elements of \cite{BK} and the generic element of \cite{YU}.  More precisely, let ${E'/E/F}$ be a tower of tamely ramified field extensions and let $V$ be an $E'$-vector space of dimension $d$. We are going to explicitly define and describe the group schemes ${H' = \Res_{E'/F} \underline{\Aut} _{E'} (V)}$, ${H=  \Res_{E/F} \underline{\Aut} _{E} (V)}$ and $G= \underline{\Aut} _F (V)$. We will show that the sequence ${(H',H,G)}$ forms a
  tamely ramified twisted Levi sequence in $G$. The choice of an $E'$-maximal decomposition $D$, $V=(V_1 \oplus \ldots \oplus V_{d})$, of $V$ in $1$-dimensional $E'$-vector spaces gives birth to a maximal torus $T_D$ of $\underline{\Aut} _{E'} (V)$. By restriction of scalar, we get
  a maximal torus $T=\Res_{E'/E} (T_D)$ of $H'$. We are going to describe the set over $\overline{F}$ of roots of $H'$ and $H$ with respect to $T$. Moreover, we will describe the condition \textbf{GE1} in this situation. Finally, given $c\in E'$, we will introduce an element $X_{c}^* \in \Lie^*(Z(H'))$. This section is devoted to prove the following theorem.

  \begin{theo} 
  Let $E'/E/F, V, H' , H $ and $G$ as before.
  \begin{enumerate}
  \item The sequence $H' \subset H \subset G $ is a tamely ramified twisted Levi sequence in $G$.
  \item Let $c \in E'$. Let $r$ be $-\ord(c)$, $A$ be $\End_F(V)$ and $Z(H')$ be the center of $H'$. Let $X_{c}^*$ be   $ \left(x \mapsto Tr_{A/F}(cx) \right) \in \Lie ^* (Z(H'))_{-r}$.
   The following assertions are equivalent. \begin{enumerate}
\item The element $X_{c}^*$ is $H$-generic of depth $r$.
\item The element $c$ is minimal\footnote{see \ref{defminimal}} relatively to the extension $E'/E$.\end{enumerate}

  \end{enumerate}
  \end{theo}

  \begin{proof} This is Corollary \ref{corotwistedlevitame} and Theorem \ref{minimalgenericelemenx}.

  \end{proof}

\subsection{The group scheme of automorphisms of a free $A$-module of finite rank} \label{autam}
Let $\mathrm{A}$ be a commutative ring and $\mM$ be a free $\mathrm{A}$-module of rank $r$. The functor 
\begin{center}
\begin{align*}
 \{ \mathrm{A}-\mathrm{algebra} \} & \to  \mathbf{Gp} \\
 \mB~~~~~~ & \mapsto  \Aut_{\mB} (\mM \otimes _{A} B) \\
\end{align*}
\end{center}
is representable by an affine $\mA$-scheme that we denote by $\underline{\Aut}_{\mA} (\mM)$. This scheme is isomorphic to the group scheme $\mathrm{GL}_N$ over $\mA$, with $N=r$. 
Let $D$ be a decomposition ${M=\mM_1 \oplus \ldots \oplus \mM_{r}}$ of $\mM$ in submodule of rank $1$. Let us define a maximal split torus of $\underline{\Aut}_{\mA} (M)$. 
The functor 
\begin{center}

\begin{small}
\begin{align*}
 &\{ \mathrm{A}-algebra \}  \to~~~~~~  \mathbf{Gp} \\
 & \mB~~ \mapsto \left\{ \begin{small}   x \in \Aut _{\mB} (\mM \otimes _{\mA} \mB ) \left\|\begin{aligned}&  \text{For all} ~i \in \{1, \ldots , r  \}, \text{ there exists }~ \lambda  ^i(x) \in B^{\times} \\ & \text{such that}~ 
  x (v_i \otimes 1) =  \lambda  ^i(x)  (v_i \otimes 1) \text{ for all }v_i \in M_i \end{aligned} \right.\end{small} \right\}
\end{align*}

\end{small}
\end{center}
is representable by an affine $\mA$-scheme that we denote by $T_D$, this is a closed affine subscheme of $\underline{\Aut}_{{\mA}} (\mM)$. The $\mA$-scheme $T_D$ is canonically isomorphic to ${\displaystyle \prod _{i=1} ^{r} \underline{\Aut} _{\mA} (M_i)}$. Let us give an explicit expression of the set of roots $\Phi (\underline{\Aut}_{\mA} (\mM), T_D )$ in this functorial point of view. The notation $ 0 \leq i \not = i' \leq r$ means that $1\leq i \leq r$, $1 \leq i' \leq r$ and that $i \not = i' $.
The set of root of $\underline{\Aut}_{\mA }(\mM)$ relative to $T_D$ is the set
\begin{center}
$\Phi (\underline{\Aut}_{\mA} (\mM), T_D )= \{ \alpha _{ii'} \mid 1 \leq i \not = i' \leq r \}$
\end{center}
where $\alpha _{ii'}$ is the morphism of algebraic group $T_D \to \mathbb{G}_m$ characterised by the formula,

\begin{center}
 for all $\mathrm{A}$-algebras $\mB$, for all $x \in T_D( \mB)$ , $\alpha _{ii'} (x) = \lambda ^i (x) (\lambda ^{i'} (x))^{-1}$.
\end{center}
For each root $\alpha$, let $\alpha ^{\vee}: \mathbb{G}_m \to T_D$ be the coroot of $\alpha$ and let $\mathrm{d}\alpha ^{\vee}$ be the derivative of $\alpha ^{\vee}$. Finally, let $H_{\alpha}$ be the element $\mathrm{d}\alpha ^{\vee} (1) \in \Lie (T_D) (A)$.
Let us make these objects explicit in our functorial point of view.
Let $1 \leq i \not = i' \leq r$, the coroot $\alpha _{ii'} ^{\vee}$ is the morphism of algebraic group $\mathbb{G}_m \to T_D$ characterised by the formula,

\begin{small}

 for all $\mathrm{A}$-algebras $\mB$, for all $\lambda \in B^{\times}$ ,\begin{small} $\left\{\begin{aligned}& \alpha _{ii'}^{\vee} (\lambda) (v_i \otimes 1) = \lambda  (v_i \otimes 1)~~~\forall v_i \in M_i \\
  &\alpha _{ii'}^{\vee} (\lambda) (v_{i'} \otimes 1) = \lambda^{-1}  (v_{i'} \otimes 1) ~~~\forall v_{i'} \in M_{i'} \\
  &\alpha _{ii'}^{\vee} (\lambda) (v_k \otimes 1) =   (v_{k} \otimes 1) ~\forall v_k \in M_k, k \not = i,i' \end{aligned} \right.$  \end{small}

\end{small}
The derivative of $\alpha _{ii'} ^{\vee}$ is the differential morphism $\mathrm{d}\alpha _{ii'} ^{\vee} : \mathbb{G}_a \to \underline{\Lie} (T_D)$, which is characterised by the formula (see \cite[3.9.4]{SGA31}),

\begin{small}
for all $\mathrm{A}$-algebra $\mB$, for all $h \in B$ ,\begin{small} $\left\{\begin{aligned}& \mathrm{d}\alpha _{ii'}^{\vee} (h) (v_i \otimes 1) = h  (v_i \otimes 1)~~~\forall v_i \in M_i\\
  &\mathrm{d}\alpha _{ii'}^{\vee} (h) (v_{i'} \otimes 1) = -h  (v_{i'} \otimes 1) ~~~\forall v_{i'} \in M_{i'}\\
  &\mathrm{d}\alpha _{ii'}^{\vee} (h) (v_k \otimes 1) =  0  ~~~\forall v_k \in M_k~\forall k \not = i,i'. \end{aligned} \right.$  \end{small}

\end{small}
\begin{sloppypar} Consequently the element $H_{\alpha _{ii'}} $ which is by definition ${\mathrm{d}\alpha _{ii'} ^{\vee}(1) \in \underline{\Lie} (T_D)(A)= \End_A(M)}$ is the element sending each element $v_i \in M_i $ to $v_i$, each element $v_{i'} \in M_{i'}$ to $-v_{i'}$ and, for all $k$ different of $i,i'$, each element $v_k \in M_k$ to $0$.
\end{sloppypar}

 \subsection{Twisted Levi sequences} \label{gllevi}
In this subsection, we state or recall algebraic facts that will be applied to the following subsections.
 Let $f$ be a commutative ring and $B$ be a commutative $f$-algebra, $C$ be an $B$-algebra. Let $A$ be an $f$-algebra. In this situation $A \otimes _f B$ is an $B$-algebra and $C$ is naturally an $f$-algebra. The $C$-algebra $(A \otimes _f B) \otimes _B C $ is canonically isomorphic to $A \otimes _f C$. Explicitly, the isomorphism is given by

 \begin{align*}
 (A \otimes _f B ) \otimes _B C &\to A \otimes _f C \\
 (a \otimes b ) \otimes c &\mapsto a \otimes bc.
 \end{align*}
 The inverse is explicitly given by 
 \begin{align*}
 A \otimes _f C &\to ( A\otimes _f B ) \otimes _B C \\
 a \otimes c &\mapsto ( a \otimes 1 ) \otimes c. 
 \end{align*}
 We now fix in the rest of this subsection a tower of finite separable extensions of fields  $l'/l/f$. In the next subsection, we will apply this to $l'=E', l=E$ and $f=F$, where $E'/E/F$ is a tower of finite tamely ramified extensions. Let $V$ be an $l'$-vector space of dimension $d$. Let $D$:$~$ $V=(D_1 \oplus \ldots \oplus D_{d) })$, be an $l'$-decomposition of $V$ in subspaces of dimension $1$. In a previous subsection we introduced an $l'$-group scheme $\underline{\Aut}_{l'} (V)$ and a maximal split torus $T_{D}$ of $\underline{\Aut}_{l'} (V)$.
  Let $H'$ be the restriction of scalar from $l' $ to $f$ of $\underline{\Aut}_{l'} (V)$. Also, let $T$ be $\Res_{l'/f}(T_D)$.
    Thus, $H'$ represents the functor 
 \begin{align*}
 \{f-algebra\}  & \to ~~~~\mathbf{Gp}\\
 A~~~~~~ & \mapsto  \underline{\Aut}_{l'} (V) ( A \otimes _f l').
 \end{align*}
 For each $f$-algebra $A$ the group $H'(A)$ is thus equal to the group \begin{center}$\Aut_{A \otimes _f l' } \left(V \otimes _{l'} (A \otimes _f  l') \right)$.\end{center}
 Because  $l \subset l'$, $V$ is an $l$-space and, we have a group $\underline{\Aut}_{l} (V)$ and its restriction of scalar $H$.  So that for each $f$-algebra $A$, the group $H(A)$ is equal to the group $\Aut_{A \otimes _f l } \left(V \otimes _{l} (A \otimes _f  l) \right)$. 
 Also let $G $ be $\underline{\Aut}_{f} (V)$.
 For each $f$-algebra $A$, the canonical morphism $A \otimes _f l \to A \otimes _f l'$ induces a canonical morphism of groups \begin{center} $\Aut_{A \otimes _f l' } \left(V \otimes _{l'} (A \otimes _f l') \right) \to \Aut_{A \otimes _f l } \left(V \otimes _{l} (A \otimes _f l) \right)$, \end{center} which is functorial in $A$.
 We thus get a canonical morphism of $f$-group scheme $H' \to H$. This morphism is a closed immersion. We also have a canonical morphism of $F$-group schemes $H \to G$.
 We are interested in Condition \textbf{GE1}, which is related to the extension of scalar from $f$ to $\overline{f}$, the algebraic closure of $f$. So let us compute \begin{center} $T \times _{f }  \overline{f} $,  $H' \times _{f} \overline{f} $ and  $H \times _{f} \overline{f} $.\end{center}
 Let $A$ be an $\overline{f}$-algebra, by definition $(H \times _{f} \overline{f} ) (A) = H (A)$. We have seen that this is equal to $\Aut_{A \otimes _f l} \left(V \otimes _{l} (A \otimes _f l) \right)$. We need to study the algebra $A \otimes _f l$. 
 We know that there exists $\sigma _{1} , \ldots , \sigma _{i} , \ldots , \sigma _{[l:f]}$, distinct morphisms of $f$-algebra from $l$ to the Galois closure of $l$. We also know that for $1 \leq i \leq [l:f]$, there exists $[l':l]$ morphisms of $f$-algebra from $l'$ to the Galois closure of $l'$ extending $\sigma_{i}$, which we denote as  $\sigma _{i1} , \ldots ,\sigma _{ij} , \ldots , \sigma _{i[l':l]}$. We write ${\displaystyle \prod _i }$ instead of ${\displaystyle \prod _{i=1}^{[l:f]}}$ and ${\displaystyle {\bigoplus _{i,j}  }}$ instead of ${\displaystyle {\bigoplus _{i=1}^{[l:f]} \bigoplus_{j=1}^{[l':l] }}}$, we use other similar notations.

 \begin{prop} \label{splitfieldgene} Let $K'$ be the Galois closure of $l'$ and let $K$ be the Galois closure of $l$.  The following assertions hold:

 \begin{enumerate}[(i)]
  \item Let $A$ be a $K$-algebra. Then, $A \otimes _{f} l$ is canonically isomorphic to ${\displaystyle \prod _{i} A_{i}}$,  where $A_{i}=A$ for each $i$. Moreover, this isomorphism is explicitly given as follows:

\centerline{
\xymatrix@R=0.0em{ A \otimes _f {l} \ar[r]^{\sim} & {\displaystyle \prod_{i} A_i } \\
  a \otimes e \ar@{|->}[r] & {\displaystyle \prod_{i} a\sigma _i(e) }}  }

 \item Let $A$ be a $K'$-algebra. Then, $A \otimes _{f} l'$ is canonically isomorphic to ${\displaystyle \prod _{i,j} A_{ij}}$, where $A_{ij}=A$ for each $i,j$.
  Moreover, this isomorphism is explicitly given as follows:

\centerline{
\xymatrix@R=0.0em{ A \otimes _f {l'} \ar[r]^{\sim} & {\displaystyle \prod_{i,j} A_{ij} } \\
  a \otimes e \ar@{|->}[r] & {\displaystyle \prod_{i,j} a\sigma _{ij}(e) }}  }

 Moreover, the $A$-algebra $A \otimes _f l '$ is  canonically an $A \otimes _f l$-algebra. The ring ${\displaystyle \prod _{i,j} A_{ij}}$ is canonically an  ${\displaystyle \prod _{i} A_{i}}$-algebra and the structure is given by 
 \begin{center}
 ${\displaystyle( \prod _i \lambda _i ).( \prod _{i,j} a_{ij} )= \prod _{i,j} \lambda _i a_{ij}}$.
 \end{center}

 \item The previous assertions are functorial in $A$.

 \end{enumerate}

 \end{prop}

 \begin{proof} This is elementary algebra.

 \end{proof}
 Let $\mA$ be a commutative ring, let $\mA _1$ and $\mA _2$ be two commutative $\mA$-algebras. Let $\mB_1$ be an $\mA_1$-algebra and let $\mB _2$ be an $\mA _2$-algebra. Let $\mM$ be a free $\mA$-module of rank $r$. There are  canonical isomorphism of groups \begin{equation} \label{autabautaautb}
{\Aut}_{\mA _1 \times \mA_2 } ( \mM \otimes _{A} (\mB_1 \times \mB_2))  \simeq {\Aut}_{\mA _1  } ( \mM \otimes _{A} \mB_1 ) \times {\Aut}_{\mA _2  } ( \mM \otimes _{A} \mB_2 )  .
\end{equation}
More generally, the following hold:

\begin{lemm} \label{autaiiiiautaiautaiautai} Let $\mathrm{A}$ be a commutative ring, let $\mA _i$ , $0 \leq i \leq d$, be some commutative $\mA$-algebras. For $0 \leq i \leq d$, let $\mB _i$ be an $\mA _i$-algebra. Let $M$ be a free $\mA$-module of finite rank. Then, we have a canonical isomorphism of groups

\begin{center}
$ \Aut_{\prod _{i=1}^d \mA _i }( \mM \otimes _{\mA} \prod _{i=1}^d \mB_i ) \simeq \displaystyle \prod _{i=1}  ^d \Aut _{\mA_i} (\mM \otimes _{\mA} {\mB _i} )$. \end{center}

\end{lemm}
 Let $i,j,k$ be integers as above and let $C/K'$ be a field extension ($K'$ is the Galois closure of $l'$). We put $V_{ij} = V \otimes _{l'} C_{ij}$ ($C_{ij}=C$ is defined as in Proposition \ref{splitfieldgene}). We put also $D_{ijk} = D_k \otimes _{l'} C_{ij}$.

 \begin{sloppypar}
 \begin{prop} \label{grousplit}  With the previously introduced notations, the following assertions hold: \begin{enumerate}[(i)]
 \item 
 There is a canonical commutative diagram  of $f$-schemes
 \begin{center}
 \begin{scriptsize}
 \xymatrix{ T \times _{f}  C \ar[r]^-{h_1} \ar[d]^-{v_1}& H ' \times _{f} C \ar[r]^-{h_2} \ar[d]^-{v_2} & H \times _{f} C  \ar[r]^-{h_3} \ar[d]^-{v_3} &G \times _{f}  C \ar[d]^-{v_4}\\
{\displaystyle \prod_{i,j,k} \underline{ \Aut} _{C} (D_{ijk}) }\ar[r]^-{f_1} &  \displaystyle \prod_{i,j}  \underline{ \Aut} _{C} \displaystyle( V_{ij}) \ar[r]^-{f_2}& \displaystyle \prod_{i} \underline{ \Aut} _{C} (\displaystyle \bigoplus _j V_{ij})  \ar[r]^-{f_3} & \underline{ \Aut} _{C} \displaystyle (\bigoplus _{i,j} V_{ij}) .}

 \end{scriptsize}
 \end{center}

 \item There is a canonical commutative diagram of $k$-spaces
 \begin{flushleft}
 \begin{scriptsize}

 \xymatrix{  
 \Lie (T ) \ar[r] \ar[d]& \Lie ( H ' ) \ar[r] \ar[d] & \Lie ( H )  \ar[r] \ar[d] & \Lie (G ) \ar[d]\\
  \Lie (T \times _{f} C) \ar[r] \ar[d]^{\simeq}& \Lie ( H ' \times _{f} C )\ar[r] \ar[d]^{\simeq} & \Lie ( H \times _{f} C)  \ar[r] \ar[d]^{\simeq} & \Lie (G \times _{f} C) \ar[d]^{\simeq} \\
{\displaystyle \prod_{i,j,k} \End _{C} (D_{ijk}) }\ar[r] &  \displaystyle \prod_{i,j}  \End _{C} \displaystyle(V_{ij})\ar[r]& \displaystyle \prod_{i} \End _{C} (\displaystyle \bigoplus _j V_{ij})  \ar[r] & \End _{C} (\displaystyle \bigoplus _{i,j}  V_{ij}) .}

 \end{scriptsize}
 \end{flushleft}

 \item Let $s $ be an element in $l'$. Let $m_s$ be the element of $\Lie (T)$ which send an element $h$ to $sh$. Let  $m_{s , C}$ be the element in $ \displaystyle \prod_{i,j}  \End _{C} \displaystyle(V_{ij}) \subset \End _{C} ( \displaystyle \bigoplus _{i,j} V_{ij})$ characterised by the formula,
 \begin{center}
 for all $i,j$, for all $v \in V_{ij}$, $m_{s , C}(v )=\sigma _{ij} (s) v$.
 \end{center}

 Then, the image of $m_s$ in $\End _{C} ( \displaystyle \bigoplus _{i,j} V_{ij})$ through the diagram introduced in $(ii)$ is $m_{s,C}$.
 \end{enumerate}

 \end{prop}
 \end{sloppypar}

\begin{proof} \begin{enumerate}[(i)] \item The upper horizontal line is induced by the previously introduced morphisms $ T \to H' \to H \to G $. We thus get some maps $h_1,h_2$ and $h_3$. Let $A$ be a $C$-algebra. In the rest of this proof, we still denote $h_1 (A)$ by $h_1$, and we do the same for $h_2$ and $h_3$.  We have
\begin{align*}
\left(T \times _{f} C \right) (A) \simeq & \left( \Res _{l'/f} \displaystyle \prod _k \underline{\Aut} _{l'} (D_k) \right) (A)\\
{\footnotesize \text{By properties of $\Res$}}~~~~\simeq & \left( \displaystyle \prod _k  \Res_{l'/f} \underline{\Aut } _{l'} (D_k)\right) (A)\\
\simeq &   \displaystyle \prod _k \left( \Res_{l'/f} \underline{\Aut } _{l'} (D_k) (A) \right)\\
\begin{footnotesize}
\text{ By definition of $\Res$}
\end{footnotesize}~~~~\simeq & \displaystyle \prod _k  \underline{\Aut } _{l'} (D_k) (A\otimes _{f} l') \\
\begin{footnotesize}
\text{ By definition of $\underline{\Aut}$}
\end{footnotesize}~~~~\simeq & \displaystyle \prod _k   {\Aut } _{A\otimes _{f} l'} (D_k \otimes _{l'} (A\otimes _{f} l'))  \\
\begin{footnotesize}\text{ By proposition \ref{splitfieldgene}}
\end{footnotesize}~~~~\simeq & \displaystyle \prod _k   {\Aut } _{\prod _i \prod _j A_{ij}} (D_k \otimes _{l'} (\prod _{i,j} A_{ij}))  \\
\begin{footnotesize}\text{ By proposition \ref{autaiiiiautaiautaiautai}}
\end{footnotesize}~~~~\simeq & \displaystyle \prod _{k,i,j}  {\Aut } _{A_{ij}} (D_k \otimes _{l'}  A_{ij})  \\
\simeq & \displaystyle   \prod _{i,j,k}  {\Aut } _{A_{ij}} (D_k \otimes _{l'}  A_{ij})  \\
\simeq & \displaystyle   \prod _{i,j,k}  {\Aut } _{A_{ij}} (D_{ijk})  \\
\end{align*}
We thus get an isomorphism \begin{center} $\left(T \times _{f} C \right) (A) \to \displaystyle   \prod _{i,j,k}  {\Aut } _{A_{ij}} (D_{ijk}) $,\end{center} let us denote it as $v_1$.
 We have
\begin{align*}
\left(H' \times _{f} C \right) (A) \simeq & \left( \Res _{l'/f}  \underline{\Aut} _{l'} (V) \right) (A)\\
\simeq & \underline{\Aut}_{l'} (V) (A \otimes _{f} l' ) \\
\simeq & \Aut _{A \otimes _{f} l' } ( V \otimes _{l'} ( A \otimes _{f} l' ))\\
\simeq & \Aut _{\prod _{i,j} A_{ij}} ( V \otimes _f ( \displaystyle \prod_{i,j} A_{ij} ))\\
\simeq & \displaystyle \prod _{i,j} \Aut _{A_{ij}} (V \otimes A_{ij})\\
\simeq & \displaystyle  \prod _{i,j} \Aut _{A_{ij}} (V_{ij}).
\end{align*}
We thus get an isomorphism \begin{center}$\left(H' \times _{F} C \right) (A) \to  \displaystyle  \prod _{i,j} \Aut _{A_{ij}} (V_{ij})$,\end{center} let us denote it as $v_2$.
We have
\begin{align*}
\left(H \times _{f} C \right) (A) \simeq & \left( \Res _{l/f}  \underline{\Aut} _{l} (V) \right) (A)\\
\simeq & \underline{\Aut}_{l} (V) (A \otimes _{f} l ) \\
\simeq & \Aut _{A \otimes _{f} l } ( V \otimes _l ( A \otimes _{f} l )).\\
\end{align*}
As an $A \otimes _f l$-module,  $ V \otimes _l ( A \otimes _{f} l )$ is isomorphic to $V \otimes _{l'} ( A \otimes _f l' )$. So
\begin{align*}
\left(H \times _{f} C \right) (A) \simeq & \Aut _{A \otimes _{f} l } ( V \otimes _{l'} ( A \otimes _{f} l' ))\\
\simeq & \Aut _{\prod _i A_i}  ( V \otimes _{l'} (\prod _{i,j} A_{ij} ))\\
\begin{footnotesize}\text{ By proposition \ref{autaiiiiautaiautaiautai}}
\end{footnotesize}~~~~\simeq & \prod _{i} \Aut _{A_i} ( V \otimes _{l'} ( \prod _j A_{ij} )\\
\simeq & \prod _i \Aut _{A_i} ( \bigoplus _j V \otimes _{l'} A_{ij} )\\
\simeq & \prod _i \Aut _{A_i} (\bigoplus _j V_{ij}).\\
\end{align*}
We thus get an isomorphism  \begin{center}$\left(H \times _{f} C\right) (A) \to  \prod _i \Aut _{A_i} (\bigoplus _j V_{ij})$,\end{center} let us denote it as $v_3$.
We have 
\begin{align*} 
\left( G \times _{f} C  \right) (A) \simeq &( \underline{\Aut } _{f} (V) )(A) \\
\simeq & \Aut _A ( V \otimes _f A) \\
\simeq & \Aut _A ( V \otimes _{l'} ( l' \otimes _f A) )\\
\simeq & \Aut _A ( V \otimes _{l'} ( \displaystyle \prod _{i,j} A_{ij} )) \\
\simeq & \Aut _A ( \displaystyle \bigoplus _{i,j} V \otimes _{l'} A_{ij} )\\
\simeq & \Aut _A ( \displaystyle \bigoplus _{i,j} V_{ij})
\end{align*}
 We thus get an isomorphism \begin{center}$\left( G \times _{f} C  \right) (A) \simeq ( \underline{\Aut } _{f} (V) )(A) \to \Aut _A ( \displaystyle \bigoplus _i V_{ij})$,\end{center} let us denote it as $v_4$. 
Let us recall that for all $i,j$, $V_{ij} = \displaystyle \bigoplus _k D _{ijk}$. In the following, $v_{ijk}$ denotes an arbitrary vector in $D_{ijk}$, and $v_{ij}$ denotes an arbitrary vector in $V_{ij}$.
 Let $f_1$ be the canonical morphism  
 \begin{center} $ {\displaystyle   \prod _{i,j,k}  {\Aut } _{A_{ij}} (D_{ijk}) \to  \displaystyle \prod _{i,j} \Aut _{A_{ij}} (\bigoplus _k D_{ijk})}$ \end{center} sending  $\prod _{i,j,k} ( L_{ijk}) $ to $  \prod _{i,j} \left(\sum _k v_{ijk} \mapsto \sum _k L_{ijk} (v_{ijk}) \right)$. It is a formal computation to verify that the morphism $v_2 \circ h_1 $ is equal to $f_1  \circ v_1$. 
 Let $f_2$ be the canonical morphism
 \begin{center} $ \displaystyle \prod _{i,j} \Aut _{A_{ij}} (V_{ij}) \to   \prod _i \Aut _{A_i} (\bigoplus _j V_{ij})$  \end{center} sending $\prod _{i,j} L_{ij} $ to $\prod _i \left( \sum _j  v_{ij} \mapsto \prod _i  \sum _j L_{ij} ( v_{ij}) \right).$ It is a formal computation to verify that the morphism $v_3 \circ h_2 $ is equal to $f_2 \circ v_2$. 
 Let $f_3$ be the canonical morphism 
 \begin{center} $\displaystyle  \prod _i \Aut _{A_i} (\bigoplus _j V_{ij}) \to \Aut _A ( \displaystyle \bigoplus _{i,j} v_{ij})$\end{center} sending $\prod _i L_i $ to  $ \left( \sum _{i,j}  v_{ij} \mapsto \sum _i L_i ( \sum _j  v_{ij} )\right)$. It is a formal computation to verify that $v_4 \circ h_3 $ is equal to $f_3 \circ v_3$.
 The previous isomorphisms are functorial in $A$ and form a canonical diagram, and thus induce the required diagram at the level of $C$-algebraic groups.
 This concludes the proof of $(i)$

 \item This is a consequence of $(i)$ taking the Lie algebra of all objects.

 \item The image of $m_s$ in $\Lie (G) = \End _{f} (V)$ is the map sending $v$ to $sv$. The map $\Lie (G) \to \Lie ( G \times _{f} C)$ is the map\begin{center} ${\End _f (V) \to \End _{C} ( V \otimes _f C)}$ \end{center} sending a $f$-linear map $L$ to the $C$-linear map $\left( v \otimes \lambda \mapsto L(v) \otimes \lambda \right)$ so the image of $m_s$ in $\Lie (G \times _{F} \overline{F}$ is the map $(v \otimes \lambda \mapsto sv \otimes \lambda )$, let still denote it $m_s$. Consider the diagram of $C$-linear maps

  \centerline{ \xymatrix{V \otimes _f C \ar[r]^-{m_s} \ar[d]^-c  & V \otimes _f C \ar[d]^-c \\
 V \otimes _{l'} ( l' \otimes _f C) \ar[d]^-i & V \otimes _{l'} ( l' \otimes _{f} C_{ij})\ar[d]^-i \\
 { V \otimes \displaystyle \prod _{i,j} C_{ij}} \ar[d]^-b &  { V \otimes \displaystyle \prod _{i,j} C_{ij}} \ar[d]^-b\\
 \displaystyle \bigoplus _{i,j} V_{ij} & \displaystyle \bigoplus _{i,j} V_{ij}} }
 \begin{sloppypar}
 where $c$ is the canonical map, $i$ is the map induced by the map introduced in proposition \ref{splitfieldgene}, and $b$ is the canonical map induced by the definition of $V_{ij}$. The image of $m_s$ in $\End _{\overline{F}} \left( \displaystyle \bigoplus _{i,j} V_{ij} \right)$ is the composition ${b \circ i \circ c \circ m_s \circ c^{-1} \circ i^{-1} \circ b ^{-1}}$.
Let us show that it is equal to $m_{s,C}$. The equality ${b \circ i \circ c \circ m_s \circ c^{-1} \circ i^{-1} \circ b ^{-1}}= m_{s,C}$ is equivalent to the equality ${b \circ i \circ c \circ m_s =m_{s,C} \circ b \circ i \circ c}$. Let us prove this last equality by calculation. Let $v \otimes \lambda \in V \otimes _{F} C$, we have 
\end{sloppypar}

\begin{align*}b \circ i \circ c \circ m_s (v \otimes \lambda ) = &b \circ i \circ c ( sv \otimes \lambda )\\
=& b \circ i ( sv \otimes ( 1 \otimes \lambda ))\\
=& b \circ i ( v \otimes (s \otimes \lambda ))\\
=& b ( v \otimes \prod _{i,j} \sigma _{ij} (s) \lambda )\\
=& \displaystyle \sum _{i,j} v \otimes \sigma_{ij}(s) \lambda 
\end{align*}

and 
\begin{align*}
m_{s,C} \circ b \circ i \circ c (v \otimes \lambda)= &m_{s,C} \circ b \circ i (v \otimes ( 1 \otimes \lambda) ) \\
=& m_{s,C} \circ b ( v \otimes \displaystyle \prod _{i,j} \lambda )\\
=& m_{s,C} (\sum _{i,j } v \otimes \lambda )\\
= &\sum _{i,j} v \otimes \sigma _{ij}(s) \lambda.
\end{align*}
This ends the proof of $(iii)$.

  \end{enumerate}

 \end{proof}

 \begin{sloppypar}
 So, the torus $T \times _{f} C $ is a maximal split torus of ${H' \times _{f} C}$, ${ H \times _{f} C}$ and ${ G \times _{f}C}$. Moreover, ${H' \times _{f}C}$ is a Levi subgroup of ${ H \times _{f} C} $, and $ {H \times _{f}C} $ is a Levi subgroup of $ {G \times _{ f} C}$. We thus have inclusions of the corresponding sets of roots. \end{sloppypar}\[ \Phi ( H' , T , C) \subset \Phi ( H , T , C) \subset \Phi (G , T , C)\] Let us identify, using \ref{grousplit}, \begin{align*}
&T \times _{f} C &~~~~~~\text{with}~~~~~~~~~~~& \displaystyle \prod_{i,j,k} \underline{ \Aut} _{C} (D_{ijk}),\\ 
 & H'\times _{f} C &~~~~~~\text{with}~~~~~~~~~~~& \displaystyle \prod_{i,j}  \underline{ \Aut} _{C} \displaystyle( V_{ij}),\\
 &H\times _{f} C &~~~~~~\text{with}~~~~~~~~~~~& \displaystyle \prod_{i}   \underline{ \Aut} _{C} \displaystyle(\bigoplus _j V_{ij}),~ \text{and}\\
  &G\times _{f}C&~~~~~~\text{with}~~~~~~~~~~~& \underline{ \Aut} _{C} ( \displaystyle \bigoplus _{i,j} V_{ij}) .\\
\end{align*}   Because  $\displaystyle \bigoplus _{i,j} V_{ij}$ is equal to $ \displaystyle \bigoplus _{i,j,k} D_{ijk} , $ we can apply \ref{autam} to describe the set of roots $\Phi (G , T , C)$. Putting \begin{align*}
I&=\{1, \ldots, i ,\ldots ,[l:f]\}\\
 J&=\{1,\ldots ,j,\ldots , [l':l]\} \\
 K&=\{1,\ldots ,k,\ldots ,d\}, 
 \end{align*} we obtain the following equality.\[
\Phi (G , T , C)= \{ \alpha _{ijk , i'j'k'} \mid (i,j,k),(i',j',k') \in  (I \times J \times K), (i,j,k) \not = (i',j',k') \}\]
The set of roots $\Phi (H , T , C)$ is the following subset of $\Phi (G , T , C)$\[\Phi (H , T , C)= \{ \alpha _{ijk , i'j'k'} \in \Phi (G , T , C) \mid i=i'  \}\]
The set of roots $\Phi (H' , T ,C)$ is the following subset of $\Phi (G , T , C)$\[\Phi (H' , T , C)= \{ \alpha _{ijk , i'j'k'} \in \Phi (G , T , C) \mid i=i' ~\text{and}~ j=j'  \}.\]
  The condition \textbf{GE1} is relative to the set $\Phi (H , T ,C) \setminus \Phi (H ', T , C)$. The following is a description of this set\[ \Phi (H , T ,C) \setminus \Phi (H' , T , C)= \{ \alpha _{ijk , i'j'k'} \in \Phi (G , T , C) \mid i=i' ~\text{and}~ j\not=j'  \}.\]
 \begin{sloppypar}
 The condition \textbf{GE1} involves the element $H_{\alpha} $ for $\alpha$ in ${ \Phi (H , T , C) \setminus \Phi (H' , T , C)}$.
 Let us recall the description given in \ref{autam}. Let $\alpha _{ijk , i'j'k'} \in \Phi (G , T , C),$ the element $H_{\alpha}$, which is by definition  ${\mathrm{d}\alpha _{ijk,i'j'k'} ^{\vee}(1) }$ is the element sending each element $v \in D_{ijk} $ to $v$, and sending each element $v \in D_{i'j'k'}$ to $-v$ and, for all $i''j''k''$ different of $ijk,i'j'k'$, sending each element $v \in D_{i''j''k''} $ to $0$.\end{sloppypar}

  \subsection{Tame twisted Levi sequences}
  \begin{sloppypar}
  Let $E'/E/F$ be a tower of finite tamely ramified extensions. Let $V$ be an $E'$-vector space of dimension $d$ and $D$ be a decomposition ${V=(D_1 \oplus \ldots \oplus D_k \oplus \ldots \oplus D_d)}$ of $V$ in one dimensional $E'$-vector spaces. 
  In the previous subsection, we introduced $H'= \Res _{E'/F} \underline{\Aut} _{E'}(V)$,  $H= \Res _{E/F} \underline{\Aut} _{E}(V)$, and $G = \underline{\Aut}_{F} (V)$. We also associated a torus $T= \Res _{E'/F} (T_D) $ to the decomposition $D$.
  In proposition \ref{grousplit}, we computed the extension of scalar of these $F$-groups scheme to an extension containing the Galois closure of $E'$. We deduce the following corollary.\end{sloppypar}

  \begin{coro} \label{corotwistedlevitame} The sequence $H' \subset H \subset G$ is a tamely ramified twisted Levi sequence in $G$, moreover $Z(H')/Z(G)$ is anisotropic.
  \end{coro}

  \begin{proof} We have to verify that the definition given at the beginning of section \ref{yu} is satisfied. We need to show that there exists a finite tamely ramified Galois extension $L$ of $F$ such that $H' \times _{F} L$ and  $H' \times _{F} L$ are Levi subgroups of $G \times _{F} L$. This is a direct consequence of Proposition \ref{grousplit}.  
Then, the isomorphism of topological groups
$ \left(Z(H') /Z(G) \right) (F)\simeq E'^{\times} /F^{\times}$ holds.
 The explicit description of the topological multiplicative group of a non-Archimedean local field given in Proposition \ref{multistru} implies that  $E'^{\times} /F^{\times}$ is compact. This implies that $Z(H') /Z(G)$ is anisotropic.
 This ends the proof.

  \end{proof}

 \subsection{Generic elements and minimal elements} \label{genericassotomini}

In this subsection, we use the notations of the previous subsection.
The center $Z'$ of $H'$ is isomorphic to $\Res _{E'/F} ( \mathbb{G}_m)$, thus it is connected; that is, $Z'^{\circ} = Z'$.
The inclusions $Z' \to H' \to H \to G $ induces a canonical diagram

\centerline{ \xymatrix{ \Lie ( Z')\ar[d] \ar[r] & \Lie (G) \ar[d] \\
 \Lie ( Z') \otimes _F \overline{F}  \ar@{}[d]|*=0[@]{\cong} \ar[r]& \Lie (G)\otimes _F \overline{F} \ar@{}[d]|*=0[@]{\cong} \\
 \Lie ( Z' \times _{F} \overline{F}) \ar[r] &\Lie ( G \times _{F} \overline{F}). }}
 As explained after Definition \ref{realizedef}, we have canonical inclusions
 \[{\Lie^*(Z') \to \Lie^*(H') \to \Lie^*(H) \to \Lie ^* (G)},\]
  inducing
   a canonical inclusion $\Lie ^* (Z') \to \Lie^* (G)$ and a canonical commutative diagram

 \centerline{ \xymatrix{\Lie ^* ( Z') \ar[r]\ar[d] & \Lie ^* (G)\ar[d] \\
 \Lie ^* (Z') \otimes _F \overline{F} \ar[r] \ar@{}[d]|*=0[@]{\cong} & \Lie ^* (G ) \otimes _F \overline{F} \ar@{}[d]|*=0[@]{\cong} \\
 \Lie ^* ( Z' \times _{F } \overline{F} )\ar[r]& \Lie ^* (G \times _{F} \overline{F}). }}

 Recall that an element $X ^* \in \Lie ^* (Z')$ is $H$-generic of depth $r$ if and only if $X^ * \in \Lie ^* (Z') _{-r} $ and if Conditions \textbf{GE1} and \textbf{GE2} hold. Because  $H'$ and $H$ are of type A, Condition \textbf{GE1} implies Condition \textbf{GE2} by \ref{ge1implyge2tors}. Given $X^* \in \Lie ^* (Z')$ we denote by $X^*_{\overline{F}}$ the image of $X^*$ in $\Lie ^* (Z' \times _{ F} \overline{F}) $ via the previous commutative diagram. Let recall that Condition \textbf{GE1} holds for $X^*$ if $X^* _{\overline{F}} (H _{\alpha}) = -r $ for all root $\alpha \in \Phi ( H , T , \overline{F}) \setminus \Phi ( H', T, \overline{F}) $.

 \begin{defi}\label{xy} Let $c \in E'$. Let $X^* _{c}$ be the element in $\Lie ^* (Z')$ sending an element $h \in Lie (Z') $ to $\Tr _{\End _F (V) / F } (m_c \circ i (h) )$ where $i$ is the map ${Lie (Z') \to \Lie (G)}$, and $m_c \in End_F (V)$ is the map sending $v \in V $ to $cv$; that is,  $m_c$ is the multiplication by $c$.
 \end{defi}

 \begin{prop}\label{xycomput} \begin{sloppypar}Let $c\in E'$. Let $X_c ^* \in Lie^* (Z')$ be the element introduced in definition \ref{xy}. Let $X^* _{c, \overline{F}}$ be the corresponding element in ${\Lie ^* (Z' \times _{ F} \overline{F})}$. Then
 \end{sloppypar}
 \begin{enumerate}[(i)]
 \item  \begin{sloppypar}$X^* _{ c , \overline{F}} ( H _{\alpha _{i_1 j_1 k_1 , i_2 j_2 k_2}}) = \sigma _{i_1j_1}(c) - \sigma _{i_2 j_2}(c) $ for all roots ${{\alpha _{i_1 j_1 k_1 , i_2 j_2 k_2}} \in \Phi (G , T , \overline{F}) }.$ \end{sloppypar}
 \item The element $X^*_c$ is in $Lie^*(Z') _{-r}$ where $r=-\ord(c)$.

 \end{enumerate}
 \end{prop}

 \begin{proof}\begin{enumerate}[(i)]
 \item
   Consider the diagram 

 \centerline{\xymatrix{\Lie (Z')\ar[d]^{Id\otimes 1} \ar[r]^-i& \End_F (V) \ar[d]^{Id\otimes 1} \ar[r]^-{m_c \circ } &\End_F (V) \ar[d]^{Id\otimes 1} \ar[r]^-{\Tr _{F }}  \ar[d]^{Id\otimes 1}  &F\ar[d]^{Id\otimes 1} \\
\Lie (Z') \otimes _F \overline{F} \ar[d]^{g} \ar[r]^-{i\otimes Id } & \End_F (V) \otimes _F \overline{F}\ar[d]^{f}  \ar[r]^-{ m_c \circ \otimes Id }  &\End_F (V) \otimes _F \overline{F}\ar[d]^{f} \ar[r]^-{\Tr _{F} \otimes Id }&\overline{F} \ar[d]^{Id} \\ \Lie (Z' \times _{F} \overline{F} \ar[r]^-{i_{\overline{F}}} &\End_{\overline{F}}(V \otimes _F \overline{F}) \ar[r]^-{ m_{c ,\overline{F}}\circ} &\End_{\overline{F}} (V \otimes _F \overline{F}) \ar[r]^-{\Tr _{\overline{F}}} &\overline{F} }}

where $i$ is the canonical inclusion, $ m_c \circ$ is the composition by $m_c$, and $ m_{c , \overline{F}} \circ $ is the composition by the image $m_{c, \overline{F}}$ of $m_c$ in $\End_{\overline{F}} ( V \otimes _F \overline{F})$.
Let us prove that it is commutative. The left-hand part of the diagram was introduced previously and is the canonical diagram induced by $Z' \to G$. The upper-middle and upper-right squares are trivially commutative. The right-hand lower square is commutative by compatibility of traces. Let us prove that the middle-lower square is commutative. Let $L \otimes \lambda \in \End _F (V) \otimes _F \overline{F}$, then 
\begin{align*}
\left( (m_{c, \overline{F}} \circ ) \circ f ) \right) ( L \otimes \lambda ) = & ( m_{c, \overline{F}} \circ )\left( v \otimes \lambda ' \mapsto L(v) \otimes \lambda \lambda ' \right) \\
= & \left( v \otimes \lambda ' \mapsto c L(v) \otimes \lambda \lambda ' \right) 
\end{align*}
 and 
 \begin{align*}
 (f \circ (m_c \otimes Id )) ( L \otimes \lambda ) &= m_c \circ L \otimes \lambda \\
 &= (v \otimes \lambda ' \mapsto cL(v) \otimes \lambda \lambda '.
 \end{align*}
This concludes the proof of the commutativity of the diagram. By definition, we have
\begin{center}$X^* _c = \Tr _F \circ ( m_c \circ ) \circ i $ \end{center} and \begin{center} $X^* _{c , \overline{F}} = \left((\Tr _F \circ  (m_c \circ ) \circ i ) \otimes Id \right) \circ g^{-1}$. \end{center} We thus get
\begin{center}
$X^* _{c , \overline{F}} = (\Tr _F \otimes Id )  \circ  ((m_c \circ ) \otimes Id) \circ (i\otimes Id )  \circ g^{-1}$.
\end{center} 
Thus, the  commutativity of the previous diagram implies 
 \begin{center}
$X^* _{c , \overline{F}} = \Tr _{\overline{F}}  \circ  (m_{c,\overline{F}} \circ ) \circ i_{\overline{F}}.$
\end{center}
Consequently, for all roots $\alpha  \in \Phi (G, T , \overline{F}) $, we have 
\begin{equation}\label{kli}
X^* _{c , \overline{F}}(H_{\alpha}) = \Tr _{\overline{F}}   (m_{c,\overline{F}} \circ   H _{\alpha})
\end{equation}
We have already computed $m_{c,\overline{F}}$ and $H_{\alpha }$ in terms of the decomposition $V \otimes _{F} \overline{F} = \displaystyle \bigoplus _{i,j,k} D_{ijk}$. Let us recall this. By proposition \ref{grousplit}, $m_{c,\overline{F}}$ is the map \begin{align*}m_{c,\overline{F}} :  \bigoplus _{i,j,k} D_{ijk}& \to  \bigoplus _{i,j,k}  D_{ijk} \\
\displaystyle \sum _{i,j,k} v_{ijk} & \mapsto \displaystyle \sum _{i,j,k} \sigma _{ij}(s) v_{ijk}\end{align*}
Let $\alpha _{i_1 j_1 k_1 ,i_2 j_2 k_2 } \in  \Phi (G , T , \overline{F}) $.
By the calculation done at the end of the subsection \ref{gllevi}, $H_{ \alpha _{i_1 j_1 k_1 ,i_2 j_2 k_2 } } $ is the map \begin{align*}H_{\alpha _{i_1 j_1 k_1 ,i_2 j_2 k_2 }} :  \bigoplus _{i,j,k}  D_{ijk}& \to  \bigoplus _{i,j,k}  D_{ijk} \\
\displaystyle \sum _{i,j,k} v_{ijk} & \mapsto v_{i_1 j_1 k_1 } - v_{i_2 j_2 k_2} .\end{align*}
Consequently, the maps $m_{c , \overline{F}} \circ H_{\alpha _{i_1 j_1 k_1 ,i_2 j_2 k_2 } } $ is the map
 \begin{align*}m_{c , \overline{F}} \circ H_{\alpha _{i_1 j_1 k_1 ,i_2 j_2 k_2 } }  :  \bigoplus _{i,j,k}  D_{ijk}& \to  \bigoplus _{i,j,k}  D_{ijk} \\
\displaystyle \sum _{i,j,k} v_{ijk} & \mapsto \sigma _{i_1j_1} (c) v_{i_1 j_1 k_1 } - \sigma _{i_2 j_2  } (c) v_{i_2 j_2 k_2} .\end{align*}
This implies that \begin{equation} \label{rhj} Tr _{\overline{F}} (m_{c , \overline{F}} \circ H_{\alpha _{i_1 j_1 k_1 ,i_2 j_2 k_2 } } ) = \sigma _{i_1 j_1 } (c) - \sigma _{i_2 j_2} (c). \end{equation}
The proposition is now a consequence of \eqref{kli} and \eqref{rhj}.
\item   By definition (see the notation at the beginning of the document) \begin{center}$\Lie^*(Z') _{-r}= \{ X \in \Lie ^* (Z') \mid X (\Lie((Z')_{r+} )\subset \mathfrak{p}_F \}$. \end{center}
We have $c\Lie(Z')_{r+}= \Lie(Z')_{0+}$ and thus ${\Tr_{End_F(V)/F} (c \Lie (Z') ) \subset \mathfrak{p}_F }$. So $X^*_c \in \Lie ^*( Z') _{-r}$.
\end{enumerate}

\end{proof}

\begin{theo} \label{minimalgenericelemenx} Let $c \in E'$. Let 
$r $ be $-\ord (c)$. The following assertions are equivalent \begin{enumerate}
\item The element $X_{c}^*$ is $H$-generic of depth $r$.
\item The element $c$ is minimal relatively to the (tame) extension $E'/E$ (see \ref{defminimal}), in particular $E[c]=E'$.\end{enumerate}
\end{theo}

\begin{proof}
Assume that the element $X_{c}^*$ is $H$-generic of depth $r$.
Let $\sigma '  \neq \sigma $ be a pair of $E$-morphisms $E'\to \overline{F}$. By Proposition \ref{xycomput} and the explicit description of $\Phi (H , T , \overline{F}) \setminus \Phi (H' , T , \overline{F})$ given previously in this section, there exists $\alpha \in\Phi (H , T , \overline{F}) \setminus \Phi (H' , T , \overline{F})$  such that 
$X^*_c (H_{\alpha}) = \sigma (c) - \sigma' (c)$. So $\ord(\sigma (c) - \sigma' (c))=-r$ and so $c$ is minimal by Proposition \ref{minigene}. Reciprocally assume that $c$ is minimal relatively to $E'/E$ and let $\alpha \in \Phi (H , T , \overline{F}) \setminus \Phi (H' , T , \overline{F})$. By Proposition \ref{xycomput}, there exists a pair of $F$-morphisms $\sigma \neq \sigma ' : E' \to \overline{F}$ whose restrictions to $E$ are  equal such that $X^*_c (H_{\alpha})= \sigma (c) - \sigma ' (c)$. We want to prove that $\ord( \sigma (c) - \sigma ' (c) )=-r$. Composing with $\sigma ^{-1}$, we can assume $\sigma , \sigma ' $ are the identity map on $E$. Now by Proposition \ref{minigene}, $\ord (X^*_c (H_{\alpha}) )=-r$ and $GE_1$ is valid. Given that  $GE_1 \Rightarrow GE_2$ here, this ends the proof.

\end{proof}

\section{Tame simple characters} \label{tameca}

Let $[\fA , n , r , \beta ]$ be a tame simple stratum. In this section, we choose and fix a defining sequence $\{[\fA , n , r _i  , \beta _i ]$, $0 \leq i \leq s \}$ and a  simple character $\theta \in \mathcal{C} (\fA, 0 ,\beta)$ , we show that $\theta = {\displaystyle \prod_{i=0}^s \theta ^i}$ where $\theta ^i$ satisfies some conditions. We then introduce two cases depending on the condition that $\beta _s \in F $ or $\beta _ s \not \in F$.

\subsection{Factorisations of tame simple characters }\label{abstfact}

Fix a tame simple stratum $[\fA , n , r , \beta ]$ in the algebra $A=\End _F (V)$. Propositions \ref{tssprop2} and \ref{tssprop3} allow us to choose a defining sequence $\{[\fA , n , r _i  , \beta _i ]$, $0 \leq i \leq s \}$ (see corollary \ref{suiteap}) such that, putting $\fB _{\beta_i}  := \fA \cap \End_{F[\beta _i] }(V)$, $r_0=0$, $\beta _0 = \beta $, the following holds: \begin{enumerate}
\item[(vii)] $F[\beta _{i+1}] \subsetneq F[\beta _i ]$ for $0 \leq i \leq s-1$  

\item [(vi')] The stratum $[\fB _{\beta _{i+1}} , r_{i+1} , r_{i+1} -1 , \beta _i - \beta _{i+1} ]$ is simple in the algebra $\End_{F[\beta _{i+1} ]} (V)$ for $0 \leq i \leq s-1 $.

\end{enumerate}

We fix such a defining sequence in the rest of this section \ref{tameca}, which includes the following subsection \ref{explifact}.
The elements $\beta _i $ , $ 0 \leq i \leq s$ are all included in $F[\beta]$. Put $E_i := F[\beta _i ] $ for $0 \leq i \leq s$.
Let us define elements $c_i$ , $ 0 \leq i \leq s$ , thanks to the following formulas.

$\bullet~ c_i = \beta _i - \beta _{i+1} ~\text{if} ~0 \leq i \leq s-1$

$\bullet~ c_s = \beta _s$.\\The following Theorem is the factorisation of tame simple characters, as announced earlier.

\begin{theo} \label{facttam} Let $\theta \in \mathcal{C} (\fA, m ,\beta)$ be a simple character. There exist smooth characters $\phi _0 , \ldots , \phi _s $ of $E_0^{\times},\ldots ,  E_s ^{\times}$ such that: 

\begin{center} ${\theta = \displaystyle \prod_{i=0}^s \theta ^i}$\end{center} where $\theta ^i, 0 \leq i \leq s$, is the character defined by the following conditions.
\begin{enumerate}[(i)] 
\item $ \theta ^i \mid _{H^{m+1} (\beta , \fA ) \cap B _{\beta _i }} = \phi _i \circ \det _{B _{\beta _i}}$
\item $ \theta ^i \mid _{H^{m_i+1 } (\beta , \fA)} = \psi _{c_i}$ where $m_i  = \max \{[\frac{-\nu _{\fA} (c_i)}{2}] , m \}$.

\end{enumerate}
\end{theo}

\begin{proof}

Let us prove the proposition by induction. Suppose first that $s=0$; that is, that $\beta$ is minimal over $F$. Put $\theta ^0 = \theta $. Then, the condition $(i)$ is trivially satisfied thanks to the definition of a simple character in the minimal case (see \cite[3.2.1]{BK} or \ref{carsimpl} ). The integer $s$ is equal to $0$, thus $\beta = \beta _0 = c_0 $. So $- \nu _{\fA (c_0)} = - \nu _{\fA} (\beta ) =n$. By the definition of simple characters in the minimal case, the restriction $\theta \mid  _{H^{m+1}(\beta , \fA ) \cap U^{[\frac{n}{2}]+1} (\fA)}$ is equal to $\psi _{\beta}$. So it is enough to verify that $H^{m+1}(\beta , \fA ) \cap U^{[\frac{n}{2}]+1 }(\fA)= H^{m_0 ' +1 }(\beta , \fA) $ where $m_0 ' = \max \{ [\frac{n}{2}] , m \}$ which is a consequence of the definition of $H^{m+1}(\beta , \fA )$.
Suppose now that $s>0$. Let us remark that $-k_0 (\beta , \fA)=-\nu _{\fA} (c_0)$, indeed the
 stratum $[\fB _{\beta _1} , -k_0(\beta , \fA ) , -k_0 (\beta , \fA )-1 , \beta_0 - \beta _1 ]$ is simple.  Thus, the definition of simple characters implies that $\theta \mid _{H^{m_0+1}(\beta , \fA) }=\theta ' \psi _{c_0}$ where $\theta ' \in \mathcal{C}(\fA , m_0 , \beta _1)$. Thanks to the induction hypothesis there exists
  characters $\phi _1 , \ldots \phi _s$ of $E_1 ^{\times }, \ldots , E_s ^{\times }$ such that  
 ${\theta ' =\displaystyle \prod_{i=1}^s{\theta '} ^{i}}$ where the ${\theta'} ^i$ are the characters defined by the following conditions.
 \begin{enumerate}[(i')]
\item ${\theta ' }^{i} \mid _{H^{m_0+1} (\beta , \fA )\cap B_{\beta _i}} = \phi _i \circ \det _{A_i} \mid _{H^{m'+1} (\beta , \fA )\cap B_{\beta _i}} $
\item\begin{flushleft}
 ${ \theta '}^i \mid _{H^{m_i+1 } (\beta , \fA)} = \psi _{c_i}$ 
\end{flushleft}
\end{enumerate}
Identity $(ii')$ is a consequence of the induction hypothesis $(ii)$ and the fact that $\max ( [\frac{-\nu_{\fA} (c_i)}{2}] , m_0 ) = \max ([\frac{-\nu_{\fA} (c_i)}{2}], [\frac{-\nu_{\fA} (c_0)}{2}], m ) =m_i$, because $-\nu_{\fA} (c_0) < {-\nu_{\fA} (c_i)}$.
For $1 \leq i \leq s$, the character ${\theta '} ^{i}$ is defined on $H^{m_0 +1} (\beta , \fA ) $ and we can extend ${\theta '}^{i}$ to $H^{m+1} (\beta , \fA) $ thanks to the character $\phi _i$ as follows. The group $H^{m+1} (\beta , \fA)$ is equal to $ U^{m+1}(\fB _{\beta_0}) H^{m_0 +1} (\beta , \fA ) $, we extend ${\theta '}^i$ to a function
 $\theta ^i $ of $H^{m+1} (\beta , \fA)$ by putting $\theta ^i (x) = \phi _i \circ \det _{A_0} (x)$ for $x \in U^{m+1}(\fB _{\beta_{0}})$. The function $\theta ^i$ is a character. The character $\theta ^i $ satisfies the required conditions $(i)$ and $(ii)$ by construction.
 Finally, put $\theta ^0 = \theta \times \displaystyle \prod_{i=1}^s (\theta ^i)^{-1}$. The restriction $\theta ^0 $ to ${H^{m+1}(\beta , \fA ) \cap B_{\beta _i} }$ is equal to the product of the restriction of $\theta$ to ${H^{m+1}(\beta , \fA ) \cap B_{\beta _i} }$ by the restriction of $\theta _i ^{-1}$ for $1 \leq i \leq s$. Let us show that
  each factor factors through $\det_{B_{\beta _0}}$. By definition of a simple character, this is the case for $\theta ^0$. Let $1 \leq i \leq s, $ because of ${H^{m+1}(\beta , \fA ) \cap B_{\beta _i} } \subset {H^{m+1}(\beta , \fA ) \cap B_{\beta _0} }$,  the restriction of $\theta ^i $ to ${H^{m+1}(\beta , \fA ) \cap B_{\beta _i} }$ is equal to  $\phi _i \circ \det _{B _{\beta _i}}\mid _{H^{m+1}(\beta , \fA ) \cap B_{\beta _i} }$. 
 However, a basic fact of algebraic number theory shows that $\det _{B _{\beta _i} } \mid _{B_{\beta _0}} = \det _{B_{\beta _0 }} \circ N_{E_0 / E_i}$, where$ N_{E_0 / E_i}$ is the norm map. Thus each factor factors through $\det _{B_{\beta _0}}$. Consequently, there exists a smooth character $\phi _0$ of $E_0 ^{\times}$ such that the condition $(i)$ is satisfied. Let us prove that $(ii)$ holds.

 \begin{align*} \theta ^0 \mid_{H^{m_0+1 } (\beta , \fA)}=& \left( \theta  \mid_{H^{m_0+1 } (\beta , \fA)}\times  \prod\limits_{i=1}^s (\theta ^i )^{-1}\mid_{H^{m_0+1 } (\beta , \fA)}\right)\\
    &= \left( \theta \mid_{H^{m_0+1 } (\beta , \fA)} \times (\theta ')^{-1}\right)\\
    &= \left( \psi _{c_0} \times \theta ' \times (\theta ')^{-1}\right)\\
    & = \psi _{c_0} \\
   \end{align*}
 This completes the proof of the theorem, indeed we have found the required characters $\phi _i$ , $0 \leq i \leq s$ such that Conditions $(i)$ and $(ii)$ are satisfied.

\end{proof}

\begin{prop} \label{minicici}For $0 \leq i \leq s-1$, the element $c_i$ is minimal relatively to $E_i/E_{i+1}$ and $c_s$ is minimal relatively to $E_s/F$.

\end{prop}
\begin{proof}
For $0 \leq i \leq s-1$, this is a direct consequence of Proposition \ref{minisimplealfalfa} and (vi') of the beginning of this section. For $i=s$, this is a direct consequence of Proposition $ \ref{minisimplealfalfa}$ and the fact that $k_0 (\beta _s , \fA) = -n $ or $ - \infty $.
\end{proof}
\subsection{Explicit factorisations of tame simple characters}\label{explifact}

To associate a generic Yu datum to each Bushnell-Kutzko datum, we need to introduce two cases. The two cases are denoted as follows: ${(Case~ A)}$ or ${(Case~ B)}$. In the rest of this paper, we write ${(Case ~A)}$ at the beginning of a paragraph or in a sentence to signify that we work under the ${(Case~ A)}$ hypothesis. We will introduce particular notations in the ${(Case ~A)}$. The same holds for ${(Case~ B)}$. ${(Case~ A)}$ is by definition when the last element $\beta _s$ of the fixed chosen defining sequence is inside the field $F$; that is,  $\beta _s \in F$. $(Case~ B)$ is the other case; that is, when  $ \beta _s \not \in F$.

\subsubsection{Explicit factorisations of tame simple characters in $(Case~ A)$} \label{explifacta}

Recall that in this case $\beta _s \in F$. In this case, we put $d=s$. Let us give an explicit description of the group $H^1(\beta , \fA)$ in this case. This explicit description is written in a convenient manner  to compare with Yu's construction.

\begin{prop} \label{huna} $(Case~ A)$ The group $H^1(\beta , \fA)$ is equal to the following group

\begin{equation}
 U^1 (\fB _{\beta _0}) U^{[\frac{-\nu _{\fA} (c_0)}{2}]+1}(\fB _{\beta _1} )\ldots U^{[\frac{-\nu _{\fA} (c_{i-1})}{2}]+1}(\fB _{\beta _i}) \ldots U^{[\frac{-\nu _{\fA} (c_{s-1})}{2}]+1}(\fB _{\beta _s})
\end{equation}
\end{prop}
\begin{proof} Recall that $\beta = \beta _0$. By \cite[3.1.14,3.1.15]{BK}, it is enough to show that 

\begin{center}
\begin{equation} \label{HHli}
\mathfrak{H}(\beta , \fA ) = \fB _{\beta _0 } + \fQ _{\beta _1} ^{[- \frac{\nu _{\fA} (c_o)}{2}]+1}+ \ldots +  \fQ _{\beta _s} ^{[- \frac{\nu _{\fA} (c_{s-1})}{2}]+1}.
\end{equation}
\end{center}
\begin{sloppypar}
Let us prove \eqref{HHli} by induction on $s$. If $s=0$, by definition, $\mathfrak{H}(\beta , \fA )= \fB _{\beta _0 } + \fP ^{[\frac{n}{2}]+1}$. The element $\beta _0$ is in $F$, thus $\fB _{\beta _0} = \fA $. Consequently, $\mathfrak{H}(\beta , \fA )= \fB _{\beta _0 }$. If ${s>0}$, then by induction hypothesis we have $ {\mathfrak{H}(\beta _1 , \fA )= \fB _{\beta _1} +  \fQ _{\beta _2} ^{[- \frac{\nu _{\fA} (c_1)}{2}]+1} +\ldots + \fQ _{\beta _s} ^{[- \frac{\nu _{\fA} (c_{s-1})}{2}]+1} }$.
\end{sloppypar}
By definition $\mathfrak{H}(\beta , \fA ) = \fB _{\beta _0 } + \mathfrak{H}(\beta _1, \fA ) \cap \fP ^{[\frac{-k_0 (\beta _0 , \fA)}{2}]+1}$.
 Let us remark that because the stratum $[\fB _{\beta _1} , -k_0(\beta _0 , \fA ) , - k_0 (\beta _0 , \fA )+1 , \beta _0 - \beta _1 ]$
  is simple by the condition $(vi')$, the equality $\nu _{\fB _{\beta _1}} (\beta_0 - \beta _1 )= k_0 (\beta _0, \fA)$ holds. We have $\nu _{\fB _{\beta _1}} (\beta_0 - \beta _1 )= \nu _{\fA} (\beta_0 - \beta _1 ) = \nu _{\fA} (c_0)$. So $k_0 (\beta _0, \fA)=\nu _{\fA} (c_0)$.
Consequently, \begin{align*}\mathfrak{H}(\beta , \fA ) &= \fB _{\beta _0 } +  \mathfrak{H}(\beta _1, \fA ) \cap \fP ^{[\frac{-\nu _{\fA} (c_0)}{2}]+1} \\
&=\fB _{\beta _0 } + \fQ _{\beta _1} ^{[- \frac{\nu _{\fA} (c_o)}{2}]+1}+ \ldots +  \fQ _{\beta _s} ^{[- \frac{\nu _{\fA} (c_{s-1})}{2}]+1},\end{align*} as required.
\end{proof}
We now reformulate Theorem \ref{facttam} in (Case A) for simple characters in $\mathcal{C}(\fA , 0, \beta )$. 
\begin{coro} \label{tamcoa} $(Case~ A)$ Let $\theta  \in \mathcal{C}(\fA , 0, \beta )$, let $\phi _0 , \phi _1 , \ldots , \phi _s$ be the characters introduced in theorem \ref{facttam}, then  $\theta = \displaystyle \prod_{i=0}^s \theta ^i$ where  $\theta ^i$ is the character defined as follows: If $0 \leq i \leq s-1$, then the character  $\theta _i$ is defined by the following two conditions: \begin{enumerate}[(i)] 
\item $ \theta ^i \mid _{U^1 (\fB _{\beta _0}) U^{[\frac{-\nu _{\fA} (c_0)}{2}]+1}(\fB _{\beta _1} )\ldots U^{[\frac{-\nu _{\fA} (c_{i-1})}{2}]+1}(\fB _{\beta _i}) } = \phi _i \circ \det _{B _{\beta _i}}$
\item $ \theta ^i \mid _{U^{[\frac{-\nu _{\fA} (c_{i})}{2}]+1}(\fB _{\beta _{i+1}}) \ldots U^{[\frac{-\nu _{\fA} (c_{s-1})}{2}]+1}(\fB _{\beta _s})} = \psi _{c_i}$ .
\end{enumerate}
If $i = s$, then $\theta ^i $ is defined by $\theta ^i \mid _{H^1(\beta , \fA)} = \phi _i \circ \det _{A}$.

\end{coro}

\begin{proof} The proof consists in applying Theorem \ref{facttam} using the explicit description of $H^1(\beta , \fA)$ given in the lemma \ref{huna}.
In Theorem \ref{facttam}, we have introduced smooth characters $\phi _0 , \ldots \phi _s $ of $E_0 ^{\times} , \ldots E_s ^{\times}$ such that $\theta =  \displaystyle \prod_{i=0}^s \theta ^i $ where $\theta ^i$ is defined by the following two conditions: \begin{enumerate}[(i)]

\item $ \theta ^i \mid _{H^{1} (\beta , \fA ) \cap B _{\beta _i }} = \phi _i \circ \det _{B _{\beta _i}}$
\item $ \theta ^i \mid _{H^{m_i+1 } (\beta , \fA)} = \psi _{c_i}$ where $m_i  = \max \{[\frac{-\nu _{\fA} (c_i)}{2}] , 0 \}$.

\end{enumerate} 
\begin{sloppypar}
Let $0 \leq i \leq s-1$, then Lemma \ref{huna} shows that $H^1(\beta , \fA)  \cap B _{\beta _i} = U^1 (\fB _{\beta _0}) U^{[\frac{-\nu _{\fA} (c_0)}{2}]+1}(\fB _{\beta _1} )\ldots U^{[\frac{-\nu _{\fA} (c_{i-1})}{2}]+1}(\fB _{\beta _i})$. Consequently, the condition $(i)$ of the corollary \ref{tamcoa} is satisfied for $\theta ^i$.
Trivially $m_i  = [\frac{-\nu _{\fA} (c_i)}{2}]$, moreover  the lemma \ref{huna} shows that $H^{[\frac{-\nu _{\fA} (c_i)}{2}]+1  } (\beta , \fA) =U^{[\frac{-\nu _{\fA} (c_{i})}{2}]+1}(\fB _{\beta _{i+1}}) \ldots U^{[\frac{-\nu _{\fA} (c_{s-1})}{2}]+1}(\fB _{\beta _s})$. Thus, Condition $(ii)$ of Corollary \ref{tamcoa} is satisfied for $\theta ^i$.
Finally, for $i=s$, we have $ \theta ^i \mid _{H^{1} (\beta , \fA ) \cap B _{\beta _i }} = \phi _i \circ \det _{B _{\beta _i}}$ by the theorem and the condition of the corollary is satisfied remarking that $B_{\beta _s} =A$ because $\beta _s \in F$.

\end{sloppypar}

\end{proof}

\subsubsection{Explicit factorisations of tame simple characters in $(Case ~B)$} \label{explifactb}

 Recall that in this case $\beta _s \not \in F$. In this case, we put $d=s+1$. Let us give an explicit description of the group $H^1(\beta , \fA)$ in this case. This explicit description is written in a convenient way in order to compare with Yu's construction.

\begin{prop} \label{hunb} $(Case ~B)$ The group $H^1(\beta , \fA)$ is equal to the following group:
\begin{scriptsize}

\begin{equation} U^1 (\fB _{\beta _0}) U^{[\frac{-\nu _{\fA} (c_0)}{2}]+1}(\fB _{\beta _1} )\ldots U^{[\frac{-\nu _{\fA} (c_{i-1})}{2}]+1}(\fB _{\beta _i}) \ldots U^{[\frac{-\nu _{\fA} (c_{s-1})}{2}]+1}(\fB _{\beta _s}) U^{[\frac{-\nu _{\fA} (c_{s})}{2}]+1}(\fA).
\end{equation}

\end{scriptsize}

\end{prop}
\begin{proof} By \cite[3.1.14,3.1.15]{BK}, it is enough to show that 

\begin{center}
\begin{equation} \label{HHlib}
\mathfrak{H}(\beta , \fA ) = \fB _{\beta _0 } + \fQ _{\beta _1} ^{[- \frac{\nu _{\fA} (c_o)}{2}]+1}+ \ldots +  \fQ _{\beta _s} ^{[- \frac{\nu _{\fA} (c_{s-1})}{2}]+1} +\fP^{[\frac{-\nu _{\fA} (c_{s})}{2}]+1} .
\end{equation}
\end{center}

Let us prove \eqref{HHlib} by induction on $s$. If $s=0$, by definition, $\mathfrak{H}(\beta , \fA )= \fB _{\beta _0 } + \fP ^{[\frac{n}{2}]+1}$, where by definition $n = - \nu _{\fA} (\beta, \fA)$. Because $s=0$, the equality $\beta =c_s=c_0$ hold. Thus, $\mathfrak{H}(\beta , \fA ) = \fB _{\beta _0 }  +\fP^{[\frac{-\nu _{\fA} (c_{s})}{2}]+1} $ as required.
If $s>0$, then by induction hypothesis we have
\begin{center}
 {$ {\mathfrak{H}(\beta _1 , \fA )= \fB _{\beta _1} +  \fQ _{\beta _2} ^{[- \frac{\nu _{\fA} (c_1)}{2}]+1} +\ldots + \fQ _{\beta _s} ^{[- \frac{\nu _{\fA} (c_{s-1})}{2}]+1} } +\fP^{[\frac{-\nu _{\fA} (c_{s})}{2}]+1}$}.\end{center}
By definition $\mathfrak{H}(\beta , \fA ) = \fB _{\beta _0 } + \mathfrak{H}(\beta _1, \fA ) \cap \fP ^{[\frac{-k_0 (\beta _0 , \fA)}{2}]+1}$.
 Let us remark that because the stratum $[\fB _{\beta _1} , -k_0(\beta _0 , \fA ) , - k_0 (\beta _0 , \fA )+1 , \beta _0 - \beta _1 ]$
  is simple by the condition $(vi')$, the equality $\nu _{\fB _{\beta _1}} (\beta_0 - \beta _1 )= k_0 (\beta _0, \fA)$ holds. We have $\nu _{\fB _{\beta _1}} (\beta_0 - \beta _1 )=\nu _{\fA} (\beta _0 - \beta _1) = \nu _{\fA} (c_0)$. So $\nu _{\fA} (c_0)=k_0 (\beta _0 , \fA)$.
Consequently, \begin{align*}
\mathfrak{H}(\beta , \fA )& = \fB _{\beta _0 } + \mathfrak{H}(\beta _1, \fA ) \cap \fP ^{[\frac{-\nu _{\fA} (c_0)}{2}]+1}\\
&=\fB _{\beta _0 } + \fQ _{\beta _1} ^{[- \frac{\nu _{\fA} (c_o)}{2}]+1}+ \ldots +  \fQ _{\beta _s} ^{[- \frac{\nu _{\fA} (c_{s-1})}{2}]+1}+\fP^{[\frac{-\nu _{\fA} (c_{s})}{2}]+1}, \end{align*} as required.

\end{proof}
We now reformulate Theorem \ref{facttam} in  $(Case ~B)$ for the simple characters in $\mathcal{C}(\fA , 0 ,\beta )$. 

\begin{coro} \label{tamcob} $(Case ~B)$ Let $\theta  \in \mathcal{C}(\fA , 0, \beta )$, there exists $\phi _0 , \phi _1 , \ldots , \phi _s$ such that $\theta =   \displaystyle \prod_{i=0}^s \theta ^i$ where the $\theta ^i$ are the characters defined by the following conditions for $0 \leq i \leq s$: \begin{enumerate}[(i)] 
\item $ \theta ^i \mid _{U^1 (\fB _{\beta _0}) U^{[\frac{-\nu _{\fA} (c_0)}{2}]+1}(\fB _{\beta _1} )\ldots U^{[\frac{-\nu _{\fA} (c_{i-1})}{2}]+1}(\fB _{\beta _i}) } = \phi _i \circ \det _{B _{\beta _i}}$
\item $ \theta ^i \mid _{U^{[\frac{-\nu _{\fA} (c_{i})}{2}]+1}(\fB _{\beta _{i+1}}) \ldots U^{[\frac{-\nu _{\fA} (c_{s-1})}{2}]+1}(\fB _{\beta _s}) U^{[\frac{-\nu _{\fA} (c_{s})}{2}]+1}(\fA)} = \psi _{c_i}$ .
\end{enumerate}

\end{coro}

\begin{proof} The proof consists in applying Theorem \ref{facttam} using the explicit description of $H^1(\beta , \fA)$ given in Lemma \ref{hunb}.
By Theorem \ref{facttam}, there exist smooth characters $\phi _0 , \ldots , \phi _s $ of $E_0 ^{\times} , \ldots , E_s ^{\times}$ such that $\theta =  \displaystyle \prod_{i=0}^s \theta ^i $, where $\theta ^i$, $0 \leq i \leq s$, is defined by the following two conditions: \begin{enumerate}[(i)]

\item $ \theta ^i \mid _{H^{1} (\beta , \fA ) \cap B _{\beta _i }} = \phi _i \circ \det _{B _{\beta _i}}$
\item $ \theta ^i \mid _{H^{m_i+1 } (\beta , \fA)} = \psi _{c_i}$ where $m_i  = \max \{[\frac{-\nu _{\fA} (c_i)}{2}] , 0 \}$.

\end{enumerate} \begin{sloppypar}
Let $0 \leq i \leq s$. Then Lemma \ref{hunb} shows that \begin{center} $H^1(\beta , \fA)  \cap B _{\beta _i} = U^1 (\fB _{\beta _0}) U^{[\frac{-\nu _{\fA} (c_0)}{2}]+1}(\fB _{\beta _1} )\ldots U^{[\frac{-\nu _{\fA} (c_{i-1})}{2}]+1}(\fB _{\beta _i})$.\end{center} Thus, condition $(i)$ of Corollary \ref{tamcob} is satisfied for $\theta ^i$.
Trivially we have $m_i  = [\frac{-\nu _{\fA} (c_i)}{2}]$. Moreover Lemma \ref{hunb} shows that $H^{[\frac{-\nu _{\fA} (c_i)}{2}]+1  } (\beta , \fA) =U^{[\frac{-\nu _{\fA} (c_{i})}{2}]+1}(\fB _{\beta _{i+1}}) \ldots U^{[\frac{-\nu _{\fA} (c_{s-1})}{2}]+1}(\fB _{\beta _s})U^{[\frac{-\nu _{\fA} (c_{s})}{2}]+1}(\fA)$. Thus Condition $(ii)$ of Corollary \ref{tamcob} is satisfied for $\theta ^i$.
\end{sloppypar}

\end{proof}

\section{Generic characters associated to tame simple characters} \label{sectgecha}

We continue with the same notations as in Section  \ref{tameca}. Thus, we have a fixed tame simple stratum $[\fA , n , 0 , \beta ]$ and various objects and notations relative to it. In particular we have a defining sequence and a simple character $\theta  \in \mathcal{C}(\fA , 0 , \beta ) . $ We have  also distinguished two cases. In both $(Case ~A)$ and $(Case ~B)$, we have introduced various  objects and notations and have established results relative to them. In this section, we are going to introduce a $4$-uple $(\overrightarrow{G} , y , \overrightarrow{r} , \overrightarrow{\boldsymbol{\Phi}} )$, which will be part of a complete Yu datum.

\subsection{The characters $\boldsymbol{\Phi}_i$ associated to a factorisation of a tame simple character} \label{phiiassocitaedtofacto}

We start with $(Case ~A)$.

\subsubsection{The characters $\boldsymbol{\Phi}_i$ in $(Case ~A)$} \label{phiiia}

 In section \ref{tameca}, we have introduced a sequence of fields \begin{center} ${E_0 \supsetneq E_1 \supsetneq \ldots \supsetneq E_i \supsetneq \ldots \supsetneq E_s  }$.\end{center} Recall that in this case $d=s$ and $E_s =F$ because $\beta _s \in F $ and $E_s=F[\beta _s]$.
  For each $i$, the field $E_i$ is included in the algebra $A=\End_F(V)$; that is, $V$ is an $E_i$-vector space.
For $0 \leq i \leq s$, put $G^i = \Res _{E_i /F} \underline{\Aut}_{E_i} (V)$.  If $0 \leq i \leq j \leq d$, then $G^i$ is canonically a closed subgroup scheme of $G^j$.
Let $\overrightarrow{G}$ be the sequence \[G^0 \subsetneq G^1 \subsetneq \ldots \subsetneq G^s .\] 

\begin{prop} \label{sequencetwia} $(Case ~A)$ The sequence $\overrightarrow{G}$ is a tamely ramified twisted Levi sequence in $G$.

\end{prop}

\begin{proof} This is a consequence of \ref{corotwistedlevitame}.

\end{proof}

We now introduce some real numbers $\mathbf{r} _i$ for $0 \leq i  \leq s$. Put $\mathbf{r}_i := -\ord(c_i)$ for $0 \leq i \leq s$. Put also $\overrightarrow{\mathbf{r}}=(\mathbf{r}_0 , \mathbf{r}_1 , \ldots ,\mathbf{r}_i ,\ldots ,\mathbf{r}_s )$.

\begin{prop}$(Case ~A)$ \label{rivucia}For $ 0 \leq i \leq s $, the real number $\mathbf{r_i} $ satisfies the following formula:
\begin{center}
 $\mathbf{r_i}= \frac{- \nu _{\fA} (c_i)}{e(\fA \mid \of )}$.
 \end{center}
\end{prop}

\begin{proof}
 By definition, $\mathbf{r}_i = - \ord(c_i)$. In addition, by definition of $\ord$ we know that \begin{equation}\ord(c_i) = \frac{\nu _{E_i} (c_i)}{e(E_i \mid F)}. \label{dz1}\end{equation} Lemma \ref{valval} shows that \begin{equation} \frac{ \nu _{\fA} (c_i)}{e(\fA \mid \of)} = \frac{\nu _{E_i} (c_i) }{e(E_i \of F)} \label{dz2} .\end{equation} Equations \ref{dz1} and \ref{dz2} together finish the proof of the proposition.

\end{proof}

\begin{prop} \label{filaeg}$(Case ~A)$  There exists a point $y $ in $\BT ^E (G^0 , F)$ such that the following properties  hold.
\begin{enumerate}
\item[(I)] The following equalities hold.
\begin{enumerate}[(i)]
\item $ U^0(\fB_{\beta _0})=G^0(F)_{y,0}$

   \item $ U^1(\fB_{\beta_0})=G^0(F)_{y,0+}$

 \item $  \fQ _{\beta_0}  = \mathfrak{g}^0(F)_{y, 0 +}$

 \item $  \fB_{\beta_0} = \mathfrak{g}^0(F)_{y, 0} $ 

   \item  $ F[\beta]^{\times }  U^0(\fB_{\beta _0})=G^0(F)_{[y]}$
\end{enumerate}

\item[(II)]
\begin{flushleft}
 There exist continuous, affine and $G^{i-1} (F)$-equivariant maps  ${\iota _i : \xymatrix{ \BT^E(G^{i-1},F) \ar@{^{(}->}[r]^-{\iota _i} & \BT^E(G^i,F)}}$, for $1\leq i \leq s$, such that, denoting $\iota ^i$ the composition ${\iota _i \circ \iota _{\iota _{i-1}} \circ \ldots \circ \iota _1}$, the following equalities hold:
  \end{flushleft}

\begin{enumerate}[(i)]
 \item $  U^{[\frac{-\nu _{\fA} (c_{i-1})}{2}]+1}(\fB_{\beta_i} )   = G^i(F)_{\iota ^i(y), \frac{\mathbf{r}_{i-1}}{2}+}$
 \item $ U^{[\frac{-\nu _{\fA} (c_{i-1})+1}{2}]}(\fB_{\beta_i} )= G^i(F)_{\iota ^i(y), \frac{\mathbf{r}_{i-1}}{2}}$
 \item $ U^{-\nu _{\fA} (c_{i-1})+1}(\fB_{\beta_i} ) = G^i(F)_{\iota ^i(y), \mathbf{r}_{i-1} +}$
 \item $ U^{-\nu _{\fA} (c_{i-1})}(\fB_{\beta_i} ) = G^i(F)_{\iota ^i(y), \mathbf{r}_{i-1}} $
 \item $  \fQ _{\beta_i}  ^{[\frac{-\nu _{\fA} (c_{i-1})}{2}]+1} = \mathfrak{g}^i(F)_{\iota ^i(y), \frac{\mathbf{r}_{i-1}}{2}+}$
 \item $  \fQ _{\beta_i}  ^{[\frac{-\nu _{\fA} (c_{i-1})+1}{2}]} = \mathfrak{g}^i(F)_{\iota ^i(y), \frac{\mathbf{r}_{i-1}}{2}}$
 \item $  \fQ _{\beta_i} ^{-\nu _{\fA} (c_{i-1})+1} = \mathfrak{g}^i(F)_{\iota ^i(y), \mathbf{r}_{i-1} +}$
 \item $  \fQ _{\beta_i} ^{-\nu _{\fA} (c_{i-1})} = \mathfrak{g}^i(F)_{\iota ^i(y), \mathbf{r}_{i-1}} $ and moreover,

 \item $ U^{-\nu _{\fA} (c_{i})}(\fB_{\beta_i} ) = G^i(F)_{\iota ^i(y), \mathbf{r}_{i}} $

 \item $U^{-\nu _{\fA} (c_i)+1}(\fB_{\beta _i})= G^i(F)_{\iota ^i(y),\mathbf{r}_i+}$

\end{enumerate}
\end{enumerate}

In the rest of this paper, we identify $\iota ^i(y)$ and $y$.
\end{prop}

\begin{proof} \begin{sloppypar} In \cite{Brou}, the authors construct an explicit bijection between the set $Latt^1(V)$ of all lattices functions in $V$ (see \cite[Definition I.2.1]{Brou} for the definition of a lattice function)  and the enlarged Bruhat-Tits building of $\underline{\Aut}_F (V)$ (combine \cite[Prop I.1.4]{Brou} and \cite[Prop I.2.4]{Brou}). The group $\R$ acts on $Latt^1(V)$ and the previous bijection induces a bijection between $Latt(V):= Latt^1(V) / \R$ and the reduced Bruhat-Tits building $\BT^R(\underline{\Aut}_F (V),F)$. The authors show \cite[Theorem II.1.1]{Brou} that if $E/F \subset A$ is a separable extension of fields, then  there is a canonical affine and continuous embedding from $\BT^R(\Res _{E/F} (\underline{\Aut} _E (V) , F)$ to $\BT^R (\underline{\Aut}_F , F)$. Using the general fact that if $G$ is a connected reductive $k'$-group and $k'/k$ is a separable finite extension of non-Archimedean local field, then $\BT^R(\Res _{k'/k} (G) , k') = \BT^R(G,k)$; we deduce canonical maps $\BT^R(G^{i-1},F) \to \BT^R(G^i,F)$ for $1 \leq i \leq d$. Recall that $\BT^E(G,k)$ is defined as $\BT^R(G,k) \times X_*(Z(G),F) \otimes _{\Z} \R.$ Because $Z(G^0 ) /Z(G)$ is anisotropic, $X_*(Z(G^{i-1}), F) \otimes _{\Z} \R $ and $X_* (Z(G^{i},F)$ are isomorphic for $1 \leq i \leq d$. Fix such isomorphisms . They induce continous, affine and $G^{i-1}(F)$-equivariant embeddings \end{sloppypar}
\begin{center}
$\BT^E(G^{i-1},F) \to \BT^E (G^i,F)$.
\end{center}
In \cite[I §7]{Brou}, the authors explain that there are injective maps 

\begin{center}
{\small $\{\text{Lattices chains in}~ V  \} \to \{\text{Lattices sequences in} ~V\} \to \{\text{Lattices functions in}~ V\}$}.
\end{center}
Let $\Lambda \in Latt^1(V)$. To the class $\overline{\Lambda}$ of $\Lambda$, Broussous-Lemaire attach a filtration $a_r(\overline{\Lambda})$ of $A$ and a filtration $U_r(\overline{\Lambda})$ of $A^{\times} =G$, they are indexed by $\R$ and $\R _{\geq 0}$. If $\Lambda$ comes from a lattices chain $\mathcal{L}$, then the filtration of $A$ of Broussous-Lemaire is compatible with the filtration, indexed by $\Z$, given by powers of the radical of the hereditary order associated to $\mathcal{L}$. 
Let $\mathcal{L}$ be an $\mathfrak{o}_E$-lattices chain associated to $\mathfrak{B}$. We thus get a point in $\BT^E(G^0 , F)$ by the previous considerations. The rest of the proposition is a consequence of \cite{Brou}[Theorem II.1.1] and \cite{Brou}[Appendix A].

\end{proof}

 Let us introduce some characters $\boldsymbol{\Phi} _i$ , $ 0 \leq i \leq s$. 

\begin{defi} \label{PHIIa} $(Case ~A)$ Let $0 \leq i \leq s$, and let $\boldsymbol{\Phi } _i$ be the smooth complex character of $G^i(F)$ defined by $\boldsymbol{\Phi } _i:= \phi _i \circ \det_{B_{\beta _i}}$ where $\phi _i$ is the character introduced in  \ref{facttam} ,\ref{tamcoa}.

\end{defi}

\begin{prop} \label{PHIIgenea} $(Case ~A)$ The following assertions hold: \begin{enumerate}[(i)]
\item For $0 \leq i \leq s-1$, the character $\boldsymbol{\Phi}_i$  is $G^{i+1}$-generic of depth $\mathbf{r}_i$ relatively to $y$. \item The character $\boldsymbol{\Phi}_s$ is of depth $\mathbf{r}_s$ relatively  to $y$. \end{enumerate}
\end{prop}

\begin{proof} \begin{enumerate}[(i)] \item Let us first prove that $\boldsymbol{\Phi}_i$ is of depth $\mathbf{r}_i$ relatively  to $y$ for $0 \leq i \leq s-1$. The restriction $\boldsymbol{\Phi } _i \mid_{G^i(F)_{\mathbf{r_i} }}$ is equal to the restriction $\boldsymbol{\Phi } _i \mid_{U^{-\nu _{\fA} (c_{i})} (\fB _{\beta _i} )}$ by proposition \ref{filaeg}. 
 Let us prove that the two inclusions  \begin{equation} \label{inclu1}{U^{-\nu _{\fA} (c_{i})} (\fB _{\beta _i} )} \subset {U^1 (\fB _{\beta _0}) U^{[\frac{-\nu _{\fA} (c_0)}{2}]+1}(\fB _{\beta _1} )\ldots U^{[\frac{-\nu _{\fA} (c_{i-1})}{2}]+1}(\fB _{\beta _i}) }\end{equation} and \begin{equation}\label{inclu2}{U^{-\nu _{\fA} (c_{i})} (\fB _{\beta _i} )} \subset U^{[\frac{-\nu _{\fA} (c_{i})}{2}]+1}(\fB _{\beta _{i+1}}) \ldots U^{[\frac{-\nu _{\fA} (c_{s-1})}{2}]+1}(\fB _{\beta _s}) \end{equation} hold.
If $i=0$, then the first inclusion is trivial. Assume now that $i>0$. To prove the first inclusion in this case, we remark that the inequality of integers $-\nu _{\fA} (c_{i-1}) < - \nu _{\fA} (c_i)$ holds. 
We deduce easily and successively the inequalities
\begin{align*}
-\nu _{\fA} (c_{i-1}) &< - \nu _{\fA} (c_i)\\
\frac{-\nu _{\fA} (c_{i-1})}{2}&< - \nu _{\fA} (c_i)\\
[\frac{-\nu _{\fA} (c_{i-1})}{2}]+1 &\leq - \nu _{\fA} (c_i).
\end{align*}
So ${U^{-\nu _{\fA} (c_{i})} (\fB _{\beta _i} )}  \subset U^{[\frac{-\nu _{\fA} (c_{i-1})}{2}]+1} (\fB _{\beta _i} )$, and the first equality holds.
To prove the inclusion \eqref{inclu2}, we remark that the integer $-\nu _{\fA} (c_i)$ is strictly bigger than $0$. We easily deduce successively that
\begin{align*}
-\nu _{\fA} (c_{i}) &>\frac{- \nu _{\fA} (c_i)}{2}\\
-\nu _{\fA} (c_{i})&\geq [\frac{- \nu _{\fA} (c_i)}{2}]+1.\\
\end{align*}
Thus, because $\fB _{\beta _i} \subset \fB _{\beta _{i+1}}$, we get \begin{center} ${U^{-\nu _{\fA} (c_{i})} (\fB _{\beta _i} )} \subset  U^{[\frac{-\nu _{\fA} (c_{i})}{2}]+1}(\fB _{\beta _{i+1}})$ \end{center} and the second inequality follows.
The inclusions \eqref{inclu1} and \eqref{inclu2} together with \ref{tamcoa} imply that  \begin{center}$\boldsymbol{\Phi}_i \mid _{G^i(F)_{y,\mathbf{r}_i}} =\phi _i \circ \det  \mid_{U^{-\nu _{\fA} (c_{i})} (\fB _{\beta _i} )} = \theta ^i \mid_{U^{{-\nu _{\fA} (c_{i})}} (\fB _{\beta _i} )}= \psi _{c_i}\mid_{U^{{-\nu _{\fA} (c_{i})}} (\fB _{\beta _i} )}$.\end{center}
We know that $\psi _{c_i} $ is trivial on $U^{-\nu _{\fA} (c_i)+1}(\fB_{\beta _i})$ and non-trivial on $U^{-\nu _{\fA} (c_i)}(\fB_{\beta _i})$. Consequently, because $U^{-\nu _{\fA} (c_i)+1}(\fB_{\beta _i})= G^i(F)_{y,\mathbf{r}_i+}$ by \ref{filaeg}, the character $\boldsymbol{\Phi}_i$ is of depth $\mathbf{r}_i$ relatively to $y$. Now we have to show that  $\boldsymbol{\Phi _i}$ is $G^{i+1}$-generic of depth $\mathbf{r}_i$ for  ${0\leq i \leq s-1}$. By definition, $\psi _{c_i} (1+x) = \psi \circ \Tr _{A/F} ( c_i x)$. 
We have thus obtained that 
\begin{equation} \label{realize}
\boldsymbol{\Phi } _i \mid_{G^i(F)_{\mathbf{r_i} :\mathbf{r_i}+} }(1+x) = \psi \circ \Tr _{A/F} ( c_i x)
\end{equation}
\begin{sloppypar}
As explained in Section \ref{yu}, the characters of $\mathfrak{g}^i(F)_{\mathbf{r_i} : \mathbf{r_i}+} \simeq \mathfrak{g}^i(F)_{\mathbf{r_i} }/ \mathfrak{g}^i(F)_{\mathbf{r_i} +}$ are in bijection via $\psi$ with $\mathfrak{g}^i(F)_{\mathbf{r_i} +}^{\bullet} / \mathfrak{g}^i(F)_{\mathbf{r_i} }^{\bullet}$ where 
\end{sloppypar}

$\mathfrak{g}^i(F)_{\mathbf{r_i} +}^{\bullet} =\{ x\in {\mathfrak{g}^i}^*(F) \mid x (\mathfrak{g}^i(F)_{\mathbf{r_i} +}) \subset \mathfrak{o}_F \} \otimes _{\mathfrak{o}_F} \mathfrak{p}_F= {\mathfrak{g}^i}^*(F)_{-\mathbf{r_i} } $

and

$\mathfrak{g}^i(F)_{\mathbf{r_i} }^{\bullet} =\{ x\in { \mathfrak{g}^i}^*(F) \mid x (\mathfrak{g}^i(F)_{\mathbf{r_i} }) \subset \mathfrak{o}_F \} \otimes _{\mathfrak{o}_F} \mathfrak{p}_F= {\mathfrak{g}^i}^*(F)_{(-\mathbf{r_i})+ }$

The isomorphism  $G^i(F)_{\mathbf{r_i} : \mathbf{r_i}+} \simeq \mathfrak{g}^i(F)_{\mathbf{r_i} : \mathbf{r_i}+}$ used by Yu \cite{YU}, is the same as the one used by Adler in \cite{Adle} , and it is given in our case by the map ${((1+x )\mapsto x )}$.
The element $X^*_{c_i}= (x\mapsto \Tr _{A/F} (c_i x))$ is an element in $\Lie^* (Z(G^i))_{-\mathbf{r_i}} \subset
 {\mathfrak{g}^i}^* (F) _{y, -\mathbf{r_i}}$. The equation \eqref{realize} shows that $X^*_{c_i}$ realises  $\boldsymbol{\Phi } _i \mid_{G^i(F)_{\mathbf{r_i} :\mathbf{r_i}+ } }$. The element $c_i$ is minimal relatively to $E_i/E_{i+1}$ by Proposition \ref{minicici}.
 So  $X^*_{c_i}$ is $G^{i+1}$-generic of depth $-\ord (c_i)$ by Theorem \ref{minimalgenericelemenx}. Thus, $\boldsymbol{\Phi}_i$ is $G^{i+1}$-generic of depth $\mathbf{r}_i$.
\item \begin{sloppypar} Let us show that $\boldsymbol{\Phi}_s$ is of depth $\mathbf{r}_s$  relatively to $y$. This is easier than $(i)$.
By $\ref{filaeg}$, we have ${G(F)_{y,\mathbf{r}_s } = U^{-\nu _{\fA} (c_s)}(\fB _{\beta _s} )}$ and ${G(F)_{y,\mathbf{r}_s +} = U^{-\nu _{\fA} (c_s)+1}(\fB _{\beta _s} )}$. \end{sloppypar}
Thus, by using  \ref{facttam}, we get \begin{center} $\boldsymbol{\Phi _s} \mid _{G(F)_{y,\mathbf{r}_s}} = \phi _s \circ \det \mid _{U^{-\nu _{\fA} (c_s)}(\fB _{\beta _s} )} = \theta ^s \mid _{U^{-\nu _{\fA} (c_s)}(\fB _{\beta _s} )} =  \psi _{c_s}$. \end{center}
The character $\psi _{c_s}$ is trivial on ${G(F)_{y,\mathbf{r}_s +} = U^{-\nu _{\fA} (c_s)+1}(\fB _{\beta _s} )}$ and non-trivial on ${G(F)_{y,\mathbf{r}_s } = U^{-\nu _{\fA} (c_s)}(\fB _{\beta _s} )}$. This ends the proof of $(ii)$

\end{enumerate}

\end{proof}

\subsubsection{The characters $\boldsymbol{\Phi}_i$ in  $(Case~B)$} \label{phiiib}

 We have already introduced a sequence of fields
\begin{center}$E_0 \supsetneq E_1 \supsetneq \ldots \supsetneq E_i \supsetneq \ldots \supsetneq E_s  $.\end{center} Recall that in this case $d=s+1$ and $E_d =F$ by definition.
  For each $i$, the field $E_i$ is contained in the algebra $A=\End_F(V)$; that is, $V$ is an $E_i$-vector space.
For $0 \leq i \leq d$, put $G^i = \Res _{E_i /F} \underline{\Aut}_{E_i} (V)$.  
Let $\overrightarrow{G}$ be the sequence \[G^0 \subset G^1 \subset \ldots \subset G^d.\] 

\begin{prop} \label{sequencetwib} $(Case~B)$ The sequence $\overrightarrow{G}$ is a tamely ramified twisted Levi sequence in $G$.

\end{prop}

\begin{proof} The $(Case~A)$ proof adapts to $(Case~B)$ without change.
\end{proof}

We now introduce some real numbers $\mathbf{r} _i$ for $0 \leq i  \leq d$. Put $\mathbf{r}_i := -\ord(c_i)$ for $0 \leq i \leq s$. Put $\mathbf{r}_d = \mathbf{r}_s$. Put also $\overrightarrow{\mathbf{r}}=(\mathbf{r}_0 , \mathbf{r}_1 , \ldots ,\mathbf{r}_i ,\ldots ,\mathbf{r}_s ,\mathbf{r}_d)$.

\begin{prop} \label{rivucib}$(Case~B)$ For $ 0 \leq i \leq s $, the real number $\mathbf{\mathbf{r_i}} $ satisfies the formula
\begin{center}
 $\mathbf{\mathbf{r_i}}= \frac{- \nu _{\fA} (c_i)}{e(\fA \mid \of )}$.
 \end{center}
\end{prop}

\begin{proof}
The $(Case~A)$ proof adapts to $(Case~B)$ without change.
\end{proof}

\begin{prop}(Case B) \label{filbeg}  There exists a point $y $ in $\BT ^E (G^0 , F)$ such that the following properties  hold: \begin{enumerate}
\item[(I)] The following equalities hold.
\begin{enumerate}[(i)]
\item $ U^0(\fB_{\beta _0})=G^0(F)_{y,0}$

   \item $ U^1(\fB_{\beta_0})=G^0(F)_{y,0+}$

 \item $  \fQ _{\beta_0}  = \mathfrak{g}^0(F)_{y, 0 +}$

 \item $  \fB_{\beta_0} = \mathfrak{g}^0(F)_{y, 0} $ 

   \item  $ F[\beta]^{\times }  U^0(\fB_{\beta _0})=G^0(F)_{[y]}$
\end{enumerate}

\item[(II)] \begin{flushleft}
 There exist continuous, affine and $G^{i-1} (F)$-equivariant maps  $\iota _i : \xymatrix{ \BT^E(G^{i-1},F) \ar@{^{(}->}[r]^-{\iota _i} & \BT^E(G^i,F)}$ for $1\leq i \leq s$, such that, denoting $\iota ^i$ the composition $\iota _i \circ \iota _{\iota _{i-1}} \circ \ldots \circ \iota _1$, the following equalities hold: \end{flushleft}

\begin{enumerate}[(i)]
 \item $  U^{[\frac{-\nu _{\fA} (c_{i-1})}{2}]+1}(\fB_{\beta_i} )   = G^i(F)_{\iota ^i(y), \frac{\mathbf{r}_{i-1}}{2}+}$
 \item $ U^{[\frac{-\nu _{\fA} (c_{i-1})+1}{2}]}(\fB_{\beta_i} )= G^i(F)_{\iota ^i(y), \frac{\mathbf{r}_{i-1}}{2}}$
 \item $ U^{-\nu _{\fA} (c_{i-1})+1}(\fB_{\beta_i} ) = G^i(F)_{\iota ^i(y), \mathbf{r}_{i-1} +}$
 \item $ U^{-\nu _{\fA} (c_{i-1})}(\fB_{\beta_i} ) = G^i(F)_{\iota ^i(y), \mathbf{r}_{i-1}} $
 \item $  \fQ _{\beta_i}  ^{[\frac{-\nu _{\fA} (c_{i-1})}{2}]+1} = \mathfrak{g}^i(F)_{\iota ^i(y), \frac{\mathbf{r}_{i-1}}{2}+}$
 \item $  \fQ _{\beta_i}  ^{[\frac{-\nu _{\fA} (c_{i-1})+1}{2}]} = \mathfrak{g}^i(F)_{\iota ^i(y), \frac{\mathbf{r}_{i-1}}{2}}$
 \item $  \fQ _{\beta_i} ^{-\nu _{\fA} (c_{i-1})+1} = \mathfrak{g}^i(F)_{\iota ^i(y), \mathbf{r}_{i-1} +}$
 \item $  \fQ _{\beta_i} ^{-\nu _{\fA} (c_{i-1})} = \mathfrak{g}^i(F)_{\iota ^i(y), \mathbf{r}_{i-1}} $  and moreover,

 \item $ U^{-\nu _{\fA} (c_{i})}(\fB_{\beta_i} ) = G^i(F)_{\iota ^i(y), \mathbf{r}_{i}} $

 \item $U^{-\nu _{\fA} (c_i)+1}(\fB_{\beta _i})= G^i(F)_{\iota ^i(y),\mathbf{r}_i+}$

 \end{enumerate}

 \item[(III)] \begin{flushleft}
 There exists a continuous, affine and $G^{s} (F)$-equivariant map  $\iota _d : \xymatrix{ \BT^E(G^{s},F) \ar@{^{(}->}[r]^-{\iota _i} & \BT^E(G^d,F)}$ such that, denoting $\iota ^d$ the composition $\iota _d \circ \iota _{\iota _{d}} \circ \ldots \circ \iota _1$, the following equalities hold: \end{flushleft}

\begin{enumerate}[(i)]
 \item $  U^{[\frac{-\nu _{\fA} (c_{s})}{2}]+1}(\fA )   = G^d(F)_{\iota ^i(y), \frac{\mathbf{r}_{s}}{2}+}$
 \item $ U^{[\frac{-\nu _{\fA} (c_{s})+1}{2}]}(\fA)= G^d(F)_{\iota ^i(y), \frac{\mathbf{r}_{s}}{2}}$
 \item $ U^{-\nu _{\fA} (c_{s})+1}(\fA ) = G^d(F)_{\iota ^i(y), \mathbf{r}_{s} +}$
 \item $ U^{-\nu _{\fA} (c_{s})}(\fA) = G^d(F)_{\iota ^i(y), \mathbf{r}_{s}} $
 \item $  \fP  ^{[\frac{-\nu _{\fA} (c_{s})}{2}]+1} = \mathfrak{g}^d(F)_{\iota ^i(y), \frac{\mathbf{r}_{s}}{2}+}$
 \item $  \fP  ^{[\frac{-\nu _{\fA} (c_{s})+1}{2}]} = \mathfrak{g}^d(F)_{\iota ^i(y), \frac{\mathbf{r}_{s}}{2}}$
 \item $  \fP ^{-\nu _{\fA} (c_{s})+1} = \mathfrak{g}^d(F)_{\iota ^i(y), \mathbf{r}_{s} +}$
 \item $  \fP ^{-\nu _{\fA} (c_{s})} = \mathfrak{g}^d(F)_{\iota ^i(y), \mathbf{r}_{s}} $ 
\end{enumerate}

\end{enumerate}

In the rest of this paper, we identify $\iota ^i(y)$ and $y$.
\end{prop}

\begin{proof} The $(Case~A)$ proof adapts to $(Case~B)$ without change for $(I)$ and $(II)$, the proof of $(II)$ adapts to $(III)$ without effort.
\end{proof}

Let us introduce certain  characters $\boldsymbol{\Phi }_i$ , $ 0 \leq i \leq d$. 

\begin{defi}\label{PHIIb} $(Case~B)$ Let $0 \leq i \leq s$, and let $\mathbf{\Phi } _i$ be the smooth complex character of $G^i(F)$ defined by $\boldsymbol{\Phi } _i:= \phi _i \circ \det_{B_{\beta _i}}$, where $\phi _i$ is the character introduced in \ref{facttam}, \ref{tamcob}. Let also $\boldsymbol{\Phi } _d$ be the trivial character $1$ of $G^d(F)$.

\end{defi}

\begin{prop} \label{PHIIgeneb}$(Case~B)$ For $0 \leq i \leq s$, the character $\boldsymbol{\Phi}_i$  is $G^{i+1}$-generic of depth $\mathbf{r}_i$. 
\end{prop}

\begin{proof}The $(Case~A)$ proof adapts to $(Case~B)$ without change.

\end{proof}

\subsection{The characters $\hat{\boldsymbol{\Phi}}_i$} \label{hatphiiii}

In both $(Case~A)$ and $(Case~B)$, we have obtained part of a Yu datum $(\overrightarrow{G}, y , r , \overrightarrow{\boldsymbol{\Phi} })$. To  $(\overrightarrow{G}, y , r , \overrightarrow{\boldsymbol{\Phi} })$ is attached by Yu various objects.  In the rest of this section, we show that the characters $\hat{\boldsymbol{\Phi}} _i$ (see section \ref{yu}) are equal to the factors $\theta _i$ of $\theta$.

\begin{prop} \label{hunkd} \begin{sloppypar}In both $(Case~A)$ and $(Case~B)$, let ${K_+^d=K_+^d (\overrightarrow{G}, y , r , \overrightarrow{\boldsymbol{\Phi} })}$ be the group attached to $(\overrightarrow{G}, y , r , \overrightarrow{\boldsymbol{\Phi} })$ (see section \ref{yu}). Then $H^1( \beta , \fA ) = K_+^d$. \end{sloppypar}
\end{prop}

\begin{proof} $(Case~A)$ By proposition \ref{huna}, we have the equality \begin{center}{\scriptsize $H^1(\beta , \fA)= U^1 (\fB _{\beta _0}) U^{[\frac{-\nu _{\fA} (c_0)}{2}]+1}(\fB _{\beta _1} )\ldots U^{[\frac{-\nu _{\fA} (c_{i-1})}{2}]+1}(\fB _{\beta _i}) \ldots U^{[\frac{-\nu _{\fA} (c_{s-1})}{2}]+1}(\fB _{\beta _s})$}.\end{center}
By definition of $K_+^d (\overrightarrow{G}, y , r , \overrightarrow{\boldsymbol{\phi} })$, and because of $d=s$, we have the equality

\begin{center} {\footnotesize $K_+^d(\overrightarrow{G}, y , r , \overrightarrow{\boldsymbol{\phi} })= G^0(F)_{y,0+} G^1(F)_{y,\mathbf{s}_0+} \cdots G^i (F) _{y, \mathbf{s} _{i-1}+} \ldots G^s(F)_{y,\mathbf{s}_{s-1}+} $}.\end{center}
The required statement is now a formal consequence of \ref{filaeg}.

$(Case~B)$ By proposition \ref{hunb}, we have the equality

\centerline{ {\tiny { $H^1(\beta , \fA)=  U^1 (\fB _{\beta _0}) U^{[\frac{-\nu _{\fA} (c_0)}{2}]+1}(\fB _{\beta _1} )\ldots U^{[\frac{-\nu _{\fA} (c_{i-1})}{2}]+1}(\fB _{\beta _i}) \ldots U^{[\frac{-\nu _{\fA} (c_{s-1})}{2}]+1}(\fB _{\beta _s}) U^{[\frac{-\nu _{\fA} (c_{s})}{2}]+1}(\fA)$.}}}By definition of $K_+^d (\overrightarrow{G}, y , r , \overrightarrow{\boldsymbol{\phi} })$, and because of $d=s+1$, we have the equality 

\begin{center}{\scriptsize  { $K_+^d(\overrightarrow{G}, y , r , \overrightarrow{\boldsymbol{\phi} })= G^0(F)_{y,0+} G^1(F)_{y,\mathbf{s}_0+} \cdots G^i (F) _{y, \mathbf{s} _{i-1}} \ldots G^s(F)_{y,\mathbf{s}_{s-1}+} G^d(F)_{y,\mathbf{s}_{s}+} $}}.\end{center}
The required statement is now a formal consequence of \ref{filbeg}.\\

\end{proof}

\begin{prop} \label{chapeau} In both $(Case~A)$ and $(Case~B)$, let $0 \leq i \leq d$ and let $\hat{\boldsymbol{\Phi}}_i$ be the character attached to $\boldsymbol{\Phi} _i$ (see Section \ref{yu}). Then
 \begin{center}

\begin{enumerate}[(i)]
\item $\hat{\boldsymbol{\Phi}}_i= \theta ^i$ for $0 \leq i \leq s$

\item $  {\displaystyle \prod_{i=0}^{d} \hat{\boldsymbol{\Phi}}_i }= \theta $

\end{enumerate} 

\end{center}
\end{prop} \begin{proof}
Recall that $\hat{\boldsymbol{\Phi} }_i $ is defined in 
\cite[Section 4]{YU}
 and also in Section \ref{yu} of this text. To prove $(i)$, we need first to study the decomposition $\mathfrak{g}= \mathfrak{g}^i \oplus \mathfrak{n} ^i $. In our situation, where $G = \underline{\Aut} _F( V)$, the Lie algebra $\mathfrak{g}$ is $\End_F(V)$ and the Lie algebra of $G^i$ denoted $\mathfrak{g}^i$ is $End_{F[\beta _i]} (V)$. The space $\mathfrak{g}^i$ is characterised by the fact that it is the maximal subspace of $\mathfrak{g}$, such that the adjoint action of the center ${Z}(G^i(F))$ of $G^i(F)$ is trivial.  By definition, $\mathfrak{n}^i$ is the sum of the other isotypic spaces for the adjoint action of $T^i(F)$ on $\mathfrak{g
}$. This implies that there is an integer $R_i$, such that each $n \in \mathfrak{n}^i$ is a finite sum 

\begin{center}
$n= \displaystyle \sum _{k=0} ^{R_i} n_k$
\end{center}
 such that for each $0 \leq k \leq R_i$, there is an element $t_k \in Z(G^i(F))$ and $\lambda _k\not = 1$, such that $\mathrm{ad}_{t_k} (n_k)= \lambda _k n_k$.
 We are now able to prove $(i)$ of the proposition \ref{chapeau}. 
If $x\in \mathfrak{g}$, then  let $x= \pi _{\mathfrak{g}^i} (x) + \pi _{\mathfrak{n}^i} (x)$ denote the decomposition of $x$ relatively to the decomposition $\mathfrak{g}= \mathfrak{g}^i \oplus \mathfrak{n} ^i $.
Let $0\leq i \leq s$. By definition (see Section \ref{yu}) $\hat{\boldsymbol{\Phi} }_i$ is the character of $K^d _+$ defined by

$\bullet $ $\hat{\boldsymbol{\Phi}}_i \mid _{G^i(F) \cap K^d _+} = {\boldsymbol{\Phi }_i} \mid _{G^i(F) \cap K^d _+}$

$\bullet $ $\hat{\boldsymbol{\Phi  }}_i \mid _{G(F)_{y,\mathbf{s}_i +} \cap K^d _+} ( 1+x) =\boldsymbol{\Phi} _i ( 1+ \pi _{\mathfrak{g}^i} (x))$.

Let us verify that it is equal to the character $\theta ^i$ defined in proposition \ref{facttam}.
First, note that the group $K_+^d $ is equal to the group $H^1(\beta , \fA)$ by Proposition \ref{hunkd}, and so it makes sense to compare $\hat{\boldsymbol{\Phi}}_i$ and $\theta ^i$.
The group $G^i (F) \cap K_+ ^d $ is equal to $B _{\beta _i} \cap H^1(\beta , \fA)$. Thus, the definitions of $\theta ^i$ given in {proposition  \ref{facttam}} shows that

 \begin{equation} \label{eqchapeau}\hat{\boldsymbol{\Phi}}_i \mid _{G^i(F) \cap K^d _+}  = \boldsymbol{\Phi } _i \mid _{G^i(F) \cap K^d _+}=\phi _i \circ {\det}_{B _{\beta _i}} \mid _{G^i(F) \cap K^d _+} = \theta ^i \mid _{G^i(F) \cap K^d _+}.
 \end{equation}
 It is enough to show that
 $\hat{\boldsymbol{\Phi }} _i \mid _{G(F)_{y,\mathbf{s}_i +} \cap K^d _+} $ is equal to $\theta ^i \mid _{G(F)_{y,\mathbf{s}_i +} \cap K^d _+ } $.  The group $G(F)_{y,\mathbf{s}_i +}  $ is equal to $  U^{[\frac{- \nu _{\fA} (c_i)}{2}]+1} (\fA) $.
Consequently, \begin{align*}
\hat{\boldsymbol{\Phi} }_i \mid _{G(F)_{y,\mathbf{s}_i +} \cap K^d _+} (1+x) &= \boldsymbol{\Phi} _i  \mid _{G(F)_{y,\mathbf{s}_i +} \cap K^d _+ }(1+\pi _{\mathfrak{g}^i} (x)) \\ {{\tiny \textit{(Because $1+\pi _{\mathfrak{g}^i} (x) \in G^i(F)$)}}}  &= \boldsymbol{\Phi} _i  \mid _{G(F)_{y,\mathbf{s}_i +} \cap K^d _+ \cap G^i(F)}(1+\pi _{\mathfrak{g}^i} (x)) \\
{{\tiny \textit{(By eq. \eqref{eqchapeau} and equality of groups)}}}&= \theta ^i \mid_{H^1 (\beta , \fA) \cap B_{\beta _i} \cap U ^{[\frac{- \nu _{\fA} (c_i)}{2}]+1}(\fA)}(1+\pi _{\mathfrak{g}^i} (x)) \\
\tiny \textit{(By def. of $\theta ^i$ on $H^1(\beta , \fA) \cap U ^{[\frac{- \nu _{\fA} (c_i)}{2}]+1}) $}  &=\psi \circ \Tr _{A/F} (c_i \pi _{\mathfrak{g}^i}(x)) .
\end{align*}
Let us now compute $\Tr _{A/F} (c_i \pi _{\mathfrak{g}^i}(x)) $.  We have the equalities 
\begin{center}$
\Tr  (c_i x) = \Tr  ( c_i (\pi _{\mathfrak{g}^i} (x) +\pi _{\mathfrak{n}^i} (x)) )= \Tr  (c_i \pi _{\mathfrak{g}^i }(x)) + \Tr  (c_i \pi _{\mathfrak{n}^i} (x) )$.
\end{center}
Let us compute $ \Tr  (c_i \pi _{\mathfrak{n}^i} (x) )$. Because  $ \pi _{\mathfrak{n}^i} (x) \in \mathfrak{n}^i$, there is an integer $R_i$   such that $\pi _{\mathfrak{n}^i} (x)$ is a finite sum 

\begin{center}
$\pi _{\mathfrak{n}^i} (x)= \displaystyle \sum _{k=0} ^{R_i} n_k$
\end{center}
 such that for each $0 \leq k \leq R_i$, there is an element $t_k \in Z(G^i(F))$ and $\lambda _k\not = 1$ such that $\mathrm{ad}_{t_k} (n_k)= \lambda _k n_k$.
 We have \begin{center}$ \Tr  (c_i \pi _{\mathfrak{n}^i} (x) )=  \Tr  (c_i \displaystyle \sum_{k=0} ^{R_i} n_k) = \displaystyle \sum_{k=0} ^{R_i} \Tr  (c_i n_k )$.\end{center}
Fix $ 0 \leq k \leq R_i$. The element $t_k$ commutes with $c_i$. Consequently, $t c_i n_k t^{-1} = c_i tn_k t^{-1} = c_i \lambda  n_k$. So
\begin{center}$
\Tr (c_i  n_k  ) = \Tr  (t c_i  n_k t^{-1}) = \lambda \Tr (c_i n_k)$
\end{center}
This implies that \begin{center} $\Tr(c_i  n_k)  =0$, \end{center} and so \begin{center} $ \Tr  (c_i \pi _{\mathfrak{n}^i} (x) )=0$.
\end{center}
Thus, the equality \begin{center}$\Tr _{A/F} (c_i \pi _{\mathfrak{g}^i}(x)) =\Tr _{A/F} ( c_i x) $ \end{center} holds.
Consequently, \begin{center} $\hat{\boldsymbol{\Phi}} _i \mid _{G(F)_{y,\mathbf{s}_i +} \cap K^d _+} (1+x) = \psi \circ \Tr _{A/F} ( c_i x) = \psi _{c_i}= \theta ^i \mid _{G(F)_{y,\mathbf{s}_i +} \cap K^d _+}$,\end{center} as required. This concludes the proof of $(i)$ of Proposition \ref{chapeau}.
The proof of $(ii)$ is now easy because $\theta = \displaystyle \prod _{i=0} ^s  \theta ^i$ and also because in  $(Case ~A)$, $d=s$, and in $(Case~ B)$, $d=s+1$ and $\hat{\boldsymbol{\Phi}}_d=1$.
This ends the proof of Proposition \ref{chapeau}.
\end{proof}

\section{Extensions and main theorem for the direction: From Bushnell-Kutzko's construction to Yu's construction} \label{extensionn}

In this section, we keep the notations of Sections \ref{tameca} and \ref{sectgecha}. In particular, we have fixed a tame simple stratum $[\fA , n , r , \beta ]$ and a chosen defining sequence $\{[\fA , n , r_i , \beta _i ]$  , $0 \leq i \leq s\}$, such that $F[\beta _{i+1}] \subsetneq F[\beta _i] $ for all $0 \leq i \leq s-1$. We have also fixed a simple character $\theta \in \mathcal{C}(\fA , 0 ,  \beta)$. We have distinguished two cases: in the first , $(Case~ A)$ occurs when $\beta _s \in F$. In this case, we have put $d=s$. In the second case,  $(Case~ B)$, we have put $d=s+1$. In both case, we have introduced part of a Yu datum $(\overrightarrow{G}, y , r , \overrightarrow{\boldsymbol{\Phi} })$. We have also proved some results relative to these objects. 
In this section, we are going to show that the representation $^{\circ} \lambda (\overrightarrow{G}, y , r , \overrightarrow{\boldsymbol{\Phi} })$ is a $\beta$-extension of $\theta$. Then, given a cuspidal
 representation $\sigma$ of  $U^0 (\fB _{\beta _0}) / U^1(\fB _{\beta _0})$ and $\Lambda$ an extension to $E^{\times} J^0 (\beta , \fA)$ of $\kappa \otimes \sigma $, we are going to show that there exists $\rho$ such that $ \Lambda = \rho _d (\overrightarrow{G}, y , r , \overrightarrow{\boldsymbol{\Phi} } , \rho)$.

\begin{prop} \label{rondkdj0} \begin{sloppypar} In both $(Case~ A)$ and $(Case~ B)$, the group  ${^{\circ} K ^d (\overrightarrow{G}, y , r , \overrightarrow{\boldsymbol{\Phi} }) }$ is equal to $ J^0 (\beta , \fA)$. \end{sloppypar}

\end{prop}

\begin{proof} This proposition is similar to Proposition \ref{hunkd}, and the proof of it adapts trivially. 
\end{proof}

\begin{prop} \label{circlambabetaext} In both $(Case~ A)$ and $(Case~ B)$, the representation \begin{center}$^{\circ} \lambda (\overrightarrow{G}, y , r , \overrightarrow{\boldsymbol{\Phi} })$\end{center} of $^{\circ} K ^d $ is a $\beta $-extension of $\theta$.

\end{prop}

\begin{proof} Let us verify that $^{\circ} \lambda=^{\circ} \lambda(\overrightarrow{G}, y , r , \overrightarrow{\boldsymbol{\Phi} }) $ satisfies the criterion given in Proposition \ref{kappakappa}.

\begin{enumerate}[(a)]
\item  The representation $^{\circ} \lambda $ is equal to ${^{\circ}\kappa _0} \otimes \ldots \otimes {^{\circ}\kappa _d}$ (see section \ref{yu}). By construction of $\kappa _i , 0 \leq i \leq d$, the representation ${^{\circ}\kappa _i}$ contains $\hat {\boldsymbol{\Phi} }_i $ (see \cite[3.27]{Hamu}). Consequently $^{\circ} \lambda $ contains $\hat {\boldsymbol{\Phi}} _0 \otimes \ldots \otimes \hat {\boldsymbol{\Phi} }_d  $. Thus $^{\circ} \lambda $ contains $\theta$ by \ref{chapeau}.

\item  Again, $^{\circ} \lambda = ^{\circ} \kappa _{0} \otimes \ldots \otimes ^{\circ} \kappa _{d}$. Thus, it is enough to show that $G^0(F)$ is contained in $I_{G(F)} (^{\circ} \kappa _i )$ for $0\leq i \leq d$. Theorem 14.2 of \cite{YU}, which is satisfied here, implies that $G^0(F)$ is contained in $I_{G(F)} ( {\boldsymbol{\Phi}_i}' \mid _{^{\circ} K^i})$. However, $^{\circ}\kappa _i$ is an inflation of  ${\boldsymbol{\Phi}_i}'  \mid _{^{\circ} K^i}$ (see definition \ref{kappaijijiji}). Consequently $ I_{G(F)} ( {\boldsymbol{\Phi}_i}'\mid _{^{\circ} K^i}) \subset I_{G(F} ( ^{\circ} \kappa _i )$. Consequently, $G^0(F) \subset I_{G(F)} (^{\circ} \kappa _i )$ as required.

\item \begin{sloppypar} The representation $^{\circ} \lambda $ is equal to ${^{\circ} \kappa _0} \otimes \ldots \otimes {^{\circ} \kappa _i} \otimes \ldots \otimes{ ^{\circ} \kappa _d}$. For ${0 \leq i \leq d-1}$ the dimension of $^{\circ} \kappa _i $ is $[J^{i+1} : J^{i+1}_+ ]^{\frac{1}{2}}$. The representation $^{\circ} \kappa _d$ is one-dimensional, so it is enough to show that ${\displaystyle \prod _{i=1}^{d} [J^{i+1} : J^{i+1}_+]} = [J^1(\beta , \fA) : H^1(\beta , \fA)]$.
The group $ J^1 (\beta , \fA ) $ is equal to $G^0(F) _{y,0+} G^1(F)_{y, \mathbf{s}_0} \ldots G^d(F)_{y, \mathbf{s}_{d-1}}$, thus this is also equal to $G^0(F) _{y,0+} J^1 \ldots J^d$. 
The group $ H^1 (\beta , \fA ) $ is equal to $G^0(F) _{y,0+} G^1(F)_{y, \mathbf{s}_0+} \ldots G^d(F)_{y, \mathbf{s}_{d-1}+}$, thus this is also equal to $G^0(F) _{y,0+} J^1_+ \ldots J^d_+$. 
Because ${G^0(F) _{y,0+} J^1 \ldots J^d / G^0(F) _{y,0+} J^1_+ \ldots J^d_+} \simeq {J^1 \ldots J^d /  J^1_+ \ldots J^d_+ }$ it is enough to show that  
 ${{\displaystyle \prod _{i=1}^{d} [J^{i} : J^{i}_+]} = [ J^1 \ldots J^d  :  J^1_+ \ldots J^d_+ ]}$. Let us prove this by induction on $d$. If $d=1$, this is trivial. Let us assume that this is true for $d-1$. It is now enough to show that $[J^d : J^d_+ ] = \frac{[J^1 \ldots J^d : J^1 _+ \ldots J^d _+ ]} {[J^1 \ldots J^{d-1} : J^1 _+ \ldots J^{d-1} _+ ]}$.
 \end{sloppypar}
 The following fact will be useful.

\textit{Fact: Let $G' \subset G$ be groups and let $H $ be a normal subgroup of $G$. Let $\iota $ be the injective morphism of group $ G' / ( G' \cap H ) \hookrightarrow  G/ H $. As $G$-set, $G/HG'$ and $ (G/H) /\iota (G' / (G' \cap H )) $ are isomorphic.}
 \begin{sloppypar}Because $J^1 _+ \ldots J^d _+$ is a normal subgroup of $ J^1 \ldots J^d$, we can apply the previous fact to $G= J^1 \ldots J^d$, $G' = J^1 \ldots J^{d-1}$ , $H= J^1 _+ \ldots J^d _+$. Using the fact that  $H \cap G' = J^1 _+ \ldots J^{d-1} _+$, we deduce that, as $J^1 \ldots J^d $-sets, $J^1 \ldots J^d / J^1 \ldots J^{d-1} J^d_+$ and $ (J^1 \ldots J^d / J^1 _+ \ldots J^d _+ ) / \iota (J^1 \ldots J^{d-1} / J^1 _+ \ldots J^{d-1} _+ )$ are isomorphic. Let $X$ be this $J^1 \ldots J^d$-set. 
 The set $X$ is a fortiori a $J^d$-set. The group $J^d$ acts transitively on $X=J^1 \ldots J^d / J^1 \ldots J^{d-1} J^d_+$, and the stabiliser of $ (J^1 \ldots J^{d-1} J^d_+) \in J^1 \ldots J^d / J^1 \ldots J^{d-1} J^d_+$ is  $J^1 \ldots J^{d-1} J^d_+ \cap J^d $. The group $J^1 \ldots J^{d-1} J^d_+ \cap J^d $ is equal to $J^d_+$. Consequently, \end{sloppypar} \begin{center}$[J^d : J^d_+ ] = \#(X) = \frac{[J^1 \ldots J^d : J^1 _+ \ldots J^d _+ ]} {[J^1 \ldots J^{d-1} : J^1 _+ \ldots J^{d-1} _+ ]},$\end{center} as required.
 This ends the proof of the proposition.

\end{enumerate}
\end{proof}

The following theorem is the outcome of Sections \ref{tameca} and \ref{sectgecha}. It shows that given a Bushnell-Kutzko datum, there exists a Yu datum $(\overrightarrow{G}, y , r , \overrightarrow{\boldsymbol{\Phi} } , \rho )$,  such that $\Lambda = \rho _d (\overrightarrow{G}, y , r , \overrightarrow{\boldsymbol{\Phi} } , \rho).$ The objects $(\overrightarrow{G}, y , r , \overrightarrow{\boldsymbol{\Phi} } , \rho) $ are given explicitly in terms of the Bushnell-Kutzko datum. 

\begin{theo} \label{bigtheofin}Let $V$ be an $N$-dimensional $F$-vector space. Let $A$ denote $\End _F (V)$ and let $\mathrm{G}$ denote $A^{\times}\simeq GL_N(F)$. The following assertions hold.

\begin{enumerate}[(I)]
\item Let $([\fA , n , r , \beta ] , \theta , \sigma , \kappa , \Lambda ) $ be a tame Bushnell-Kutzko datum of type $(a)$ in $A$. Let $\{[\fA , n , r , \beta _i ]$, $0 \leq i \leq s \}$ be a defining sequence such that $F[\beta _i] \subsetneq F [\beta _{i+1} ]$ for $0 \leq i \leq s-1$, as in Section \ref{tameca}.

$\bullet ~~ (Case~A)$ If $\beta _s $ is in $F$, put $d=s$, and $G^i = \Res _{F[\beta _i] /F} \underline{\Aut} _{F[\beta _i]} (V)$ for $0 \leq i \leq s$.
 Put $\overrightarrow{G}=(G^0 , \ldots , G^s)$. Choose a factorisation ${\displaystyle \theta = \prod _{i=0} ^s \theta ^i}$ as in Theorem \ref{facttam}, Corollary \ref{tamcoa}.
 Let $\boldsymbol{\Phi} _i $ , $0 \leq i \leq s$, be the associated characters as in Definition \ref{PHIIa}. Put $\overrightarrow{\boldsymbol{\Phi}}= (\boldsymbol{\Phi}_0 , \ldots , \boldsymbol{\Phi} _s  )$. Let $y \in \BT^E (G^0 , F )$
 and $\overrightarrow{\mathbf{r}}$ as in Proposition \ref{rivucia}.
  Then, there exists a representation $\rho$ of $G^0 _{[y]}$ such that $(\overrightarrow{G} , y , \overrightarrow{r} , \overrightarrow{\boldsymbol{\Phi}}, \rho )$ is a Yu datum and $\rho _d (\overrightarrow{G} , y , \overrightarrow{r} , \overrightarrow{\boldsymbol{\Phi}}, \rho )$ is isomorphic to $\Lambda$ (see section \ref{yu}).

$\bullet ~~(Case~B)$ If $\beta _s \not \in F$, put $d=s+1$, and $G^i = \Res _{F[\beta _i] /F} \underline{\Aut} _{F[\beta _i]} (V)$ for $0 \leq i \leq s$. Put also $G^d = \underline{\Aut}_F (V)$. 
Put $\overrightarrow{G}= (G^0 , \ldots , G^s , G^d)$. Choose a factorisation ${\displaystyle \theta = \prod _{i=0} ^s \theta ^i}$ as in \ref{facttam}, \ref{tamcob}. Let $\boldsymbol{\Phi} _i$, $0 \leq i \leq s$ be the associated characters and let $\boldsymbol{\Phi}_d$ be the trivial character as in \ref{PHIIb}.
 Put $\overrightarrow{\boldsymbol{\Phi}}= (\boldsymbol{\Phi}_0 , \ldots , \boldsymbol{\Phi} _s , \boldsymbol{\Phi} _d )$.  Let $y \in BT ^e (G^0 , F)$ and $\overrightarrow{\mathbf{r}}$ as in Proposition \ref{rivucib}.   Then, there exists a representation $\rho$ of $G^0 _{[y]}$ such that $(\overrightarrow{G} , y , \overrightarrow{r} , \overrightarrow{\boldsymbol{\Phi}}, \rho )$ is a Yu datum and $\rho _d (\overrightarrow{G} , y , \overrightarrow{r} , \overrightarrow{\boldsymbol{\Phi}}, \rho )$ is isomorphic to $\Lambda$ (see Section \ref{yu}).

\item Let $(\fA , \sigma , \Lambda )$ be a Bushnell-Kutzko datum of type (b). Put $d=0$, $G^0 = \underline{\Aut }_F (V)$ and $\overrightarrow{G}= (G^0)$. Put  $ \mathbf{r}_0 = 0$ and $\overrightarrow{\mathbf{r}}=  (\mathbf{r}_0 )$. Let $y \in BT ^e (G^0 , F) $ such that $\fA ^{\times} = G^0(F) _{y}$. Put $\boldsymbol{\Phi} _0 = 1$ and $\overrightarrow{\boldsymbol{\Phi}}= (\boldsymbol{\Phi} _0)$. Let $\rho $ be $\Lambda$.
Then $(\overrightarrow{G} , y , \overrightarrow{r} , \overrightarrow{\boldsymbol{\Phi}}, \rho )$ is a Yu datum and $\rho _d (\overrightarrow{G} , y , \overrightarrow{r} , \overrightarrow{\boldsymbol{\Phi}}, \rho )$ is isomorphic to $\Lambda$.
\end{enumerate}
\end{theo}

\begin{proof}

\begin{enumerate}[(I)]
\item As usual, put $E= F[\beta]$. Let $\rho'$ be an arbitrary extension of $\sigma $ to $G^0(F)_{[y]}$. Then, the compact induction of $\rho '$ to $G^0(F)$ is irreducible and supercuspidal and so $(\overrightarrow{G} , y , \overrightarrow{r} , \overrightarrow{\boldsymbol{\Phi}}, \rho' )$ is a Yu datum. We are going to show that there exists a character $\chi$ of $G^0(F)_{[y]}$  such that $ (\overrightarrow{G} , y , \overrightarrow{r} , \overrightarrow{\boldsymbol{\Phi}}, \rho '\otimes \chi)$ is a Yu datum such that $\rho _d (\overrightarrow{G} , y , \overrightarrow{r} , \overrightarrow{\boldsymbol{\Phi}}, \rho \otimes \chi)$ is isomorphic to $\Lambda$.
The representation $^{\circ} \lambda  (\overrightarrow{G} , y , \overrightarrow{r} , \overrightarrow{\boldsymbol{\Phi}} ) $ is a $\beta$-extension of $\theta$ by Proposition \ref{circlambabetaext}. Consequently, by \ref{kappakappa}, there exists  a character \begin{center}$\xi': U^0(\fB _{\beta _0})/ U^1 (\fB _{\beta _1})\simeq J^0 (\beta , \fA ) / J^1 (\beta , \fA ) \to \C ^{\times}$ \end{center} of the form $\alpha' \circ \det $ with ${\alpha' : U^0 (\mathfrak{o} _E ) / U^1 (\mathfrak{o} _E ) \to \C ^{\times}}$ and such that $\kappa $ is isomorphic to $ ^{\circ} \lambda \otimes \xi '$. Let $\chi '$ be an extension of $\xi '$ to $E^{\times} U^0 ( \fB _{\beta _0} )= G^0 (F)_{[y]}$. The compact induction of $\rho ' \otimes \chi '$ to $G^0(F)$ is irreducible and supercuspidal, and so $(\overrightarrow{G} , y , \overrightarrow{r} , \overrightarrow{\boldsymbol{\Phi}}, \rho ' \otimes \chi ' )$ is a Yu datum. The representation $^{\circ} \rho _d (\overrightarrow{G} , y , \overrightarrow{r} , \overrightarrow{\boldsymbol{\Phi}}, \rho ' \otimes \chi ' )$ is equal to $ \sigma \otimes \xi ' \otimes ^{\circ} \lambda (\overrightarrow{G} , y , \overrightarrow{r} , \overrightarrow{\boldsymbol{\Phi}} )$. Thus, it is isomorphic to $\sigma \otimes \kappa$. Consequently, $\rho _d (\overrightarrow{G} , y , \overrightarrow{r} , \overrightarrow{\boldsymbol{\Phi}}, \rho ' \otimes \chi ' )$ and $\Lambda$ are two extensions of $\sigma \otimes \kappa$. This implies that there exists a character \begin{center}$\chi '' : E^{\times} J^0 (\beta , \fA ) \to E^{\times} J^0 (\beta , \fA ) / J^0 (\beta , \fA) \simeq G^0(F)_{[y]} / G^0(F)_{y} \to \C ^{\times}$,\end{center} such that $\rho _d (\overrightarrow{G} , y , \overrightarrow{r} , \overrightarrow{\boldsymbol{\Phi}}, \rho ' \otimes \chi ' ) \otimes \chi '' $ is isomorphic to $\Lambda$. Seeing $\chi '' $ as a character of $G^0(F)_{[y]}$, the compact induction  of the representation $\rho ' \otimes \chi ' \otimes \chi ''$ to $G^0(F)$ is irreducible and supercuspidal, and  $\rho _d (\overrightarrow{G} , y , \overrightarrow{r} , \overrightarrow{\boldsymbol{\Phi}}, \rho ' \otimes \chi ' ) \otimes \chi '' $ is isomorphic to $\Lambda$.
 The assertion $(I)$ follows, putting $\rho = \rho ' \otimes \chi ' \otimes \chi ''$.

\item In this case, the representation $\rho _d$ is $\rho$, and there is nothing to prove.

\end{enumerate}

\end{proof}

\begin{rema} \label{remabushab}

In this remark, we briefly explain the relationship between $(Case~ A)$ and $(Case~ B)$. Previously, we have treated separately both cases. Of course, $(Case ~A)$ and $(Case ~B)$ have no fundamental difference. Here, we explain that it is possible to reduce one to the other: $(Case~A)$ is $(Case~B)$ tensored by a character in the following sense. Let $[\fA, n , 0, \beta ] $ be a tame simple stratum and let $[\fA , n , r_i , \beta _i ] $, $0 \leq i \leq s$ be a defining sequence such that $\beta _s \in F$, so that it is $(Case~A)$. By \cite[Appendix, Lemma]{BKa}, one sees that $[\fA , n_1 , 0 ,\beta - \beta _s ]$ is a simple stratum ($n_1= -\nu_{\fA}(\beta - \beta _s$) and that $[\fA , n_1 , r_i , \beta_i - \beta _s ]$, $0 \leq i \leq s-1$ is a defining sequence of it such that $\beta_{s-1} - \beta _s$ is not contained in $F$, so that now $(Case~B)$ appears. Using the stratum $[\fA , n_1 , 0 , \beta - \beta _s ]$ and the defining sequence  $[\fA , n_1 , r_i , \beta_i - \beta _s ]$, $0 \leq i \leq s-1$, by $(Case~ B)$, one can attach the integer $d=s(=(s-1)+1)$ and a tamely ramified twisted Levi sequence $\overrightarrow{G}$.  By \cite[Appendix, Lemma]{BKa}, we have $\mathcal{C}(\fA , 0 , \beta -\beta _s )= \mathcal{C} (\fA , 0 , \beta ) . \chi \circ \det$ where $\chi$ is a certain character. So a character of $\mathcal{C}(\fA , 0 , \beta  )$ gives a character of  $\mathcal{C}(\fA , 0 , \beta -\beta _s )$ that we can factorise  using $(Case~ B)$, so we can define characters $\Phi _0 , \ldots , \Phi _{s-1}$ and 
$\chi \circ \det$ gives a character $\Phi_d$. Consequently, one sees that $(Case~A)$ is $(Case~B)$ tensored by a character.
\end{rema}

\section{Reversing arguments: from Yu's construction to Bushnell-Kutzko's construction} \label{YUBK}

Let $V$ be a finite dimensional $F$-vector space and $G$ be the connected reductive $F$-group scheme $ \underline{ \Aut } _F (V)$. Recall that we have fixed a uniformiser $\pi _F$ and an additive character $\psi$ of $F$ with conductor $\fp _F$. In this section, we start with a Yu datum $(\overrightarrow{G},y,\overrightarrow{R},\rho,\overrightarrow{\boldsymbol{\Phi} })$ in $G$, and associate step-by-step and \guillemotleft explicitly\guillemotright$~$ a Bushnell-Kutzko datum $([\fA,n,0,\beta],\theta,\kappa,\sigma,\Lambda )$ such that $\rho _d (\overrightarrow{G},y,\overrightarrow{R},\rho,\overrightarrow{\boldsymbol{\Phi} })= \Lambda $. Many steps are remarks that the arguments previously explained in this paper for the direction $BK \longrightarrow YU$  can be reversed. We give the details for the other steps.

Using a similar argument as in Remark \ref{remabushab}, we reduce to the case $\boldsymbol{\Phi} _d =1 $. We also assume that $d>0$ because the case $d=0$ presents no difficulty and should be treated separately.
So let us fix a Yu datum $(\overrightarrow{G},y,\overrightarrow{R},\rho,\overrightarrow{\boldsymbol{\Phi} })$ such that $d >0 $ and $\boldsymbol{ \Phi }_d =1 $. Because $\overrightarrow{G}= (G^0,\dots,G^d)$ is an anisotropic tamely ramified twisted Levi sequence, there exists a tower of tamely ramified fields extensions (included in $\End _F(V)$) $E_0 \supset \cdots \supset E_d =F$ such that $G^i = \Res _{E_i /F} \underline{ \Aut } _{E_i} (V)$. Using Broussous-Lemaire, the point $y$ in the building of $G^0$ gives us $\fo _{E_i}$-hereditary orders $\fB _i $ in $B_i:=\End_{E_i}(V)$ for $0 \leq i \leq d $. We notably have $\fB _j \cap \End _{E_i} (V) = \fB _i$ if $ 0 \leq i \leq j \leq d$. We denote $\fB _d $ by $\fA$. We set $s=d-1$. Because the extensions $E_i/F$ are tamely ramified for $0\leq i \leq s $, $\psi _{E_i}:=\psi \circ \Tr_{E_i/F}$ is a character of $E_i$ with conductor $\mathfrak{p}_{E_i}$ (see \cite[1.3.8 (ii)]{BK}). We have the following proposition:

\begin{prop} \label{eximi}For $0 \leq i \leq s $, the following assertions hold:

\begin{enumerate} \item There exists a unique smooth character $\phi _i$ of $E_i ^{\times}$ such that $\boldsymbol{\Phi} _i = \phi _i \circ \det _{B_i}$
\item There exists an element $c_i \in E_i$, which is minimal relatively to the field extension $E_i/E_{i+1}$ and such that the following (non-independent) assertions hold \begin{enumerate}

\item $ \phi _i \mid _{U^{[\frac{- \nu _{E_i}(c_i)}{2}]+1}(\mathfrak{o}_{E_i})}( 1+x)=\psi_{E_i} (c_i x) $

 \item $\boldsymbol{\Phi} _i \mid _{U^{[\frac{- \nu _{\fA}(c_i)}{2}]+1}(\fB _i)}( 1+x)=\psi _{E_i} (c_i ({\det}_{B_i/E_i}(1+x)-1))$

 \item $\boldsymbol{\Phi} _i \mid _{U^{[\frac{- \nu _{\fA}(c_i)}{2}]+1}(\fB _i)}( 1+x)=\psi _{E_i} \circ \Tr _{B_i/E_i}(c_i x)$ 

 \item $\boldsymbol{\Phi} _i \mid _{U^{[\frac{- \nu _{\fA}(c_i)}{2}]+1}(\fB _i)}( 1+x)=\psi  \circ \Tr _{A/F}(c_i x)$ 
   \item $\ord(c_i)=-R_i$ .\end{enumerate}

\end{enumerate}
\end{prop}

\begin{proof} We need the following Lemma.
\begin{lemm} \label{intoproof}Let $\mathrm{K}$ be a non-Archimedean local field and $\psi _{\mathrm{K}}$ be an additive character of $\mathrm{K}$ with conductor $\mathfrak{p}_{\mathrm{K}}$. Let $\fA_{\mathrm{K}}$ be a hereditary order in $\mathrm{A}:=\End _{\mathrm{K}}(\mathrm{V})$, where $\mathrm{V}$ is a finite dimensional vector space over $\mathrm{K}$. Let $c$ be an element in $\mathrm{K} \cap \fP _{\mathrm{K}}^{-n}$ where $n=-\nu_{\fA_{\mathrm{K}}}(c)$. Then

\begin{enumerate} \item $n$ is divisible by $e(\fA_{\mathrm{K}} , \mathfrak{o}_{\mathrm{K}} )$, and $n=e(\fA _{\mathrm{K}}, \mathfrak{o}_{\mathrm{K}} ) n_0$ where $n_0 =- \nu _{{\mathrm{K}}} (c)$
\item ${\det}_{\mathrm{A}/\mathrm{K}} ( U^{[\frac{n}{2}]+1}(\fA_{\mathrm{K}})) = U ^{ [\frac{n_0}{2}]+1} (\mathfrak{o}_{\mathrm{K}})$
\item $
 \psi  _{\mathrm{K}} \circ \Tr _{\mathrm{A}/\mathrm{K}} (cx) 
= \psi _ {\mathrm{K}} ( c ({\det}_{\mathrm{A}/\mathrm{K}} (1+x)-1)) ~~~~~\forall ~1+x \in  U^{[\frac{n}{2}]+1}(\fA _{\mathrm{K}}).
$
\end{enumerate}
\end{lemm}

\begin{proof} The first assertion is a direct consequence of Lemma \ref{valval}. The second assertion is \cite[last line of page 126]{BK}. The third assertion is the formula \cite[line 33 page 98]{BK} with $E=\mathrm{K}$. Moreover, as in \cite[page 129]{BK} we suggest the reader to see \cite{bushandfro} for detailed computations in this area.\end{proof}
Let us prove the proposition. The first assertion is well known: a character of $GL_{E_i}(V)(=\Aut _{E_i}(V))$ factors through $GL_{E_i}(V)/SL_{E_i}(V)$ since $SL_{E_i}(V)$ is the derived group of $GL_{E_i}(V)$. Now, $\det _{B_i}$ provides an isomorphism $GL_{E_i}(V)/SL_{E_i}(V) \overset{\det _{B_i}}{\simeq}{ E_i ^{\times}}$. The character $\phi _i $ is the composition $E_i ^{\times} \overset{\det _{B_i}}{\simeq} GL_{E_i}(V)/SL_{E_i}(V) \overset{\boldsymbol{\Phi} _i}{\to} \C ^{\times}$.
Let us now prove the second assertion. Consider the character $\phi _i$ of the first assertion. There exists $n_0$ such that $\phi _i \mid _{U^{n _0 +1}(\mathfrak{o}_{E_i})} =1$ and $\phi _i \mid _{U^{n _0 }(\mathfrak{o}_{E_i})} \neq1$. Using the well-known fact (see \cite[page 22]{BK}) that the characters of $U^{[{\frac{n_0}{2}}] +1}(\mathfrak{o}_{E_i})/U^{n _0 +1}(\mathfrak{o}_{E_i})$ are in bijection with $\mathfrak{p}_{E_i}^{-n_0} / \mathfrak{p}_{E_i}^{-[\frac{n _0}{2}]}$, there exists an element $c_i \in E_i$ such that $\nu _{E_i} (c_i) = - n _0 $ and such that $\phi _i \mid _{U^{[{\frac{n _0}{2}}] +1 }(\mathfrak{o}_{E_i})}(1+x) = \psi _{E_i} ( c_i x). $ So we found $c_i$, such that the first identity holds.
 Now Lemma \ref{intoproof} (taking $\mathrm{K}=E_i$, $c=c_i$, $\fA _{\mathrm{K}}= \mathfrak{B}_i$) shows that\[{\det} _{B_i/E_i}(U^{[\frac{- \nu _{\fA}(c_i)}{2}]+1}(\fB _i))\subset U^{[{\frac{n _0}{2}}] +1 }(\mathfrak{o}_{E_i}),\] and that the equalities 

 \begin{align*}  
 \boldsymbol{\Phi} _i \mid _{U^{[\frac{- \nu _{\fA}(c_i)}{2}]+1}(\fB _i)}( 1+x)&=\phi _i \circ {\det} _{B_i} (1+x) \\ 
 &= \psi _{E_i} (c_i ({\det}_{B_i/E_i}(1+x)-1)) \\
 &=\psi _{E_i} \circ \Tr _{B_i /E_i} (c_i x) \\
 &=\psi  \circ \Tr _{A/F}(c_i x) ~~~~ {\footnotesize \text{since $\psi \circ \Tr _{A/F} \mid _{B_i} = \psi _{E_i} \circ \Tr _{B_i/ E_i}$}}
\end{align*} 

hold. So it is enough to prove that $c_i$ is minimal relatively to  $E_i/E_{i+1}$. Let us do it. The character $\boldsymbol{\Phi} _i $ is $G^{i+1}$-generic of depth $R_i$.  By definition of depth and the relation between $y$ and $\fB _i $ (Broussous-Lemaire \cite{Brou}), we have $G^i(F)_{y,R_i}=U^{-\nu _{\fA}(c_i)}(\fB _i) $ and $G^i(F)_{y,R_i+}=U^{-\nu _{\fA}(c_i)+1}(\fB _i) $. So we have \[\boldsymbol{\Phi} _i \mid _{G^i(F)_{y,R_i:R_i+}}(1+x)=\psi \circ \Tr _{A/F}(c_ix),\]and so $\boldsymbol{\Phi} _i \mid _{G^i(F)_{y,R_i:R_i+}}$ is realised  by the element $X^*_{c_i}:x \mapsto \Tr_{A/F}(c_ix)\in \Lie ^* (Z(G^i))_{-R_i} $ in the sense of \cite[§5 §9]{YU}. Moreover, because of $\boldsymbol{\Phi} _i$  is generic of depth $R_i$,  $\boldsymbol{\Phi} _i \mid _{G^i(F)_{y,R_i:R_i+}}$ is also realised by an element $X^*$ which is $G^{i+1}$-generic of depth $R_i$. Let us show that this implies that the element  $X^*_{c_i}\in \Lie ^* (Z(G^i))_{-R_i} $ is $G^{i+1}$-generic of depth $R_i$. The elements $X^*$ and $X^*_{c_i}$ both realise $\boldsymbol{\Phi} _i \mid _{G^i(F)_{y,R_i:R_i+}}$, so they are equal modulo $\mathfrak{g}^i(F)_{y,R_i}^{\bullet}={\mathfrak{g}^i}^*(F)_{y,(-R_i)+}$. So we have $X^*_{c_i} -X^* \in {\mathfrak{g}^i}^*(F)_{y,(-R_i)+} \cap \Lie ^* (Z(G^i))_{y,-R_i}$. Thus, $Y^*:=X^*_{c_i} -X^* \in \Lie ^* (Z(G^i))_{y,(-R_i)+}\subset$\footnote{$\Lie ^* (Z(G^i))\overset{\sim}{=}\Hom_F(E_i,F)\overset{\sim}{=}\{X^*_{\gamma}:E_i\to F,x \mapsto \Tr(\gamma x)\mid \gamma \in E_i \} $ }$ \Lie ^* (Z(G^i))$. So there exists $\gamma$ in $E_i$ such that $Y^*=(x\mapsto \gamma x)$ with $\ord (\gamma) > -R_i $. Now $\ord(Y^* (H_a))>- R_i$ for all $a \in\footnote{We fix a $T$} \Phi (G^{i+1},T,\overline{F})\setminus \Phi (G^{i},T,\overline{F})$, by the explicit formula given by Proposition \ref{xycomput}. So for all $a \in \Phi (G^{i+1},T,\overline{F})\setminus \Phi (G^{i},T,\overline{F})$, $\ord(X^*_{c_i}(H_a))=-R_i$, and thus $X^*_{c_i}$ is a generic element of depth $R_i$, in particular $\ord(c_i)=-R_i$. So by Theorem \ref{minimalgenericelemenx}, $c_i$ is minimal relatively to the extension $E_i/E_{i+1}$.\end{proof}
Fix such elements $c_i$ for $ 0 \leq i \leq s$.  Put \begin{align*} &\beta _s = c_s \\
&\beta _{s-1} = c_{s-1} +c_s \\
& ~~~~\vdots \\
&\beta _i = \displaystyle \sum _{k=i}^s c_i  ~~\text{for} ~s\geq i \geq 0\\
& ~~~~\vdots \\
& \beta _{0} = c_0 +c_1 + \ldots + c_s,
\end{align*} 
so that $\beta _i - \beta _{i+1} = c_i $ for $ 0 \leq i \leq s-1 $.
Put also $r_0=0$ and for $1\leq i \leq d$ put $r_i = - \nu _{\fA} (c_{i-1})$. Finally, put $\beta = \beta _0$, $E=E_0$ and $n=r_d$.
\begin{prop} The following assertions hold: \begin{enumerate}
\item For $0 \leq i \leq s$, $[\fA,n,r_i,\beta_i]$ is a simple stratum in $A=\End_F (V)$, moreover $k_0 ( \beta _i , \fA ) =- r_{i+1}$.
\item The sequence $([\fA,n,r_i,\beta _i ] , 0\leq i \leq s )$ is a defining sequence of the simple stratum $[\fA,n,0,\beta]$.
\end{enumerate}
\end{prop}
\begin{proof}\begin{enumerate}\item Let us prove this by decreasing induction on $i$, \guillemotleft $s$\guillemotright $~$ being the first step. Before starting let us note that because the orders $\fB_0 , \ldots ,\fB _i , \dots ,\fB _d $ come from a point $y $ in the building of $G^0=\underline{\Aut}_{E}(V)$, we have $E^{\times} \subset \mathfrak{K}(\fB_j)$ for all $0\leq j \leq d$. Let us start.
By Proposition \ref{eximi}, $c_s=\beta _s$ is minimal over $F$. Consequently $[\fA,n,n-1,\beta _s]$ is a simple stratum and $k_0(\beta_s,\fA) =-n =\nu_{\fA} (\beta_s) $ by \ref{minisimplealfalfa}. By Lemma \ref{valval}, $\nu_{\fA}(c_j)=\ord(c_j)e(\fA,\of)$ for all $j$, so $-\nu_{\fA}(c_{s-1})< -\nu_{\fA}(c_{s})$ by Proposition \ref{eximi} and axiom of Yu data on real numbers. We deduce that  $r_s=-\nu_{\fA}(c_{s-1})< -\nu_{\fA}(c_{s}) =n = - k_0 (\beta_s , \fA)$. Consequently, $[\fA,n,r_s,\beta_s]$ is simple. Now assume that $[\fA , n , r_{i}, \beta _{i}] $ is simple and that $k_0(\beta _{i} , \fA )= -r_{i+1}$. By Proposition \ref{eximi}, $c_{i-1}$ is minimal relatively to the field extension $E_{i-1}/E_i$ and so $[\fB _i , r_i , r_i -1 , c_{i-1}]$ is simple and $k_0 (c_{i-1},\fB_i)= -r_i$. By Proposition \ref{tssprop1}, we obtain that $[\fA,n,r_{i}-1,c_{i-1 }+\beta _{i}] $ is a simple stratum and that $k_0(\beta _{i-1}, \fA)=-r_i$. Because, as before, $r_{i-1} < r_i $, we deduce that $[\fA , n ,r_{i-1} , \beta _{i-1} ]$ is simple. This ends the proof of the assertion.
\item Let us check the definition. These stratum are simple: we have $0= r_0 < \ldots <r_s <n$, for  $0 \leq i \leq s-1$, we have $r_{i+1}= -k_0 (\beta _i , \fA)$, and for  $0 \leq i \leq s-1$, $[\fA , n  , r_{i+1} , \beta _i ]$ is equivalent to $[\fA , n  , r_{i+1} , \beta _{i+1} ]$ because  $\nu _{\fA} (\beta _i - \beta _{i+1} ) = \nu _{\fA} (c_i) = -r_{i+1}$. We have $k_0 (\beta _s , \fA ) =-n$. The derived stratum $[\fB _i , r_i , r_i -1 , c_{i-1} ]$ is simple for $1 \leq i \leq s $ because $c_{i-1}$ is minimal relatively to $E_{i-1}/E_i$.
\end{enumerate}
\end{proof}
The stratum $[\fA,n,0,\beta]$ obtained in the previous proposition is the first element of the Bushnell-Kutzko datum that we are constructing. 
\begin{prop}We have the following equalities of groups \begin{align*}
&H^1(\beta , \fA) = K^d_+ (\overrightarrow{G},y,\overrightarrow{R})\\
&J^0(\beta , \fA) = {}^{\circ} K^d (\overrightarrow{G},y,\overrightarrow{R})\\
&E^{\times}  J^0 (\beta , \fA) = K^d(\overrightarrow{G},y,\overrightarrow{R})
\end{align*}
\end{prop}
\begin{proof} The proof of Propositions \ref{hunkd} and \ref{rondkdj0} adapt. \end{proof}
 \begin{prop} \label{c2} The character $ \theta : = \displaystyle \prod _{i=0}^d \hat{\boldsymbol{\Phi}}_i$ of \cite[Proposition 4.4]{YU} is a simple character associated to the simple stratum $[\fA,n,0,\beta]$.
 \end{prop}
 \begin{proof} The character $ \hat{\boldsymbol{\Phi}}_d$ is equal to $1$, so $\theta = \displaystyle \prod _{i=0}^s \hat{\boldsymbol{\Phi}}_i$. By reversing  Proposition \ref{chapeau}, the character $\hat{\boldsymbol{\Phi}}_i$ is equal to the character $\theta ^i$ where $\theta ^i$ is the character defined by 
 \begin{align*}
 &\bullet \theta ^i \mid _{H^1 (\beta , \fA) \cap B_i} = \phi _i \circ {\det} _{B_i}  \\
 &\bullet \theta ^i \mid _{H^{1}(\beta , \fA ) \cap U^{[\frac{-\nu _{\fA} (c_i)}{2}+1]} (\fA)} = \psi _{c_i}.
 \end{align*}
 Now by reversing the arguments of Theorem \ref{facttam}, we obtain that the character $\theta$ is a simple character associated to the simple stratum $[\fA, n , 0 , \beta ]$.
 \end{proof}
\begin{prop} The representation $\kappa :={}^{\circ} \lambda  (\overrightarrow{G},y,\overrightarrow{R},\overrightarrow{\boldsymbol{\Phi} })$ is a $\beta$-extension of $\theta$.
\end{prop}
\begin{proof} The Proposition is proved as Proposition \ref{circlambabetaext}, using Proposition \ref{beta}.
\end{proof}
 \begin{prop} There exists an irreducible cuspidal representation $\sigma$ of $U^0 (\fB _0) / U^1 (\fB _0 ) $ and an extension $\Lambda$ of $\kappa \otimes \sigma $ to $E^{\times}J^0 (\beta , \fA ) $ such that $\Lambda = \rho _d (\overrightarrow{G},y,\overrightarrow{R},\rho,\overrightarrow{\boldsymbol{\Phi} })$.
 \end{prop}
 \begin{proof}  It is enough to take $\sigma = \rho \mid _{G^0(F) _{y,0}= \fB _0 ^{\times}}$ and $\Lambda = \rho _d (\overrightarrow{G},y,\overrightarrow{R},\rho,\overrightarrow{\boldsymbol{\Phi} })$
 \end{proof}
  Let us sum up every thing at this point.
  \begin{theo} \label{theo2} Let $(\overrightarrow{G},y,\overrightarrow{R},\rho,\overrightarrow{\boldsymbol{\Phi} })$ be a Yu datum in $G= \underline{\Aut} _F (V)$ such that $d>0$ and $\boldsymbol{\Phi}_d=1$. Let $E_0 , \ldots , E_d$ be the fields attached to $\overrightarrow{G}$ as before. Fix $c_i $ a minimal element relatively to $E_i /E_{i+1} $ as in Proposition \ref{eximi} for $0\leq i \leq s$. Put $\beta = \displaystyle \sum _{k=0}^s c_i $ and $n=-\nu _{\fA}(c_s)(=-\nu _{\fA}(\beta))$. Then $[\fA , n , 0 , \beta ] $ is a simple stratum and the character\footnote{of \cite[Proposition 4.4]{YU}} $\theta := \displaystyle \prod _{i=0} ^d \displaystyle  \hat{\boldsymbol{\Phi}}_i $ is a simple character attached to the stratum $[\fA , n , 0 , \beta ] $. The representation\footnote{see end of Section \ref{yu}} $\kappa :={}^{\circ} \lambda  (\overrightarrow{G},y,\overrightarrow{R},\overrightarrow{\boldsymbol{\Phi} })$ is a $\beta$-extension of $\theta$. There exists an irreducible cuspidal representation $\sigma$ of $U^0 (\fB _0) / U^1 (\fB _0 ) $ and an extension $\Lambda$ of $\kappa \otimes \sigma $ to $E^{\times}J^0 (\beta , \fA ) $ such that $\Lambda = \rho _d (\overrightarrow{G},y,\overrightarrow{R},\rho,\overrightarrow{\boldsymbol{\Phi} })$.
 \end{theo}
 \begin{proof} This is the content of this section.
 \end{proof}

 \begin{rema}\label{r2} As we have explained at the beginning, we have reduced to the case $d>0$ , $\boldsymbol{\Phi} _d =1$. Let us explain briefly the final statement for other cases. In the case $d \geq 0 $, $\boldsymbol{\Phi} _d \neq 1$, we can proceed similarly as before. We put $s=d$ and we factorise each $\boldsymbol{\Phi } _i $ and introduce $c_i$ as before for $0 \leq i \leq d-1$. Moreover, we factorise $\boldsymbol{\Phi} _d = \phi _d  \circ \det _{F} $ and we introduce an element $c_d \in F$ using Lemma \ref{intoproof}. Now we put $\beta = \displaystyle \sum _{k=0}^d c_i $ and the rest of the statement is valid. In the case $d=0 , \boldsymbol{\Phi} _0 = 1 $, we easily attach a type $(b)$ Bushnell-Kutzko datum.

 \end{rema}
 \newpage
 \section{Consequences and remarks}

 Let us state a natural corollary of our comparison.

 \begin{coro} \label{cororor}Let $V$ be a finite dimensional $F$-vector space. Let $\mathrm{G}$ be $\underline{\Aut}_F(V)~/F$ and $G= \mathrm{G}( F)=\mathrm{Aut}_F (V)$. \begin{enumerate} \item Let \[\text{TameType}_{\text{BK }} := \left\{ (E^{\times } J^0 , \Lambda) ~\Bigg| \begin{matrix} (E^{\times } J^0 , \Lambda) \text{ is a tame Bushnell-}\\ \text{ Kutzko maximal extended simple } \\ \text{type in } G \text{ as in Definition \ref{tamebkdatum}} \end{matrix} \right\}\] and

 \[\text{GType}_{\text{YU }} := \left\{ (K^d , \rho _d) ~\Bigg| \begin{matrix} (K^d ,\rho_d) \text{ is a tame Yu  }\\ \text{ extended type for } \mathrm{G}\\ \text{as in Definition \ref{yuextended}} \end{matrix} \right\}.\] 
  The following equality holds  \[ \text{TameType}_{\text{BK }} =\text{GType}_{\text{YU}} ,\]  here  $(E^{\times } J^0 ,\Lambda) = (K^d , \rho _d )$ means $E^{\times} J^0 = K^d$ and $\Lambda \simeq \rho_d$.

  \item Let \[\text{TameCar}_{\text{BK }} := \left\{  \theta ~\Bigg| \begin{matrix}  \theta \text{ is a tame Bushnell-Kutzko}\\ \text{ maximal simple character in } G\\ \text{as in Definition \ref{tamebkdatum}} \end{matrix} \right\}\] and

 \[\text{GCar}_{\text{YU }} := \left\{  \theta _d ~\Bigg| \begin{matrix} \theta_d \text{ is a character  } \\ \text{ constructed by Yu for } \mathrm{ G} \\ \text{as in Construction \ref{yuconstrutru}} \end{matrix} \right\}.\] 
  The following equality holds \[ \text{TameCar}_{\text{BK }} =\text{GCar}_{\text{YU}} .\]

  \end{enumerate}
 \end{coro}

 Of course, our comparison gives many  more than equalities of Corollary \ref{cororor}: we have explicitly compared all steps of both constructions.
 .
 \begin{proof} \begin{enumerate} \item This is a direct consequence of Theorem \ref{bigtheofin} , Theorem \ref{theo2} and Remark \ref{r2}.

 \item This is a direct consequence of Proposition \ref{chapeau} , Proposition \ref{c2} and Remark \ref{r2}

 \end{enumerate}

 \end{proof}

 \begin{rema}(Terminology) We saw in this article that the irreducible supercuspidal representations of $GL_N(F)$ associated with the tame simple strata can also be obtained through Yu's construction. Reciprocally we have seen that Yu's $GL_N$ representations can be obtained from tame simple strata through Bushnell-Kutzko's construction. Yu called the representations that  he constructed "tame representations." Bushnell-Henniart called the representation associated to tame simple strata "essentially tame representations" (note that we have avoided the word "essentially" in our terminology for related objects (types, data, characters) in this article), and they show that these are the representations $\pi$ such that $p$ does not divide $\frac{N}{t(\pi)}$ where $t(\pi)$ is the torsion number of $\pi$ (see \cite[Lemma page 495]{Esse}). So the representations considered in this article can be called "essentially tame representations" or "tame representations."

\end{rema}

\begin{rema} Decreasing defining sequences on tame $BK$-side reflects axioms of inclusions in Yu's twisted Levi sequence. This phenomenon is correlated to the fact that tame constructions always factorise(e.g.,  Howe's factorisation of admissible characters \cite{HO} \cite{MO}, factorisation  of tame simple characters (Section \ref{tameca} of this article), Yu's construction as tensor product, Hakim-Murnaghan's refactorisation  \cite{Hamu}, Kaletha's Howe's factorisation  \cite{Kaletha})
\end{rema}

\begin{rema} As we have previously explained at the end of \cite{ArnMay}, the comparison $BK \longleftrightarrow YU$ could be extended to other works and constructions. Works \cite{BKa} \cite{Sech}, \cite{Stev}, \cite{blbl}, \cite{SKO}, \cite{vdng} deal with constructions of supercuspidal representations in the spirit and language of $BK$ and have the same global structure. So our comparison should theoretically be adapted following a similar global organisation. Similar comparisons should also theoretically be written for \cite{bksemi}, \cite{goro}, \cite{SteMi}, on one hand, and 
\cite{ykge},  on the other hand.

\end{rema}









\begin{rema}\textbf{(General construction)} \label{UNIF} \cite{MOber}
\begin{enumerate} \item Let $[\fA , n , 0 , \beta]$ be a simple stratum. If we remove the assumption that $F[\beta]/F$ is tame, then the sequence of fields $E_0 , \ldots , E_s$ attached to a defining sequence is not decreasing for $\subset$ in general. It always decreases for $[\bullet : F]$.
\item In a certain sense, we have explained that Bushnell-Kutzko and Yu's constructions are compatible (they are compatible on their common domain of definition). Is there a unified construction $\circledS$ that generalise both of them ?
\item If one tries to obtain $\circledS$ generalising Yu's approach, then one has to remove (among other things) the axiom of inclusions in the twisted Levi sequence by (1) of this remark and by definition of $\overrightarrow{G}$. This implies that one cannot expect a factorable construction \guillemotleft $~\rho _d = \otimes \kappa ^i~$\guillemotright, as in  Yu's work.
In others words, if one tries to obtain $\circledS$ by  combining \cite{BK} and Fintzen's work \cite{fintzen} on Yu's construction, then one has to expect to work with a non-increasing (for $\subset$) sequence of twisted Levi $G^0  , \ldots , G^d $.

\end{enumerate}
\end{rema}

\bibliographystyle{plain}
\bibliography{library}

\begin{thebibliography}{10}

\bibitem{Adle}
Jeffrey~D. Adler.
\newblock Refined anisotropic {$K$}-types and supercuspidal representations.
\newblock {\em Pacific J. Math.}, 185(1):1--32, 1998.

\bibitem{SGA31}
M.~Artin, J.~E. Bertin, M.~Demazure, P.~Gabriel, A.~Grothendieck, M.~Raynaud,
  and J.-P. Serre.
\newblock {\em Sch\'{e}mas en groupes. {F}asc. 1: {E}xpos\'{e}s 1 \`a 4},
  volume 1963 of {\em S\'{e}minaire de G\'{e}om\'{e}trie Alg\'{e}brique de
  l'Institut des Hautes \'{E}tudes Scientifiques}.
\newblock Institut des Hautes \'{E}tudes Scientifiques, Paris, 1963/1964.

\bibitem{blbl}
Laure Blasco and Corinne Blondel.
\newblock Caract\`eres semi-simples de {$G_2(F)$}, {$F$} corps local non
  archim\'edien.
\newblock {\em Ann. Sci. \'Ec. Norm. Sup\'er. (4)}, 45(6):985--1025 (2013),
  2012.

\bibitem{Brou}
Paul Broussous and Bertrand Lemaire.
\newblock Building of {${\rm GL}(m,D)$} and centralizers.
\newblock {\em Transform. Groups}, 7(1):15--50, 2002.

\bibitem{brti1}
François Bruhat and Jacques Tits.
\newblock Groupes r\'{e}ductifs sur un corps local.
\newblock {\em Inst. Hautes \'{E}tudes Sci. Publ. Math.}, (41):5--251, 1972.

\bibitem{brti2}
François Bruhat and Jacques Tits.
\newblock Groupes r\'{e}ductifs sur un corps local. {II}. {S}ch\'{e}mas en
  groupes. {E}xistence d'une donn\'{e}e radicielle valu\'{e}e.
\newblock {\em Inst. Hautes \'{E}tudes Sci. Publ. Math.}, (60):197--376, 1984.

\bibitem{bushandfro}
C.~J. Bushnell and A.~Fr\"{o}hlich.
\newblock Nonabelian congruence {G}auss sums and {$p$}-adic simple algebras.
\newblock {\em Proc. London Math. Soc. (3)}, 50(2):207--264, 1985.

\bibitem{BHTL1}
Colin~J. Bushnell and Guy Henniart.
\newblock Local tame lifting for {${\rm GL}(N)$}. {I}. {S}imple characters.
\newblock {\em Inst. Hautes \'Etudes Sci. Publ. Math.}, (83):105--233, 1996.

\bibitem{bhtl4}
Colin~J. Bushnell and Guy Henniart.
\newblock Local tame lifting for {${\rm GL}(n)$}. {IV}. {S}imple characters and
  base change.
\newblock {\em Proc. London Math. Soc. (3)}, 87(2):337--362, 2003.

\bibitem{bhess1}
Colin~J. Bushnell and Guy Henniart.
\newblock The essentially tame local {L}anglands correspondence. {I}.
\newblock {\em J. Amer. Math. Soc.}, 18(3):685--710, 2005.

\bibitem{bhess2}
Colin~J. Bushnell and Guy Henniart.
\newblock The essentially tame local {L}anglands correspondence. {II}.
  {T}otally ramified representations.
\newblock {\em Compos. Math.}, 141(4):979--1011, 2005.

\bibitem{bhtl3}
Colin~J. Bushnell and Guy Henniart.
\newblock Local tame lifting for {${\rm GL}(n)$}. {III}. {E}xplicit base change
  and {J}acquet-{L}anglands correspondence.
\newblock {\em J. Reine Angew. Math.}, 580:39--100, 2005.

\bibitem{bhess3}
Colin~J. Bushnell and Guy Henniart.
\newblock The essentially tame local {L}anglands correspondence, {III}: the
  general case.
\newblock {\em Proc. Lond. Math. Soc. (3)}, 101(2):497--553, 2010.

\bibitem{Esse}
Colin~J. Bushnell and Guy Henniart.
\newblock The essentially tame {J}acquet-{L}anglands correspondence for inner
  forms of {${\rm GL}(n)$}.
\newblock {\em Pure Appl. Math. Q.}, 7(3, Special Issue: In honor of Jacques
  Tits):469--538, 2011.

\bibitem{BK}
Colin~J. Bushnell and Philip~C. Kutzko.
\newblock {\em The admissible dual of {${\rm GL}(N)$} via compact open
  subgroups}, volume 129 of {\em Annals of Mathematics Studies}.
\newblock Princeton University Press, Princeton, NJ, 1993.

\bibitem{BKa}
Colin~J. Bushnell and Philip~C. Kutzko.
\newblock The admissible dual of {${\rm SL}(N)$}. {II}.
\newblock {\em Proc. London Math. Soc. (3)}, 68(2):317--379, 1994.

\bibitem{bksemi}
Colin~J. Bushnell and Philip~C. Kutzko.
\newblock Semisimple types in {${\rm GL}_n$}.
\newblock {\em Compositio Math.}, 119(1):53--97, 1999.

\bibitem{Cara}
Henri Carayol.
\newblock Représentations cuspidales du groupe linéaire.
\newblock {\em Ann. Sci. \'Ecole Norm. Sup. (4)}, 17(2):191--225, 1984.

\bibitem{fintzen}
Jessica Fintzen.
\newblock Types for tame p-adic groups, https://arxiv.org/abs/1810.04198.
\newblock 2018.

\bibitem{goro}
David Goldberg and Alan Roche.
\newblock Types in {${\rm SL}_n$}.
\newblock {\em Proc. London Math. Soc. (3)}, 85(1):119--138, 2002.

\bibitem{Hamu}
Jeffrey Hakim and Fiona Murnaghan.
\newblock Distinguished tame supercuspidal representations.
\newblock {\em Int. Math. Res. Pap. IMRP}, (2):Art. ID rpn005, 166, 2008.

\bibitem{HO}
Roger~E. Howe.
\newblock Tamely ramified supercuspidal representations of {${\rm GL}_{n}$}.
\newblock {\em Pacific J. Math.}, 73(2):437--460, 1977.

\bibitem{Kaletha}
Tasho Kaletha.
\newblock Regular supercuspidal representations.
\newblock {\em J. Amer. Math. Soc.}, 32(4):1071--1170, 2019.

\bibitem{Kim}
Ju-Lee Kim.
\newblock Supercuspidal representations: an exhaustion theorem.
\newblock {\em J. Amer. Math. Soc.}, 20(2):273--320, 2007.

\bibitem{ykge}
Ju-Lee Kim and Jiu-Kang Yu.
\newblock Construction of tame types.
\newblock In {\em Representation theory, number theory, and invariant theory},
  volume 323 of {\em Progr. Math.}, pages 337--357. Birkh\"{a}user/Springer,
  Cham, 2017.

\bibitem{ArnMay}
Arnaud Mayeux.
\newblock Repr\'esentations supercuspidales: comparaison des constructions de
  {B}ushnell-{K}utzko et {Y}u, https://arxiv.org/abs/1706.05920.
\newblock 2017.

\bibitem{MOber}
Arnaud Mayeux.
\newblock Comparison of {B}ushnell-{K}utzko and {Y}u's construction of
  supercuspidal representations.
\newblock {\em Oberwolfach Reports, New Developments in Representation Theory
  of p-adic Groups, Organizers: Jessica Fintzen, Wee Teck Gan, Shuichiro
  Takeda}, 2019.

\bibitem{SteMi}
Michitaka Miyauchi and Shaun Stevens.
\newblock Semisimple types for {$p$}-adic classical groups.
\newblock {\em Math. Ann.}, 358(1-2):257--288, 2014.

\bibitem{MO}
Allen Moy.
\newblock Local constants and the tame {L}anglands correspondence.
\newblock {\em Amer. J. Math.}, 108(4):863--930, 1986.

\bibitem{Mopr}
Allen Moy and Gopal Prasad.
\newblock Unrefined minimal {$K$}-types for {$p$}-adic groups.
\newblock {\em Invent. Math.}, 116(1-3):393--408, 1994.

\bibitem{Mopr2}
Allen Moy and Gopal Prasad.
\newblock Jacquet functors and unrefined minimal {$K$}-types.
\newblock {\em Comment. Math. Helv.}, 71(1):98--121, 1996.

\bibitem{Neuk}
J\"urgen Neukirch.
\newblock {\em Algebraic number theory}, volume 322 of {\em Grundlehren der
  Mathematischen Wissenschaften [Fundamental Principles of Mathematical
  Sciences]}.
\newblock Springer-Verlag, Berlin, 1999.
\newblock Translated from the 1992 German original and with a note by Norbert
  Schappacher, With a foreword by G. Harder.

\bibitem{vdng}
Van~Dinh Ngô.
\newblock Construction de représentations supercuspidales des groupes
  spinoriels définis sur des corps $p$-adiques au moyen de types semi-simples.
\newblock 2013.

\bibitem{Rena}
David Renard.
\newblock {\em Repr\'esentations des groupes r\'eductifs {$p$}-adiques},
  volume~17 of {\em Cours Sp\'ecialis\'es [Specialized Courses]}.
\newblock Soci\'et\'e Math\'ematique de France, Paris, 2010.

\bibitem{Sech}
Vincent Sécherre and Shaun Stevens.
\newblock Repr\'esentations lisses de {${\rm GL}_m(D)$}. {IV}.
  {R}epr\'esentations supercuspidales.
\newblock {\em J. Inst. Math. Jussieu}, 7(3):527--574, 2008.

\bibitem{SKO}
Daniel Skodlerack.
\newblock Cuspidal irreducible representations of quaternionic forms of p-adic
  classical groups for odd p.
\newblock {\em https://arxiv.org/abs/1907.02922v1}, 2019.

\bibitem{Steinbergtor}
Robert Steinberg.
\newblock Torsion in reductive groups.
\newblock {\em Advances in Math.}, 15:63--92, 1975.

\bibitem{Stev}
Shaun Stevens.
\newblock The supercuspidal representations of {$p$}-adic classical groups.
\newblock {\em Invent. Math.}, 172(2):289--352, 2008.

\bibitem{YU}
Jiu-Kang Yu.
\newblock Construction of tame supercuspidal representations.
\newblock {\em J. Amer. Math. Soc.}, 14(3):579--622 (electronic), 2001.

\end{thebibliography}

\end{document}